\documentclass{article}
\usepackage{graphicx}
\usepackage{mathtools}
\usepackage{tikz}
\usepackage{marginnote}
\usepackage{verbatim}
\usepackage{nicefrac}
\usepackage{bbm}

\usepackage[enumitem,hyperref,theorem,comment, Rn, amsmath]{paper_diening}
\usepackage[paper=a4paper,left=30mm,right=30mm,top=37mm,bottom=37mm]{geometry}
\usepackage{mathtools}
\usepackage{tikz}
\usepackage{nicefrac}

\RequirePackage{amsmath}
\RequirePackage{amssymb}
\RequirePackage{amsfonts}
\RequirePackage{xspace}

\providecommand{\Pih}{{\Pi_h}}

\providecommand{\supp}{{\mathrm{supp}\,}}
\DeclareMathOperator*{\osc}{osc}
\providecommand{\dx}{\,\mathrm{d}x}
\providecommand{\dy}{\,\mathrm{d}y}

  \providecommand{\divergence}{\nabla\cdot}
  \providecommand{\support}{{\operatorname{supp}}}

  \providecommand{\loc}{{\ensuremath{\mathrm{loc}}}}

  \providecommand{\setR}{\ensuremath{\mathbb{R}}}

  \providecommand{\normtmp}[2]{{#1\lVert{#2}#1\rVert}}
  \providecommand{\norm}[1]{\normtmp{}{#1}}

  \providecommand{\normmtmp}[2]{{#1\lvert\hspace{-0.07em}#1\lvert\hspace{-0.07em}#1\lvert{#2}
    #1\rvert\hspace{-0.07em}#1\rvert\hspace{-0.07em}#1\rvert}}
  \providecommand{\normm}[1]{\normmtmp{}{#1}}

  \providecommand{\abstmp}[2]{{#1\lvert{#2}#1\rvert}}
  \providecommand{\abs}[1]{\abstmp{}{#1}}
  \providecommand{\bigabs}[1]{\abstmp{\big}{#1}}
  \providecommand{\Bigabs}[1]{\abstmp{\Big}{#1}}

  \providecommand{\skptmp}[3]{{\ensuremath{#1\langle {#2}, {#3} #1\rangle}}}
  \providecommand{\skp}[2]{\skptmp{}{#1}{#2}}

  \providecommand{\hskptmp}[3]{{\ensuremath{#1( {#2}, {#3} #1)}}}

  \providecommand{\settmp}[2]{{#1\{{#2}#1\}}}
  \providecommand{\set}[1]{\settmp{}{#1}}

  \providecommand{\embedding}{\ensuremath{\hookrightarrow}}

  \providecommand{\meantmp}[2]{#1\langle{#2}#1\rangle}
  \providecommand{\mean}[1]{\meantmp{}{#1}}

  \providecommand{\Poincare}{{Poincar{\'e}}\xspace}
  
  \providecommand{\Ruzicka}{{R{\r u}{\v z}i{\v c}ka\xspace}}

  \providecommand{\Xint}[1]{\mathchoice
    {\XXint\displaystyle\textstyle{#1}}%
    {\XXint\textstyle\scriptstyle{#1}}%
    {\XXint\scriptstyle\scriptscriptstyle{#1}}%
    {\XXint\scriptscriptstyle\scriptscriptstyle{#1}}%
    \!\int}
  \providecommand{\XXint}[3]{{\setbox0=\hbox{$#1{#2#3}{\int}$}
      \vcenter{\hbox{$#2#3$}}\kern-.5\wd0}}
  
  \providecommand{\dashint}{\mathop{\Xint-}}

\begin{document}

\title{Uniform H\"older-norm bounds for finite element approximations of second-order elliptic equations}

\author{%
{\sc Lars Diening}\thanks{Email: lars.diening@uni-bielefeld.de},\\
University Bielefeld,\\
Universit\"atsstrasse 25, 33615 Bielefeld, Germany\\[6pt]
{\sc and}\\[6pt]
{\sc Toni Scharle}\thanks{Corresponding author. Email: toni.scharle@queens.ox.ac.uk}
{\sc and}
{\sc Endre S\"uli}\thanks{Email: endre.suli@maths.ox.ac.uk}\\[2pt]
Mathematical Institute, University of Oxford,\\ Woodstock Road, Oxford OX2 6GG, UK
}

\maketitle

\begin{abstract}
{We develop a discrete counterpart of the De Giorgi--Nash--Moser theory, which provides uniform H\"older-norm bounds on continuous piecewise affine
finite element approximations of second-order linear elliptic problems of the form $-\divergence(A\nabla u)=f-\divergence  F$ with $A\in L^\infty(\Omega;\setR^{n\times n})$ a uniformly elliptic matrix-valued function,
$f\in L^{q}(\Omega)$, $F\in L^p(\Omega;\setR^n)$, with $p > n$ and $q > n/2$, on $A$-nonobtuse shape-regular triangulations, which are not required to be quasi-uniform, of a bounded polyhedral Lipschitz domain $\Omega \subset \setR^n$.}
{elliptic differential equations; finite element methods; discrete $C^\alpha$ regularity.}
\end{abstract}

\begin{center}
\textbf{Dedicated to the memory of John W. Barrett}
\end{center}

\section{Introduction}
\label{sec;introduction}

    Given a bounded domain $\Omega\subset\setR^n$, a uniformly elliptic matrix-valued function $A\in L^\infty(\Omega;\setR^{n\times n})$, a function $f\in L^q(\Omega)$ with $q>{n}/{2}$ and a vector-valued function $F \in L^p(\Omega;\setR^n)$ with $p>n$, it is well known from the work of De Giorgi, Nash, and Moser that weak solutions $u\in W^{1,2}_0(\Omega)$ to the elliptic boundary-value problem
    \begin{alignat}{2}\label{eq:aPoisson}
        \begin{aligned}
          -\divergence  (A\nabla u)&=f - \divergence  F &&\quad \text{ on }\Omega,\\
          u& = 0 &&\quad \text{ on }\partial\Omega
        \end{aligned}
    \end{alignat}
    are in fact H\"older-continuous. The combination of the De Giorgi--Nash--Moser iteration technique based on level sets and a Caccioppoli-type inequality, which estimates the norm of $\abs{\nabla u}$ locally in terms of the norms of $u$ and $f$, is in fact flexible enough to be applied to a variety of nonlinear elliptic or parabolic problems. For further details we refer, for example, to \cite{DiB93} and \cite{DSS19}.

    While the finite element method is one of the most general and powerful techniques for the numerical approximation of solutions to partial differential equations, there is currently no discrete counterpart of the De-Giorgi--Nash--Moser theory for elliptic boundary-value problems of the form \eqref{eq:aPoisson} under the regularity hypotheses on the functions $A$, $f$ and $F$ stated above.  Our aim here is to fill this gap by identifying conditions under which De-Giorgi-type regularity results hold in the discrete setting, for continuous piecewise affine finite element approximations of the problem \eqref{eq:aPoisson}.

   We begin by surveying the related literature. The first step towards proving H\"older regularity is usually a local $L^\infty$-norm bound.
   The earliest result in this direction in the finite element literature appears in the work of \cite{Nitsche}, who proved an $\mathcal{O}(h)$ error bound in the $L^\infty(\Omega)$-norm
   for elliptic equations of the form $-\nabla \cdot (A \nabla u) + cu = f$ in two space dimensions with $A \in W^{1,\infty}(\Omega;\setR^{2 \times 2})$, $c \in L^\infty(\Omega)$, $f \in L^2(\Omega)$
   on convex domains $\Omega$, subject to a homogeneous Dirichlet boundary condition, and a continuous piecewise affine approximation on $\alpha-\kappa$-regular triangulations of granularity $h>0$ (i.e. the size of any angle in the triangulation is bounded below by $\alpha$,  and the ratio of the side-lengths of any two triangles in the triangulation is bounded above by $\kappa$). It was proved by \cite{Helfrich} that for $f \in L^2(\Omega)$ and a continuous piecewise affine finite element approximation this error bound cannot be improved.
   Subsequently, \cite{CR} extended Nitsche's $\mathcal{O}(h)$ error bound in the $L^\infty(\Omega)$-norm to $n$ space dimensions assuming that $f \in L^q(\Omega)$ and $q > n/2$, on regular simplicial subdivisions of \textit{nonnegative type}; in the case of Poisson's equation in two space dimensions a triangulation is guaranteed to be of nonnegative type if all the angles of the triangles of the triangulation
are $\leq \pi/2$.  The question arose therefore whether an error order $\mathcal{O}(h^2)$ with $f \in L^{\infty}(\Omega)$ and a continuous piecewise affine finite element approximation could perhaps be proved, as numerical experiments conducted at the time appeared to indicate that such an assertion might be true. On a polygonal domain in $\setR^2$ with a quasi-uniform triangulation, $f\in L^\infty(\Omega)$ and a finite element space consisting of piecewise affine functions, \cite{Nat75} was able to prove an $\mathcal{O}(h^{2-\epsilon})$ bound on the approximation error in the $L^\infty(\Omega)$-norm, for Poisson's equation $-\Delta u = f$ subject to a homogeneous Dirichlet boundary condition using a technique based on weighted Sobolev spaces. Soon thereafter \cite{SchWah77} proved a local $L^\infty$-best-approximation property. They analyzed  approximate solutions to the elliptic equation
    \begin{equation*}
        -\divergence  (A(x)\nabla u)+b(x)\cdot\nabla u+d(x)u=f
    \end{equation*}
    for smooth uniformly positive definite matrix-valued functions $A$, vector-valued functions $b$, and scalar-valued functions $c$ on a quasi-uniform and shape-regular triangulation with a few additional technical assumptions on the finite element space that are, for example, satisfied by Lagrange or Hermite elements. Their final result is the following estimate between the analytical solution $u$ and its finite element approximation $u_h$ on a triangulation of $\Omega$ of granularity $h \in (0,1)$:
    \begin{equation*}
        \norm{u-u_h}_{L^\infty(\Omega_1)}\leq C\left(\left(\log \frac{1}{h}\right)^{\overline{r}}\norm{u-\chi}_{L^\infty(\Omega)}+\norm{u-u_h}_{W^{-s,q}(\Omega)}\right).
    \end{equation*}
    Here $W^{-s,q}(\Omega)$ is the dual space of the Sobolev space $W^{s,q'}_0(\Omega)$ with $s>0$, $1<q<\infty$, and $\frac {1}{q}+\frac{1}{q'}=1$; the exponent $\overline{r}$ is $0$ if the optimal approximation order in terms of $h$ with which finite element functions can approximate functions in the $L^q(\Omega)$ norm is $3$ or higher, and is $1$ if this order is $2$ (as is the case for a continuous piecewise affine finite element approximation); $\chi$ is an arbitrary finite element function; $\Omega_1 \Subset \Omega$ and $C$ is a positive constant independent of $h$ and $\chi$. The proof of this result is based on pointwise estimates using the discrete Green's function. This also implies that its generalization to nonlinear equations, using the same technique, is impossible.

    Subsequently \cite{Hav84} and \cite{BLR} proved that with a continuous piecewise affine finite element
approximation the error $\|u-u_h\|_{L^\infty(\Omega)}$ is \textit{not} of order $O(h^2)$ in general, even if the data are smooth; for higher degree piecewise polynomial spaces an optimal order error bound in the $L^\infty(\Omega)$ norm, without the additional logarithmic factor, was already known to hold (cf. \cite{Scott}).

A uniform H\"older-norm bound on the sequence of finite element approximations can also be obtained from a uniform $W^{1,p}(\Omega)$-norm bound for $p>n$ thanks to the continuous embedding of $W^{1,p}(\Omega)$ into $C^\alpha(\overline\Omega)$ with $\alpha = 1 - n/p$ guaranteed by Morrey's embedding theorem.
\cite{RS} showed that the Ritz projection onto spaces of continuous piecewise affine finite elements in two space dimensions is bounded in the Sobolev space $W^{1,p}(\Omega)$ for $2 \leq  p \leq \infty$ on triangulations which satisfy the condition that each triangle in the triangulation contains a circle of radius $c_1 h$ and is contained in a circle of radius $c_2h$,
with positive constants $c_1$ and $c_2$; they used this to prove that for functions in
$W^{1,p}_0(\Omega) \cap W^{2,p}(\Omega)$ the approximation error behaves like $O(h)$ in the norm of $W^{1,p}(\Omega)$ for
$2 \leq  p  \leq \infty$, and like $O(h^2)$ in the norm of $L^p(\Omega)$ for $2 \leq p < \infty$. In all these cases the additional logarithmic factor, which appeared in previously published error bounds for continuous piecewise affine finite elements, was shown not occur.

More generally, and closer to the weak regularity assumptions on the data considered herein, for elliptic problems of the form $-\nabla \cdot (A \nabla u) = f - \nabla \cdot F$,  with $A \in L^\infty(\Omega; \setR^{n\times n})$ a uniformly elliptic matrix-valued function, $f \in L^{\frac{pn}{p+n}}(\Omega)$ and $F \in L^p(\Omega; \setR^{n\times n})$, $p \in (1,\infty)$, subject to a homogeneous Dirichlet boundary condition, on a bounded open convex polytopal domain $\Omega \subset \mathbb{R}^n$, $n \in \{2,3\}$,  as a direct consequence of Proposition 8.6.2 in \cite{BS2008} and Theorem 5.1 of the work of Grisvard \cite{Grisvard}, one obtains a uniform $W^{1,p}(\Omega)$ norm bound on a sequence of finite element approximations on quasi-uniform triangulations, for all $p \in (2, 2+\varepsilon)$ for some, possibly small $\varepsilon>0$. This can be seen as a discrete counterpart of a Meyers-type regularity estimate. Hence, in two space dimensions at least ($n=2$), thanks to Morrey's embedding theorem a uniform H\"older-norm bound on the sequence of finite element approximations on quasi-uniform triangulations directly
follows (cf. Theorem 3.20 in \cite{KPS17}). In dimensions $n \geq 3$, however, such an indirect argument for deriving a uniform H\"older-norm bound does not work. Our aim here is therefore to develop a discrete De Giorgi theory that will directly yield such uniform H\"older-norm bounds, without assuming quasi-uniformity of the triangulation. For continuous piecewise affine finite element approximations of
Laplace's equation $\Delta u=0$ at least, \cite{AguCaf86} proved via a De-Giorgi-type argument an $h$-uniform $C^\alpha$-bound, assuming quasi-uniformity, shape regularity, and uniform acuteness of the triangulation. Their paper asserts, without providing a proof, that the result generalizes to more complicated equations.

Yet another approach is to deduce a uniform H\"older-norm bound from a uniform $W^{1,\infty}(\Omega)$-norm bound. \cite{GLRS09} considered the finite element approximation of the Poisson problem
      \begin{equation}\label{eq:Poisson}
        \begin{aligned}
            -\Delta u &= f \quad \text{ in } \Omega,\\
            u&=0 \quad \text{ on }\partial\Omega
        \end{aligned}
    \end{equation}
in three space dimensions. They established a best approximation result in the $W^{1,\infty}(\Omega)$ norm for convex polyhedral domains. Previous results, based on weighted $L^2(\Omega)$ norm bounds required $W^{2,p}(\Omega)$ regularity for $p>3$, which imposed an upper bound on the dihedral angles of the domain. The approach of \cite{GLRS09} proceeds by establishing
sharp H\"{o}lder-norm bounds on the first partial derivatives and the second mixed derivatives of the Green's function, requiring only $W^{1,\infty}(\Omega)$ regularity, thus avoiding the maximum angle condition.
A quasi-uniform family of triangulations is assumed in conjunction with a conforming finite element space $S_h$ consisting of continuous piecewise polynomials of degree $k\geq 1$ and a smooth right-hand side $f$, in order to prove that
    \begin{equation}\label{eq:bestApprox}
        \norm{\nabla(u-u_h)}_{L^\infty(\Omega)}\leq C \inf_{\chi\in S_h}\norm{\nabla (u-\chi)}_{L^\infty(\Omega)},
    \end{equation}
    where $u_h\in S_h$ is the finite element approximation of the analytical solution $u$ of the problem \eqref{eq:Poisson}.
    By taking $\chi\equiv 0$ followed by the application of the triangle inequality this implies that
    \begin{equation*}
        \norm{\nabla u_h}_{L^{\infty}(\Omega)}\leq C \norm{\nabla u}_{L^\infty(\Omega)}.
    \end{equation*}
Together with standard regularity theory for solutions to equation \eqref{eq:Poisson} with a smooth right-hand side $f$ and the embedding $W^{1,\infty}(\Omega) \hookrightarrow C^\alpha(\overline{\Omega})$ this inequality implies a uniform H\"older-norm bound on $u_h$, on quasi-uniform families of triangulations at least.

     \cite{Dol99} proved \textit{a priori} bounds and optimal error estimates on quasi-uniform families of triangulations
     for continuous piecewise affine finite element approximations of elliptic systems in divergence form with continuous coefficients contained in Campanato spaces
and extended the bounds of  \cite{RS} to elliptic systems.

    All of the results we have cited so far exclude highly graded adaptively refined triangulations. The most common condition that is required to hold within a local adaptive refinement process (see, for example, \cite{ste08}) is shape-regularity (see Definition \ref{def:shapeRegular}). This raises the question whether any of the above results can be derived assuming shape-regularity of the family of triangulations only. We note in this respect that \cite{DLSW12} have obtained a best approximation property as in inequality \eqref{eq:bestApprox} for solutions to the Poisson problem \eqref{eq:Poisson} for $f\in L^\infty(\Omega)$ on slightly graded triangulations, i.e. the local mesh-size is only varying slowly, in two and three space dimensions. The approach is again based on Galerkin orthogonality and pointwise bounds on the discrete Green's function.

    Let us now briefly discuss our results. Theorem \ref{thm:HoelderBoundary} provides a uniform \textit{a priori} H\"older-norm bound on sequences
    of continuous piecewise affine finite element approximations of the scalar elliptic problem $-\divergence(A\nabla u)=f-\divergence  F$ subject to a
    homogeneous Dirichlet boundary condition. We shall require shape-regularity only, thus admitting highly graded finite element
    triangulations. Our bounds require that $A\in L^\infty(\Omega;\setR^{n\times n})$ only; in particular, we
    shall not demand H\"older-continuity or Campanato-regularity of $A$. Our results are therefore a useful first step toward the development of similar bounds for more
    complex, nonlinear elliptic problems, such as $p$-Laplace-type equations that arise in mathematical models of non-Newtonian fluids (\cite{BL4, BL3, BL2} and \cite{BL1}). As a matter of fact, it is the finite element approximation of the models of non-Newtonian fluids considered in \cite{KPS17} and \cite{KoSu19} that motivated the work reported in this paper. Little seems to be known about uniform $L^\infty$ and $C^\alpha$ bounds on shape-regular families of triangulations for such nonlinear elliptic equations. For approximate solutions of the $p$-Laplace equation $-\divergence(\abs{\nabla u}^{p-2}\nabla u)=0$ a discrete maximum principle has been proved by \cite{DieKreSch13}.

     The discrete De-Giorgi-type iteration developed here is flexible enough to be applicable to continuous piecewise affine approximations of uniformly elliptic nonlinearities; see Theorem \ref{thm:discreteNonLinear}.
     On the other hand, we have to assume the existence of a function $G \in L^p(\Omega;\setR^n)$ such that $\nabla \cdot (G \pm F) \leq 0$ in the
     sense of distributions (cf. Definition \ref{def:AssumptionStar}); this assumption is clearly satisfied if $F$ is a constant vector-field,
     but it also holds in a number of nontrivial cases.
     We also require the triangulation to be $A$-nonobtuse. Even when $A$ is the identity matrix, in which case $A$-nonobtuseness becomes the
     standard requirement of nonobtuseness, this condition is restrictive. In general, common algorithms (for example the one proposed in \cite{ste08})
     for local mesh-refinement produce obtuse angles. However, our analysis avoids imposing the assumption of uniform acuteness,
     which was required by \cite{AguCaf86}; this allows us to admit important special cases such as the $n$-dimensional hypercube with a Kuhn-simplex triangulation that gets locally refined.
     Nor do we require quasi-uniformity of the family of triangulations. To the best of our knowledge, the proof of the uniform H\"older-norm bound on sequences of continuous piecewise affine finite element approximations established herein is the first that admits highly graded triangulations.

\section{$C^\alpha$-regularity of approximate solutions to linear elliptic equations}\label{ch:Linear}
In this section we will establish \textit{a priori} regularity results for approximate solutions to linear elliptic partial differential equations. We begin with a brief overview of the `continuous' De Giorgi theory for equations of this type and introduce different meshes and mesh-conditions. Subsequent sections will focus on developing a discrete De Giorgi theory.

\subsection{Notational conventions}
Before we start we have to introduce certain notational conventions. For two nonnegative expressions $a$ and $b$ we will write $a\lesssim b$ if there is a positive constant $C$ such that $a\leq Cb$. We will write $a\gtrsim b$ if there is a constant $c>0$ such that $a\geq cb$. If we have $a\lesssim b$ and $a\gtrsim b$, we will write $a\sim b$. The maximum of two real numbers $a$ and $b$ will be denoted either by $\max\{a,b\}$ or in short by $a\vee b$. The positive part of an expression $u$ will be denoted by $u_+:=u\vee 0$.

When working with H\"older spaces, an important quantity is the oscillation of a bounded function $v:\Omega\rightarrow \setR$ on a domain $\Omega$.
For any open or closed set $A\subset\Omega$ we define
$$\osc_A v=\sup_A v-\inf_A v.$$
We will denote the Lebesgue-measure of a measurable set $A\subset\setR^n$ by $\abs{A}$ and its characteristic function by $\chi_A$.

\subsection{Auxiliary results}
\label{app1}

We summarize here some elementary inequalities which will play a crucial role in the arguments that will follow. The first of these is easily proved by induction.

\begin{lemma}[Fast geometric convergence, cf. \cite{DiB93} Ch. I, Lemma 4.1]\label{converg}
  Let $\alpha>0$, $C>0$ and  $b>1$ be real numbers and $(a_k)$ a sequence of nonnegative real numbers with the properties
  $$ 0\leq a_{k+1}\leq C b^k a_k^{1+\alpha}, \qquad 0\leq a_0\leq C^{-\frac{1}{\alpha}}b^{-\frac{1}{\alpha^2}}. $$
  Then we have that $a_k\leq C^{-\frac{1}{\alpha}} b^{-\frac{1+k\alpha}{\alpha^2}}\rightarrow 0$ as $k\rightarrow\infty$.
\end{lemma}


By applying Lemma \ref{converg} to the sequence $\left(\frac {a_k}{\gamma}\right)$ we deduce the following result.

\begin{corollary}\label{cor}
 Let $\alpha>0$, $C>0$, $b>1$ and $\gamma>0$ be real numbers and $(a_k)$ a sequence of nonnegative real numbers, such that
 $$0\leq a_{k+1}\leq C b^k a_k\left(\frac {a_k}{\gamma}\right)^\alpha.$$
  Then we have that $a_k\rightarrow 0$ as $k\rightarrow\infty$ if $a_0 \leq \gamma\,  C^{-\frac 1 \alpha} b^{-\frac 1 {\alpha^2}} $.
\end{corollary}

\begin{lemma}[cf. \cite{AguCaf86} A2]\label{lem:iteration2}
	Suppose that a sequence $(a_k)$ of nonnegative numbers satisfies
	$$a_{k+1}^2\leq c_k (a_k-a_{k+1}), \qquad k \geq 1,$$
	for some bounded, nonnegative sequence $(c_k)$. Then, we have that
	$$a_{k+1}\leq\frac{\sqrt{\max_{i\leq k}c_i}\sqrt{a_1}}{\sqrt{k}},\qquad \mbox{for $k\geq 1,\qquad$.}$$
\end{lemma}
\begin{proof}
	Since we necessarily have $a_{k}\geq a_{k+1}\geq 0$, we get via a telescoping sum
	\begin{align*}
	k\,a_{k+1}^2&\leq \sum_{i=1}^k a_{i+1}^2\leq\sum_{i=1}^kc_i (a_i-a_{i+1})
\leq \big(\max_{i\leq k} c_i\big)\left(a_1-a_{k+1}\right)\leq\big(\max_{i\leq k} c_i\big)a_1.
	\end{align*}
	Dividing by $k$ and taking the square root concludes the proof.
\end{proof}

\begin{lemma}[cf. \cite{beck2016elliptic}, Lemma B.3]\label{lem:C-alpha-iteration}
Assume that $\phi(\rho)$ is a real-valued, nonnegative, nondecreasing function defined on the interval $[0,R_1]$. Assume further that there exists a number $\sigma \in (0,1)$ such that for all $R \leq R_1$ we have
\begin{equation*}
	\phi(\sigma R)\leq (\sigma^{\alpha_1} + \kappa)\phi(R)+CR^{\alpha_2}
\end{equation*}
for some nonnegative constant $C$, some number $\kappa \geq 0$, and positive exponents $\alpha_1>\alpha_2$. Then, there exists a positive number $\kappa_0 = \kappa(\sigma,\alpha_1,\alpha_2)$ such that for $\kappa \leq \kappa_0$ and all $r \leq R \leq R_1$ we have
\begin{equation*}
	\phi(r)\leq c(\sigma,\alpha_1,\alpha_2) \left(\left(\frac{r}{R}\right)^{\alpha_2}\phi(R)+Ar^{\alpha_2}\right).
\end{equation*}
\end{lemma}

\subsection{Local H\"older continuity in the continuous case}
First, we will present a brief overview of the proof of the local H\"older continuity of weak solutions to elliptic equations.
This is, by now, a classical result in the PDE analysis literature, however since our proof of the discrete counterpart of this result
proceeds along similar, but much more technical lines, readers may find this short overview helpful, regardless. The main ideas of the proof have been used in \cite{CaVa10}.
For simplicity, we will restrict ourselves to the homogeneous case. We begin by proving the Caccioppoli inequality stated in the next theorem.

\begin{theorem}\label{thm:ContBounded0}
		Let $\Omega\subset \setR^n$ be a domain and $A\in L^\infty(\Omega;\setR^{n\times n})$ a uniformly elliptic matrix, i.e., there is a $c>0$ such that
		\begin{equation}\label{eq:uniformlyElliptic}
		 A(x)v\cdot v\geq c |v|^2
		\end{equation}
 for any $v\in \setR^n$ and $x\in\Omega$. Let $u\in W^{1,2}(\Omega)$ be a weak solution to $\divergence(A\nabla u)=0$, i.e.,
		\begin{equation}\label{eq:contWeakSol}
			\int_\Omega A\nabla u\cdot \nabla \phi\dx=0
		\end{equation}
		for all $\phi\in W^{1,2}_0(\Omega)$. Then, we have that
		\begin{equation}\label{eq:ContCac}
			\dashint_{\supp \eta} \abs{\nabla (u-c)_+}^2 \abs{\eta}^2\dx\lesssim  \dashint_{\supp \eta} \abs{(u-c)_+}^2\abs{\nabla \eta}^2\dx
		\end{equation}
		for any function $\eta\in C^\infty_0(\Omega)$ and any $c>0$, where $(u-c)_+=(u-c)\vee 0$.
\end{theorem}

\begin{proof}
	We test equation \eqref{eq:contWeakSol} against $\phi=(u-c)_+\eta^2$. Note that $\nabla u= \nabla (u-c)=\nabla (u-c)_+$ on $\supp (u-c)_+$. This gives
	\begin{equation}\label{eq:contCac1}
	\begin{aligned}
		0&=\dashint_{\supp \eta} A \nabla u \cdot \nabla \left(\eta^2 (u-c)_+\right)\dx\\
		&=\dashint_{\supp \eta} A\nabla (u-c)_+\cdot \nabla ((u-c)_+)\eta^2\dx+\dashint_{\supp \eta} A\nabla (u-c)_+ \cdot 2(\nabla \eta) \eta (u-c)_+\dx\\
		&=:I+II.
	\end{aligned}
	\end{equation}
	We use the uniform ellipticity of $A$ to deduce that
	\begin{equation}\label{eq:contCac2}
		I\gtrsim \dashint_{\supp \eta} \abs{\nabla (u-c)_+}^2\eta^2\dx.
	\end{equation}
	The boundedness of $A$, H\"older's inequality, and Young's inequality yield
	\begin{equation}\label{eq:contCac3}
	\begin{aligned}
		\abs{II}&\lesssim \dashint_{\supp \eta}\abs{\nabla (u-c)_+}\abs{\eta}\abs{\nabla \eta}\abs{(u-c)_+}\dx\\
		&\leq \epsilon \dashint_{\supp \eta} \abs{\nabla(u-c)_+}^2\abs{\eta}^2\dx+C_\epsilon \dashint_{\supp \eta} \abs{(u-c)_+}^2\abs{\nabla \eta}^2\dx.
	\end{aligned}
	\end{equation}
Inserting the inequalities \eqref{eq:contCac2} and \eqref{eq:contCac3} into equation \eqref{eq:contCac1} and absorbing the first term of $II$ into $I$ proves the claim.
\end{proof}

\begin{theorem}\label{thm:ContBounded}
	Let $\Omega\subset \setR^n$ be a domain and $A\in L^\infty(\Omega;\setR^{n\times n})$ a uniformly elliptic matrix. Let $u\in W^{1,2}(\Omega)$ be a weak solution to $\divergence(A\nabla u)=0$. Then we have, for every $c>0$,
		\begin{equation}\label{eq:ucbound}
			\sup_{B(x_0,R)} \abs{(u-c)_+}^2\lesssim \dashint_{B(x_0,2R)}\abs{(u-c)_+}^2\dx
		\end{equation}
		for any ball $B(x_0,R)$ with $B(x_0,2R)\subset\Omega$.
	\end{theorem}
	
	\begin{proof}
		We define $\gamma_k=\gamma_{\infty}\left(1-2^{-k}\right)$ and $B_k:=B\left(x_0,\left(1+2^{-k}\right)R\right)$ where $\gamma_{\infty}>0$ is to be chosen later. Clearly, $\gamma_0 = 0$, $\lim_{k \rightarrow \infty} \gamma_k = \gamma_\infty$, and $B(x_0,2R) = B_0 \supset B_1 \supset \cdots \supset B_k \supset B_{k+1} \supset \cdots \supset B(x_0,R)$. We then define compactly supported $C^{\infty}$-functions $\phi_k$ such that
		\begin{equation}\label{eq:DefPhikCont}		
            \begin{aligned}
             \supp \phi_k&\subset B_{k},\quad 0 \leq \phi_k \leq 1,\\
             \phi_k&\equiv 1 \text{ on } B_{k+1},\\
             \abs{\nabla\phi_k}&\lesssim R^{-1}2^{k},
            \end{aligned}
		\end{equation}
		and the sequence
		\begin{equation}\label{eq:U_kContDef}	
				U_k:=\dashint_{B(x_0,2R)} \bigabs{(u-c-\gamma_k)_+}^2\abs{\phi_k}^2\dx, \qquad k=0,1,\dots .
		\end{equation}
		We use scaling-invariant norms $\normm{\cdot}_p$ with $p \in [1,\infty)$ defined by $\normm{f}_p^p:=\dashint_{B(x_0,2R)}\abs{f}^p\dx$ and apply H\"older's inequality, the Sobolev embedding theorem, and equation \eqref{eq:ContCac} to get (assume that $n \geq 3$ for simplicity, and let $2^\star:= 2n/(n-2)$ denote the critical Sobolev index; the bounds below are easily adjusted in the case of $n=2$ to reach the desired conclusion by choosing $2^\star$ as a large positive integer):
		\begin{equation}\label{eq:U_kCont1}
			\begin{aligned}
				U_k&:\leq \normm{(u-c-\gamma_k)_+\phi_k}_{2^\star}^2\normm{\chi_{\{u-c>\gamma_k\}\cap\supp\phi_k}}_{n}^2\\
				&\lesssim R^2\normm{\nabla((u-c-\gamma_k)_+\phi_k)}_{2}^2\normm{\chi_{\{u-c>\gamma_k\}\cap\supp\phi_k}}_{n}^2\\
				&\lesssim R^2\left( \normm{\nabla((u-c-\gamma_k)_+)\phi_k}_{2}^2+ \normm{(u-c-\gamma_k)_+\nabla \phi_k}_{2}^2\right)\normm{\chi_{\{u-c>\gamma_k\}\cap\supp\phi_k}}_{n}^2\\
				&\lesssim R^2\normm{(u-c-\gamma_k)_+\nabla \phi_k}_2^2 \, \normm{\chi_{\{u-c>\gamma_k\}\cap\supp\phi_k}}_{n}^2,
			\end{aligned}
		\end{equation}
		where in the transition to the last line we applied Theorem \ref{thm:ContBounded0}. For the first factor on the right-hand side in the final line of inequality \eqref{eq:U_kCont1} we use that $\abs{\nabla \phi_k}\lesssim R^{-1}2^{k}$ and $\phi_{k-1}\equiv 1$ on $\supp \phi_k$, and that $(\gamma_k)$ is a monotonically increasing sequence, to get
		\begin{equation}\label{eq:U_KCont2}
			\normm{(u-c-\gamma_k)_+\nabla \phi_k}_2^2\lesssim R^{-2}2^{2k}\normm{(u-c-\gamma_{k-1})_+\phi_{k-1}}_2^2.
		\end{equation}
		On $\{u-c>\gamma_k\}$, we have
		\begin{equation}\label{eq:ContWeakType1}
			u-c-\gamma_{k-1}>\gamma_k-\gamma_{k-1}=\gamma_\infty \left(2^{-(k-1)}-2^{-k}\right)=\gamma_\infty 2^{-k}.
		\end{equation}
		Furthermore, recall from \eqref{eq:DefPhikCont} that $\phi_{k-1}\equiv 1$ on $\support\, \phi_{k}$. Together with inequality \eqref{eq:ContWeakType1} this yields
		\begin{equation}\label{eq:ContWeakType2}
			\begin{aligned}
			\int_{B(x_0,2R)} [(u-c-\gamma_{k-1})_+]^2\phi_{k-1}^2\dx&\geq \int_{\support\, \phi_k}\chi_{\{u-c>\gamma_{k}\}}[(u-c-\gamma_{k-1})_+]^2\dx\\
			& \geq \gamma_\infty^2 \,2^{-2k} \abs{\support\,\phi_k\cap\{u-c>\gamma_k\}}.
			\end{aligned}
		\end{equation}
		We can use inequality \eqref{eq:ContWeakType2} to obtain the following weak-type estimate:
		\begin{equation}\label{eq:U_kCont3}
			 \normm{\chi_{\{u-c>\gamma_k\}\cap\supp\phi_k}}_{n}^2=\left(\frac{\abs{\support\,\phi_k\cap\{u-c>\gamma_k\}}}{\abs{B(x_0,2R)}}\right)^{\frac{2}{n}}\lesssim 2^{2k} \left( \frac{\normm{(u-c-\gamma_{k-1})_+\phi_{k-1}}_2^2}{\gamma_\infty^2}\right)^{\frac 2n}.
		\end{equation}
		Inserting the inequalities \eqref{eq:U_KCont2} and \eqref{eq:U_kCont3} in inequality \eqref{eq:U_kCont1} yields, with a positive constant $C$,
		independent of $\gamma_\infty$ and $k$,
		\begin{equation*}
			U_{k}\leq C 2^{3k} U_{k-1}\left(\frac{U_{k-1}}{\gamma_\infty^2}\right)^{\frac 2n}, \qquad k=1,2,\dots .
		\end{equation*}
		This then allows us to apply Corollary \ref{cor} with $b=2^3$ and $\alpha = 2/n$
		and $\gamma = \gamma_\infty^2:= C^{n/2} 2^{3n^2/4} U_0$,
		to deduce that $U_k\rightarrow 0$ as $k\rightarrow\infty$, and therefore $(u-c-\gamma_\infty)_+ = 0$ a.e. on $B(x_0,R)$, regardless of the sign of $u-c$.
		Hence, $\abs{(u-c)_+}^2 \leq \gamma_\infty^2$ on $B(x_0,R)$. On the ball $B(x_0,R)\subseteq\bigcap_{k\in\mathbb{N}}\supp \eta_k$ this means that
		\begin{equation*}
			\abs{(u-c)_+}^2\leq \gamma_\infty^2\sim U_0 \leq \dashint_{B(x_0,2R)}\abs{(u-c)_+}^2\dx,
		\end{equation*}
		because $\gamma_0=0$ and $0 \leq \phi_0 \leq 1$, which then implies
		\eqref{eq:ucbound}.
	\end{proof}
	
	From this result one can deduce the local $C^\alpha$-continuity of weak solutions.
	
	\begin{theorem}\label{thm:ContCalpha}
		Let $\Omega\subset \setR^n$ be a domain and $A\in L^\infty(\Omega;\setR^{n\times n})$ a uniformly elliptic matrix-valued function. Let $u\in W^{1,2}(\Omega)$ be a weak solution to $\divergence(A\nabla u)=0$.
		Then, there is an $\alpha>0$, such that
		\begin{equation*}
			\osc_{B(x,r)}u\leq C r^{\alpha}
		\end{equation*}
		for all $r \in (0, R]$ such that $B(x,4R)\subset\Omega$.
	\end{theorem}
	
	To prove Theorem \ref{thm:ContCalpha} we require the following intermediate result.
	\begin{lemma}\label{lem:ContLemma}
		Under the assumptions of Theorem \ref{thm:ContCalpha}, for each $\gamma \in (0,1)$ there exists a $\tau \in (0,1)$ that depends on $\gamma$ such that for every ball $B(x_0,R)$ such that $B(x_0,4R)\subset\Omega$, with $u\leq 1$ on $B(x_0,4R)$ and $\abs{\{u\leq 0\}\cap B(x_0,2R)}\geq \gamma \abs{B(x_0,2R)}$, we have that
		\begin{equation}\label{eq:ContSupSmall}
			\sup_{B(x_0,R)}u\leq 1 -\tau.
		\end{equation}
	\end{lemma}
	\begin{proof}
		We define
		\begin{align*}
			\lambda_k:=1-2^{-k},\qquad
			u_k:=\frac {1}{1-\lambda_k}(u-\lambda_k)_+,\qquad
			A_k:=B(x_0,2R)\cap\{u_k>0\}.
		\end{align*}
	As, by hypothesis, $u \leq 1$ on $B(x_0,4R)$, it follows that $0 < u_k \leq 1$ on $A_k$.
	By applying Theorem \ref{thm:ContBounded} to $u_k$ we deduce that
	\begin{equation}\label{eq:ContBdd}
		\sup_{B(x_0,R)}u_k\lesssim \bigg(\dashint_{B(x_0,2R)}\abs{u_k}^2\dx\bigg)^{\frac 12} = \bigg(\frac{1}{|B(x_0,2R)|}\int_{A_k}\abs{u_k}^2\dx\bigg)^{\frac 12} \lesssim \left(\frac{\abs{A_k}}{R^n}\right)^{\frac 12}.
	\end{equation}
	If we can find a $\bar{k}$ such that $\abs{A_{\bar{k}}}\leq \beta R^n$ for a sufficiently small $\beta$, we will have that $u_{\bar{k}}\leq \frac 12$ on  $B(x_0,R)$.
	
	Assume that $u_{k+1}>0$. We then have that
	\begin{align*}
		 0&<\frac{u-\left(1-\frac{1}{2^{k+1}}\right)}{\frac{1}{2^{k+1}}}=2^{k+1}\left(u-1+\frac{1}{2^k}-\frac{1}{2^k}+\frac{1}{2^{k+1}}\right)\\
		&=2^{k+1}\left(u-\lambda_k-\frac{1}{2^{k+1}}\right)\leq 2 \left(u_k-\frac 12\right).
	\end{align*}
	This means that
	\begin{equation}\label{eq:A_kCont}
		A_{k+1}\cap\left\{0<u_k<\frac 12\right\}=\emptyset.
	\end{equation}
	We have $\left\{0<u_k<\frac 12\right\}\subset A_k$ and $A_{k+1}\subset A_k$.  This yields
	\begin{equation}\label{eq:ContAkMeasure}
		\abs{A_k}\geq\left|\left(\left\{0<u_k<\frac 12\right\}\cap B(x_0,2R)\right)\cup A_{k+1}\right|=\left|\left(\left\{0<u_k<\frac 12\right\}\cap B(x_0,2R)\right)\right|+\left|A_{k+1}\right|.
	\end{equation}
	Note that if $x \in A_{k+1}$ then $u_k(x) \geq \frac{1}{2}$, and if $x \in B(x_0,2R)\setminus A_{k+1}$ then $0 < u_k(x) < \frac{1}{2}$.
	Note also that we have $\abs{B(x_0,2R)\cap\{u_k=0\}}\geq\gamma \abs{B(x_0,2R)}$. Therefore, we can use Poincar\'e's inequality to get
	\begin{equation}\label{eq:Akestimate1}
	\begin{aligned}
		\abs{A_{k+1}} & = 2 \int_{A_{k+1}} \frac{1}{2} \dx \leq 2 \int_{A_{k+1}} \frac{1}{2} \dx + 2 \int_{B(x_0,2R)\setminus A_{k+1}} u_k \dx\\		
		& =  2 \int_{B(x_0,2R)}\min\left\{u_k,\frac 12\right\}\dx\\
		&\lesssim R \int_{B(x_0,2R)} \left|\nabla\left(\min\left\{u_k,\frac 12\right\}\right)\right|\dx\\
		& = R \int_{\left\{0<u_k<\frac 12\right\} \cap B(x_0,2R)} \left|\nabla u_k\right|\dx\\
		&\leq R\left(\int_{B(x_0,2R)}\abs{\nabla u_k}^2\dx\right)^{\frac 12}\left|\left\{0<u_k<\frac 12\right\}\cap B(x_0,2R)\right|^{\frac 12}.
	\end{aligned}
	\end{equation}
By inequality \eqref{eq:ContCac}, with $\eta \in C^\infty_0(B(x_0,4R))$ such that $\eta \equiv 1$ on $B(x_0,2R)$ and $|\nabla
	\eta | \lesssim R^{-1}$, we find that
	\begin{equation}\label{eq:ContCacAppl}
		R^2\int_{B(x_0,2R)}\abs{\nabla u_k}^2\dx\lesssim \int_{B(x_0,4R)}\abs{u_k}^2\dx\lesssim R^n.
	\end{equation}
	Now, by inserting inequalities \eqref{eq:ContAkMeasure} and \eqref{eq:ContCacAppl} into inequality \eqref{eq:Akestimate1} yields
	\begin{equation*}
		\abs{A_{k+1}}\leq R^{\frac{n}{2}} \left(\abs{A_k}-\abs{A_{k+1}}\right)^{\frac 12}.
	\end{equation*}
	Consequently, we can use the iteration from Lemma \ref{lem:iteration2} to deduce the existence of a $\bar{k}$ such that $u_{\bar{k}}\leq \frac 12$ on $B(x_0,R)$.
	This gives
	\begin{equation*}
	 2^{\bar{k}}\left(u-1+\frac{1}{2^{\bar{k}}}\right) \leq \frac 12
	\end{equation*}
	and therefore
	\begin{equation*}
		u\leq 1-\frac{1}{2^{\bar{k}+1}},
	\end{equation*}
	which proves the lemma for $\tau=\frac{1}{2^{\bar{k}+1}}$.
	\end{proof}
	
	\begin{proof}[Proof of Theorem \ref{thm:ContCalpha}]
		First, note that $\osc (cu+d)=\abs{c}\osc{u}$ for constants $c,d\in\setR$. We define $\tilde{u}:=u-\frac 12 \left(\inf_{B(x_0,2R)}u+\sup_{B(x_0,2R)}u\right)$ and $\tilde{\tilde{u}}=\frac{1}{\norm{\tilde{u}}_{L^\infty(B(x_0,2R))}}\tilde{u}$. This gives $\osc_{B(x_0,2R)}\tilde{\tilde{u}}=2$. As $u$ is a solution to $-\divergence(A\nabla u)=0$, $-u$ is a solution as well. Note that we have either $\abs{\{\tilde{\tilde{u}}\leq 0\}\cap  B(x_0,2R)}\geq \frac 12 \abs{B(x_0,2R)}$ or $\abs{\{-\tilde{\tilde{u}}\geq 0\}\cap  B(x_0,2R)}\geq \frac 12 \abs{B(x_0,2R)}$. Thus we can assume that $\abs{\{\tilde{\tilde{u}}\leq 0\}\cap  B(x_0,2R)}\geq \frac 12 \abs{B(x_0,2R)}$ without loss of generality. Of course, this means that  $\abs{\{\tilde{u}_+=0\}\cap  B(x_0,2R)}\geq \frac 12 \abs{B(x_0,2R)}$; also, clearly, $-1 \leq
		\tilde{\tilde{u}} \leq 1$, and we can therefore apply Lemma \ref{lem:ContLemma} with $\gamma = \frac 12$ to deduce the existence of a
		$\tau \in (0,1)$ such that
		\begin{align}\label{eq:ContOsc}
		\begin{aligned}
				\osc_{B(x_0,R)}\tilde{\tilde{u}}&\leq 1 + \sup_{B(x_0,R)} \tilde{\tilde{u}} \leq 2-\tau = \osc_{B(x_0,2R)}\tilde{\tilde{u}}-\tau\\
				&= \osc_{B(x_0,2R)}\tilde{\tilde{u}}-\frac{\tau}{2} \osc_{B(x_0,2R)}\tilde{\tilde{u}} = \bigg(1-\frac{\tau}{2}\bigg)\osc_{B(x_0,2R)}\tilde{\tilde{u}} = 2^{-\alpha_1} \osc_{B(x_0,2R)}\tilde{\tilde{u}},
			\end{aligned}
			\end{align}
		with $\alpha_1:= - \log_2 ( 1- \frac{\tau}{2}) \in (0,1)$ (because $\tau \in (0,1)$). Hence, upon rescaling \eqref{eq:ContOsc},
		$$\osc_{B(x_0,R)}\tilde{u} \leq 2^{-\alpha_1} \osc_{B(x_0,2R)}\tilde{u}.$$
			Now, Lemma \ref{lem:C-alpha-iteration} with $\kappa=0$, $C=0$, $\sigma=\frac 12$ and $\phi(\rho):= \osc_{B(x_0,2\rho)}\tilde{u}$ gives
			\begin{equation*}
				\osc_{B(x_0,r)}\tilde{u} \leq \osc_{B(x_0,2r)}\tilde{u} = \phi(r) \lesssim \left(\frac{r}{R}\right)^\alpha \phi(R),
			\end{equation*}
			for some $\alpha \in (0,\alpha_1)$ and $0 \leq r \leq R$. This then implies, with $R$ held fixed, the assertion of the theorem by noting that $\osc_{B(x_0,r)} u = \osc_{B(x_0,r)}\tilde{u}$.
	\end{proof}

\section{Triangulations and mesh-conditions}
In this section we will introduce common mesh-conditions that will be used throughout and prove properties of the corresponding triangulations. We will always assume that $\Omega\subset\setR^n$ is a polyhedral Lipschitz domain and $\mathcal{T}_h$ is a triangulation of that domain. Henceforth, a polyhedral domain will be understood to be a polyhedral Lipschitz domain, and we shall therefore write \textit{polyhedral domain} instead of \textit{polyhedral Lipschitz domain} for the sake of brevity. Here, by a triangulation of $\Omega$ we mean a subdivision of $\overline{\Omega}$ into closed $n$-dimensional simplices with pairwise disjoint interiors, whose union is $\overline{\Omega}$, and such that for any pair of simplices $T, T' \in \mathcal{T}_h$ such that $T\cap T'$ is nonempty, $T\cap T'$ is either a shared vertex, or a shared $k$-dimensional face, which is a $k$-dimensional closed simplex, for some $k \in \{1,\dots,n-1\}$. The ($0$-dimensional) vertices of the simplices will be referred to as nodes. First we will define the Lagrange basis for a triangulation.
\begin{definition}\label{def:LagrangeBasis}
	Let $\mathcal{T}_h$ be a triangulation of the polyhedral domain $\Omega\subset\setR^n$. For a node $x_i$ of the triangulation $\mathcal{T}_h$ we denote by $\psi_i$ the associated \textit{linear Lagrange basis function}, which is, by definition, a continuous function on $\overline\Omega$ such that $\psi_i(x_j)=\delta_{i,j}$ for any node $x_j$, and for any $T \in \mathcal{T}_h$ the restriction of $\psi_i$ to $T$ is an affine function of $n$ variables.
\end{definition}
	
We will now introduce various mesh-conditions.
	
\begin{definition}\label{def:shapeRegular}
		We call a triangulation $\mathcal{T}_h$ \textit{shape-regular} with shape-regularity parameter $\Gamma>1$, if one has, for each $T \in \mathcal{T}_h$,
		\begin{equation}\label{eq:radiiShapeRegular}
			h_T \leq R_{i,T}\Gamma,
		\end{equation}
where $R_{i,T}$ is the radius of the largest $n$-dimensional ball contained in $T$ (which we shall refer to as the \textit{inscribed ball} of $T$) and $h_T:=\mathrm{diam}\,T$.	
	\end{definition}

	Next, we will introduce two important notions: $A$-\textit{nonobtuseness} and \textit{uniform} $A$-\textit{acuteness}.
	\begin{definition}
		Let $A\in L^{\infty}(\Omega;\setR^{n\times n})$ be a uniformly elliptic matrix-valued function. We call a triangulation $\mathcal{T}_h$ of $\Omega$ $A$-\textit{nonobtuse} if
		\begin{equation}\label{eq:aNonObtuse}
			\int_T A\nabla\psi_i\cdot\nabla\psi_j\dx\leq 0
		\end{equation}
		for any $T\in\mathcal{T}_h$ and for any $i\neq j$.
		
		We will call a triangulation $\mathcal{T}_h$ \textit{uniformly} $A$-\textit{acute} if there is a positive constant $\gamma$ such that
		\begin{equation}\label{eq:aAcute}
			\int_T A\nabla\psi_i\cdot\nabla\psi_j\dx\leq -\gamma\, \norm{\nabla\psi_i}_{L^2(T)}\norm{\nabla\psi_j}_{L^2(T)}
		\end{equation}
		for any $T\in\mathcal{T}_h$ and any $i\neq j$ with $T\subset \supp{\psi_i}\cap\supp{\psi_j}$.
		
		Note that if $A$ is proportional to the identity matrix, this definition coincides with the geometric idea of a nonobtuse triangulation. The existence and construction of uniformly $A$-acute and $A$-nonobtuse triangulations is discussed in detail in \cite{CRGM07}.
	\end{definition}

We will now recall a few properties of shape-regular families of triangulations $\mathcal{T}_h$. For any simplex $T\in\mathcal{T}_h$, there is an invertible linear transformation $B_T$ with bounded inverse that maps $T$ onto the standard simplex in $\setR^n$ with the nodes $(0,\ldots,0)$, $(h_T,0,\ldots,0)$, $\ldots,$ $(0,\ldots,0,h_T)$. The norms of $B_T$ and its inverse are bounded uniformly for all $T\in\mathcal{T}_h$. Furthermore, for any node $x_i$ of a simplex $T\in\mathcal{T}_h$ we have that
		\begin{equation}\label{eq:h_Tdef}
			\abs{\nabla\psi_i}\sim h_T^{-1}.
		\end{equation}
	Furthermore, if $T\in\mathcal{T}_h$ and $S\in\mathcal{T}_h$ have a nonempty intersection, we have
		\begin{equation}\label{eq:Intersection}
		h_T\sim h_S.
	\end{equation}


For simplicity, for any set $A$ contained in $\overline\Omega$,  we will write
\begin{equation*}
	\overline{\Omega(A)}:=\bigcup_{T\in\mathcal{T}_h\,:\,T\cap A\neq \emptyset}  T,
\end{equation*}
and the interior of that set is then
\begin{equation*}
	\Omega(A):=\overline{\Omega(A)}\setminus\partial\overline{\Omega(A)}.
\end{equation*}
In particular, we will write $P_i:=\Omega(\{x_i\})$ for any patch around a node $x_i$. We will also write
\begin{equation}\label{eq:OmegaPrime}
	\Omega'(A):=\bigcup_{i\,:\,x_i\in A}P_i.
\end{equation}
In short, $\Omega(A)$ is the set that contains all simplices that touch $A$ whereas $\Omega'(A)$ is the set of all simplices that have a node in $A$.

With the help of equation \eqref{eq:Intersection} one finds that
	\begin{equation}\label{eq:regularPatch}
		\mathrm{dist}\left(T,\Omega\setminus \Omega\left( T\right)\right)\gtrsim h_T.
	\end{equation}
Let us denote the connected component of $B(x_0,R)\cap\Omega$ that contains $x_0$ by $\mathcal{B}(x_0,R)$. Note that if $B(x_0,R)\subset\Omega$, we have $B(x_0,R)=\mathcal{B}(x_0,R)$. Then, there is a constant $Q>1$ that only depends on the shape-regularity parameter and the geometry of the domain $\Omega$, such that if $x_i\in T$ for some $T\in\mathcal{T}_h$ and $R\geq h_T$, we have
\begin{equation}\label{eq:regularBall}
		\Omega(\mathcal{B}(x_0,R))\subset B(x_0,QR).
	\end{equation}
	
	Additionally, for $x_0\in T$ for some $T\in\mathcal{T}_h$ and $R\geq h_T$, we have that
	\begin{equation}\label{eq:OmegaPrime1}
		\mathrm{dist}\left(x_0,\Omega\setminus \Omega'(\mathcal{B}(x_0,R))\right)\geq \kappa R.
	\end{equation}
	In particular, this means that there is a $\kappa>0$, that only depends on the shape-regularity parameter $\Gamma$ and the geometry of the domain $\Omega$, such that
	\begin{equation}\label{eq:OmegaPrime2}
		\mathcal{B}(x_0,\kappa R) \subset\Omega'(\mathcal{B}(x_0,R)).
	\end{equation}
	
		For the sake of simplicity we shall not indicate the dependence of the constants $Q$ and $\kappa$ on the geometrical properties of $\Omega$.

Next, we define the finite element space we shall be working with.
\begin{definition}
	Given a triangulation $\mathcal{T}_h$ of the domain $\Omega$, we denote the space of continuous functions that are affine on every $T\in\mathcal{T}_h$ by $V_h$. Note that $V_h\subset W^{1,2}(\Omega)$ and we can write
	\begin{equation}\label{eq:VhDecomposition}
		\begin{aligned}
			u_h&=\sum_i u(x_i)\psi_i
		\end{aligned}
	\end{equation}
	for any $u_h\in V_h$ with the Lagrange basis functions $\psi_i$ introduced in Definition \ref{def:LagrangeBasis}. This also shows that we can write the interpolatory projection $\Pi_h$ onto $V_h$ for a continuous function $f$ as
	\begin{equation}\label{eq:Pih}
		\Pi_h f:=\sum_i f(x_i)\psi_i.
	\end{equation}
The function $\Pih f$ is also called the (continuous piecewise affine) Lagrange interpolant of $f$. We will denote the space of continuous functions on $\overline\Omega$ that are affine on each $T\in\mathcal{T}_h$ and vanish on $\partial\Omega$ by $V_{h,0}$.
\end{definition}
We will formulate a few lemmas concerning continuous piecewise affine functions defined on triangulations. They will lead to a stronger version of Poincar\'e's inequality. We will always assume that $V_h$ and $V_{h,0}$ are finite element spaces associated with a shape-regular triangulation $\mathcal{T}_h$ of a polyhedral domain $\Omega$. In particular, this means that \eqref{eq:h_Tdef} is true. $\Pih$ will always denote the interpolatory projection operator onto $V_h$ as defined in equality \eqref{eq:Pih}.

\begin{lemma}
	We have
	\begin{alignat}{2}
   \max_T \abs{\Pih f} &\leq \max_{T} \abs{f} &\qquad \forall\, f \in C^0(\overline{\Omega}),\label{eq:PiL-stable_1} \\
   \max_T \abs{f-\Pih f} + \max_T h_T \abs{\nabla \Pih f} &\lesssim
   \max_T h_T \abs{\nabla f} &\qquad \forall\, f \in C^1(\overline{\Omega}). \label{eq:PiL-stable_2}
        \end{alignat}
\end{lemma}
\begin{proof}
	As an affine function on a simplex $T$, $\Pih f$ attains its maximum and minimum values on $T$ at a node of $T$. This means that we have
	\begin{equation*}
		\max_T \abs{\Pih f}= \max_{x_i\in T}\abs{\Pih f}=\max_{x_i\in T}\abs{ f}\leq  \max_{T} \abs{f},
	\end{equation*}
	which proves the inequality \eqref{eq:PiL-stable_1}. The bound
	$\max_T \abs{f- \Pih f}\lesssim \max_T h_T\abs{\nabla f}$ is standard from approximation theory.
	To prove that $\max_T h_T \abs{\nabla \Pih f} \lesssim \max_T h_T \abs{\nabla f}$ we can assume that $f(x)=0$ for some $x\in T$. We also use the relation \eqref{eq:h_Tdef} to get
	\begin{equation*}
		\max_T \abs{\nabla \Pih f}= \max_T \left|\sum_i f(x_i)\nabla\psi_i\right|\lesssim \max_{ T} \abs{f(x_i)} h_T^{-1}\lesssim \max_{ T}\abs{\nabla f}.
	\end{equation*}
	This concludes the proof of inequality \eqref{eq:PiL-stable_2}.
\end{proof}

We will denote the space of polynomials of degree $k$ or less by $\mathcal{P}_k$.
\begin{remark}
  \label{rem:local-zero-product}
  Suppose that $a_1, \dots, a_m\in\mathcal{P}_k$. Then, $a_1 \dotsm a_m = 0$
  if and only if at least one of the functions~$a_1, \dots, a_m$ is identically
  zero.
\end{remark}

The following useful Lemma \ref{lem:prod-rearrange} holds for polynomials defined on simplices $T\in\mathcal{T}_h$.
\begin{lemma}
  \label{lem:prod-rearrange}
  Let $a_1, \dots, a_m\in\mathcal{P}_k$. Then, for each element $T$ of a shape-regular family
$\mathcal{T}_h$,
  \begin{equation}\label{eq:prod-rearrange}
    \prod_{j=1}^m \max_T \abs{a_j} \sim \dashint_T \Bigabs{\prod_{j=1}^m a_j}\,\dx.
  \end{equation}
  The constants hidden in the $\sim$ symbol depend on~$k$ and $m$ and the shape-regularity parameter $\Gamma$ of
$\mathcal{T}_h$.
\end{lemma}
\begin{proof}
	Since the expressions on the two sides of \eqref{eq:prod-rearrange} are homogeneous in $a_1,\dots,a_m$, it suffices to show that there are constants $c$ and $C$ such that $0<c\leq \dashint_T \Bigabs{\prod_{j=1}^m a_j}\,\dx\leq C$ for any $a_1,\dots,a_m\in\mathcal{P}_k$ with $\max_T\abs{a_1}=\dots=\max_T\abs{a_m}=1$  and any simplex $T \in \mathcal{T}_h$. Obviously we can choose $C=1$. We will show the existence of $c$ using a compactness argument. Equation \eqref{eq:prod-rearrange} is invariant under linear scaling. Thus we can assume that $h_T=1$ in the sense of Definition \ref{def:shapeRegular}. First, we fix a simplex $T \in \mathcal{T}_h$ and note that the mapping $g:(a_1,\dots,a_m)\mapsto \dashint_T\abs{a_1\cdots a_m}\dx$ is continuous on $\mathcal{P}^{\otimes m}_k$, as this space is finite-dimensional and all norms are equivalent. Furthermore, the set $\{(a_1,\dots, a_m):\max_T\abs{a_1}=\dots=\max_T\abs{a_m}=1\}\subset\mathcal{P}^{\otimes m}_k$ is closed and bounded and therefore compact. Thus, it suffices to show that $g(a_1,\dots,a_m)>0$ for any $(a_1,\dots, a_m)$
	with $\max_T\abs{a_1}=\dots=\max_T\abs{a_m}=1$ to get
\begin{equation}\label{eq:infPolynomialIntegral}
	\inf_{\max_T\abs{a_1}=\dots=\max_T\abs{a_m}=1}\dashint_T \abs{a_1\cdots a_m}\dx>0.
\end{equation}
Assume the contrary, i.e., that $\dashint_T\abs{a_1\cdots a_m}\dx=0$. Then, $a_1\cdots a_m=0$ on $T$. However, because of Remark \ref{rem:local-zero-product}, this implies that at least one of the $a_1,\dots,a_m$ has to be identically zero, which contradicts $\max_T\abs{a_1}=\dots=\max_T\abs{a_m}=1$.
	
	If we now fix a node $x_0$ of $T$, we define the set $\mathcal{A}\subset \setR^n$ as the set of $n$-tuples $(x_1,\dots,x_n)$, which, together with $x_0$, form a simplex with diameter $h_T=1$ with a given shape-regularity parameter and see that $\mathcal{A}$ is bounded and closed and therefore compact. Furthermore, the mapping $f:\mathcal{A}\rightarrow \setR$ defined via
	$$(x_1,\dots,x_n)\mapsto \inf_{\max_T\abs{a_1}=\dots=\max_T\abs{a_m}=1}\dashint_T \abs{a_1\cdots a_m}\dx$$
	is continuous. This means that it suffices to show that $f(a_1,\dots,a_n)>0$ to prove the relation \eqref{eq:prod-rearrange}; but this is obviously true, thanks to inequality \eqref{eq:infPolynomialIntegral}.
\end{proof}

We shall also require two lemmas, which we now state. The first of them is elementary, while the second relies of Lemma \ref{lem:prod-rearrange}.

\begin{lemma}
  \label{lem:proj-product}
  For all~$f_h,g_h\in V_h$ and $a,b \in \setR$, we have
  \begin{align*}
    \Pih(f_h g_h) - f_h g_h &= \Pih\big((f_h-a)(g_h-b)\big) - \left( f_h-a\right) \left(g_h-b\right).
  \end{align*}
\end{lemma}

\begin{proof}
	$\Pih$ is linear and constant functions are contained in $V_h$. This means that we have $\Pih(ag_h)=ag_h$, $\Pih(bf_h)=bf_h$ and  $\Pih(ab)=ab$. This therefore yields that
	\begin{align*}
		&\Pih\big((f_h-a)(g_h-b)\big) - \left( f_h-a\right) \left(g_h-b\right)\\
		&=\Pih(f_h g_h)- \Pih(a g_h)-\Pih(b f_h)+ \Pih(ab)- f_h g_h + a g_h + b f_h - ab\\
		&= \Pih(f_h g_h) - f_h g_h.	
	\end{align*}	
	That concludes the proof.
\end{proof}

\begin{lemma}
  \label{lem:prod-commutator}
  For all $v_h, w_h \in V_h$ we have
  \begin{equation}\label{eq:prod-commutator1}
		\begin{aligned}
			\lefteqn{\max_T \abs{ v_h w_h - \Pih (v_h w_h)}+\max_T h_T\,\abs{\nabla (v_h w_h - \Pih (v_h w_h))}} \qquad  &  \\
			&\lesssim \dashint_T \abs{v_h -\mean{v_h}_T} \dy \, \dashint_T \abs{w_h- \mean{w_h}_T} \dy\\
		\end{aligned}
	\end{equation}
	and

	\begin{equation}\label{eq:prod-commutator2}
		\begin{aligned}
			\lefteqn{\max_T \abs{ v_h w_h - \Pih (v_h w_h)}+\max_T h_T\,\abs{\nabla (v_h w_h - \Pih (v_h w_h))}} \qquad  &  \\
			&\lesssim h_T\,\dashint_T \abs{\nabla v_h} \dy \, \dashint_T \abs{w_h- \mean{w_h}_T} \dy
		\end{aligned}
	\end{equation}
	for any $T\in\mathcal{T}_h$.
\end{lemma}

\begin{proof}
	Because $f_h$, $g_h$ are affine functions on every $T\in \mathcal{T}_h$, we deduce by an inverse estimate that
	\begin{equation}\label{eq:ProdInverse}
		\max_T h_T \abs{\nabla(v_h w_h - \Pih(v_h w_h))}\lesssim \max_T\abs{v_h {\color{blue}w_h}-\Pih(v_h w_h)}.
	\end{equation}
	Now, Lemma \ref{lem:proj-product} yields
	\begin{equation}
		\begin{aligned}
			\max_T \abs{ v_h w_h - \Pih (v_h w_h)}= &\max_T\huge|\left(v_h -\mean{v_h}_T\right)\left(w_h-\mean{w_h}_T\right)\\
				&-\Pih\left(\left(v_h -\mean{v_h}_T\right)\left(w_h-\mean{w_h}_T\right)\right)\huge|.
		\end{aligned}
	\end{equation}
	Combining this with the triangle inequality and inequality \eqref{eq:PiL-stable_1} gives
	\begin{equation}\label{eq:ProdCommutatorProof2}
		\max_T \abs{ v_h w_h - \Pih (v_h w_h)}\lesssim \max_T\left|\left(v_h -\mean{v_h}_T\right)\left(w_h-\mean{w_h}_T\right)\right|.
	\end{equation}
	We can now apply Lemma \ref{lem:prod-rearrange} to inequality \eqref{eq:ProdCommutatorProof2} to deduce that
	\begin{equation}\label{eq:ProdCommutatorProof3}
		\max_T \abs{ v_h w_h - \Pih (v_h w_h)}\lesssim \dashint_T |v_h -\mean{v_h}_T |\dx\, \dashint_T |w_h-\mean{w_h}_T|\dx,
	\end{equation}
	which proves the bound \eqref{eq:prod-commutator1}.
	
	The bound \eqref{eq:prod-commutator2} follows directly from equation \eqref{eq:prod-commutator1} by Poincar\'e's inequality on $T$.	
\end{proof}

\begin{remark}\label{rem:bound1}
	The right-hand sides of the inequalities \eqref{eq:prod-commutator1} and \eqref{eq:prod-commutator2} can be further bounded above by noting that
	\begin{equation*}
		\dashint_T|v_h-\mean{v_h}|\dx\leq \dashint_T |v_h|\dx+\dashint_T |\mean{v_h}|\dx\leq 2\dashint_T|v_h|\dx,
	\end{equation*}
and analogously in the case of $\dashint_T|w_h-\mean{w_h}|\dx$.
\end{remark}

Next, we will prove a Jensen-type inequality.
\begin{lemma}
  \label{lem:Pi-etah2}
  For $\eta_h \in V_h$ and $q \geq 1$ we have
  \begin{align*}
    \eta_h^q &\leq \Pih (\eta_h^q).
  \end{align*}
\end{lemma}

\begin{proof}
	We know that
	\begin{equation*}
		\sum_i\psi_i=\Pih(1)=1.
	\end{equation*}
	This allows us to use Jensen's inequality to deduce that
	  \begin{align*}
    \eta_h^q(x) = \Big( \sum_j \psi_j(x) \eta_h(x_j) \Big)^q
    \leq \sum_j \psi_j(x) \eta_h^q(x_j) = \big(\Pih (\eta_h^q)\big)(x)\qquad \forall\, x \in \overline{\Omega}.
  \end{align*}
That completes the proof.
\end{proof}

Finally, we will prove two versions of Poincar\'e's inequality for functions in $V_h$.

\begin{lemma}[Poincar\'e's inequality on patches]\label{lem:PoincarePatches}
	Let $x_i$ be a node of the triangulation $\mathcal{T}_h$ and $P_i=\Omega(x_i)$ the associated patch, and let $v_h\in V_h$ be a function with $v_h(x_0)=0$ for some node $x_0\in \overline{P_i}$. Then, we have Poincar\'e's inequality on $P_i$, that is
	\begin{equation*}
		\int_{P_i}\abs{v_h}\dx\lesssim h_i \int_{P_i} \abs{\nabla v_h}\dx,
	\end{equation*}
	where $h_i=h_T$ for some $T\in\mathcal{T}_h$ with $T\subset \overline{P_i}$ in the sense of Definition \ref{def:shapeRegular}. Note, that any $S,T\in\mathcal{T}_h$ with $S,T\subset \overline{P_i}$ share the node $x_i$; hence,
	$h_T\sim h_S$ by equation \eqref{eq:Intersection} and the definition of $h_i$ is meaningful.
\end{lemma}

\begin{proof}
	Let $T_0\subset P_i$ be a simplex that has $x_0$ as a node and let $B_0$ be its inscribed ball. By shape-regularity (see inequality \eqref{eq:radiiShapeRegular}), there is a fixed $\gamma>0$ such that $\abs{B_0}\geq\gamma\abs{P_i}$. Define $\dashint_{B_0}v_h\dx=:\mean{v_h}_0$. Also by shape-regularity, we find that the diameter of $P_i$ is comparable to $h_i$. Then, Poincar\'e's inequality yields
	
	\begin{equation}\label{eq:poincarePatch1}
		\int_{P_i}\bigabs{v_h-\mean{v_h}_0}\dx\lesssim h_i\int_{P_i}\abs{\nabla v_h}\dx.
	\end{equation}
	On the other hand, we have Poincar\'e's inequality on $T_0$ because $v_h$ is affine on any $T\in\mathcal{T}_h$ and we get
	\begin{equation}\label{eq:poincarePatch2}
		 \int_{P_i}\bigabs{\mean{v_h}_0}\dx\leq\frac{\abs{P_i}}{\abs{B_0}}\int_{B_0}\abs{v_h}\dx\lesssim\int_{T_0}\abs{v_h}\dx\lesssim h_i\int_{T_0}\abs{\nabla v_h}\dx\leq h_i\int_{P_i}\abs{\nabla v_h}\dx,
	\end{equation}
where the third inequality stems from the fact that $v_h$ vanishes at the node $x_0 \in T_0$.
Combining the inequalities \eqref{eq:poincarePatch1} and \eqref{eq:poincarePatch2} yields
	\begin{equation*}
		\int_{P_i}\abs{v_h}\dx\leq\int_{P_i}\bigabs{v_h-\mean{v_h}_0}\dx+\int_{P_i}\bigabs{\mean{v_h}_0}\dx\lesssim h_i\int_{P_i}\abs{\nabla v_h}\dx.
	\end{equation*}
	This concludes the proof.
\end{proof}

\begin{theorem}[Poincar\'e's inequality for $V_h$]\label{thm:Poincare}
	Let $\mathcal{T}_h$ be a shape-regular triangulation of the polyhedral domain $\Omega \subset \setR^n$ with associated finite element space $V_h$. Let $v_h\in V_h$ be a nonnegative function and let $A=\bigcup_{i} T_i$ be a connected set with diameter $R$ for a set of simplices $\{T_i\}\subset\mathcal{T}_h$.  Suppose that
	
	\begin{equation}\label{eq:PoincareAssumption}
		\bigg|\Theta\cap\bigg(\bigcup_{i\,:\,v_h(x_i)=0}P_i\bigg)\bigg|\geq \gamma |\Theta|
	\end{equation}
	for some $\gamma>0$. Then, there is a constant $c$ that depends on $\gamma$, the Poincar\'e constant of $\Theta$, and the shape-regularity parameter $\Gamma$ of the triangulation, such that
	\begin{equation*}
		\int_\Theta \abs{v_h}\dx\leq c R \int_\Theta \abs{\nabla v_h}\dx.
	\end{equation*}
\end{theorem}

\begin{proof}
	First, note that $\Theta$ is a bounded Lipschitz domain and therefore, its Poincar\'e constant is finite. Let $\mathcal{N}$ be the index set of nodes that are either nodes $x_i$ with $v_h(x_i)=0$ or that are connected to one of those nodes in $\overline{\Theta}$ by an edge. We write $w_h:=\sum_{i\in\mathcal{N}}v_h(x_i)\psi_i$ and $\tilde{w}_h:=\sum_{i\notin \mathcal{N}}v_h(x_i)\psi_i$. Note that $v_h=w_h+\tilde{w}_h$ and $\tilde{w}_h=0$ on $\Theta\cap\left(\bigcup_{v_h(x_i)=0}P_i\right)$ and thus, $\dashint_{\Theta\cap\left(\bigcup_{v_h(x_i)=0}P_i\right)}\tilde{w}\dx=0.$ Therefore, assumption \eqref{eq:PoincareAssumption} implies that we can use Poincar\'e's inequality on $\tilde{w}_h$. This leaves us with
	\begin{equation}\label{eq:poincare1}
		\int_\Theta\abs{\tilde{w}_h}\dx\lesssim R \int_\Theta\abs{\nabla\tilde{w}_h}\dx.
	\end{equation}
	
	On the other hand, we know that $w_h$ has a zero on every patch $P_i\cap\Theta =\Omega(x_i)\cap\Theta$ for any node $x_i\in \overline{\Theta}$ because every node $x_i$ with $w_h(x_i)\neq 0$ has $i\in \mathcal{N}$ and is therefore connected to a node with $v(x_j)=0$ by an edge. This means that we can use Lemma \ref{lem:PoincarePatches} on any of those patches. This gives
	\begin{equation*}
		\int_\Theta \abs{w_h}\dx\leq\sum_{i\in\mathcal{N}}\int_{P_i\cap\Theta}\abs{w_h}\dx\lesssim \sum_{i\in\mathcal{N}} h_i \int_{P_i\cap\Theta}\abs{\nabla w_h}\dx.
	\end{equation*}
	We have $h_i\lesssim R$. Furthermore, the shape-regularity of $\mathcal{T}_h$ guarantees that each simplex can only be part of  a uniformly bounded number of patches. Thus, we get
	\begin{equation}\label{eq:poincare2}
		\int_\Theta \abs{w_h}\dx\lesssim R\int_\Theta\abs{\nabla w_h}\dx.
	\end{equation}
Adding the inequalities \eqref{eq:poincare1} and \eqref{eq:poincare2} yields 
\begin{equation}\label{eq:poincare3}
	\int_\Theta |v_h|\dx\leq \int_\Theta \abs{w_h}\dx+\int_\Theta\abs{\tilde{w}_h}\dx\lesssim R\left(\int_\Theta\abs{\nabla w_h}\dx+\int_\Theta\abs{\nabla\tilde{w}_h}\dx\right).
\end{equation}
We define
$$O:=\left(\bigcup_{i\in\mathcal{N}} \overline{P_i}\right)\cap \Theta.$$
Note that by definition, if $i\in\mathcal{N}$, then $P_i$ has a node where $v_h$ has a zero. Therefore, we get
\begin{equation}\label{eq:overlap}
	\int_O\abs{\nabla v_h}\dx\gtrsim \sum_{i\in\mathcal{N}} \int_{P_i\cap\Theta}\abs{\nabla v_h}\dx\gtrsim\sum_{i\in\mathcal{N}} h_i^{-1}\int_{P_i\cap\Theta}\abs{v_h}\dx,
\end{equation}
where $h_i$ is the diameter of one of the simplices $T\subset P_i$ and where we have used the fact that any $T\in\mathcal{T}_h$ is only part of finitely many patches and Lemma \ref{lem:PoincarePatches}. 

For simplicity, let us write $\Psi_{\mathcal{N}}=\sum_{i\in\mathcal{N}}\psi_i$. On each $T\in\mathcal{T}_h$, we therefore have $\abs{\Psi_{\mathcal{N}}}\leq 1$ and $\abs{\nabla\Psi_{\mathcal{N}}}\lesssim h_T^{-1}$. Furthermore, we can write 
\begin{equation*}\label{eq:wAlternative}
	w_h=\Pih\left(v_h\Psi_{\mathcal{N}}\right).
\end{equation*}
Thus, we can use equation \eqref{eq:PiL-stable_2} and the product rule to find
\begin{equation}\label{eq:gradWAlternative}
	\abs{\nabla w_h}\lesssim\abs{\nabla v_h}+h_T^{-1}\abs{v_h}
\end{equation}
on every $T\in\mathcal{T}_h$. Additionally, we have
\begin{equation}\label{eq:gradWIntegral}
	\int_O\abs{\nabla w_h}\dx\leq\sum_{i\in\mathcal{N}}\int_{P_i\cap\Theta}\abs{\nabla w_h}\dx.
\end{equation}
Combining inequalities \eqref{eq:gradWAlternative}, \eqref{eq:gradWIntegral} and equation \eqref{eq:Intersection} yields
\begin{equation}\label{eq:gradWIntegral2}
	\int_O\abs{\nabla w_h}\dx\lesssim \sum_{i\in\mathcal{N}}\int_{P_i\cap \Theta}\abs{\nabla v_h}+h_i^{-1}\abs{v_h}\dx.
\end{equation}
Therefore, we can combine inequalities \eqref{eq:overlap} and \eqref{eq:gradWIntegral2} to get
\begin{equation}\label{eq:gradWIntegral3}
	\int_O\abs{\nabla w_h}\dx\lesssim \int_O\abs{\nabla v_h}\dx.
\end{equation}
Analogously, we find that
\begin{equation}\label{eq:gradWTildeAlternative}
	\int_O\abs{\nabla \tilde{w}_h}\dx\lesssim \int_O\abs{\nabla v_h}\dx.
\end{equation}
Now note that $w_h=0$ on $\Theta\setminus O$. As $v_h=w_h+\tilde{w}_h$, we therefore have 
\begin{equation}\label{eq:wOutsideO}
	\abs{\nabla v_h}=\abs{\nabla w_h}+\abs{\nabla\tilde{w}_h}
\end{equation} 
on $\Theta\setminus O$. This means that we can finally combine inequalities \eqref{eq:gradWIntegral3}, \eqref{eq:gradWTildeAlternative} and \eqref{eq:wOutsideO} to find
\begin{equation}\label{eq:wTildewIntegral}
	\int_\Theta \abs{\nabla w_h}+\abs{\nabla\tilde{w}_h}\dx\lesssim \int_\Theta\abs{\nabla v_h}\dx.
\end{equation}
Together with inequality \eqref{eq:poincare3}, that completes the proof.
\end{proof}

We will also need a lemma that concerns connecting pairs of nodes in an $A$-nonobtuse triangulation in a way that allows us to work as if we were on a uniformly $A$-acute triangulation. This seems to be a small change but it allows us, for example, in the case where $A$ is the identity matrix to use triangulations that are generated by newest vertex bisection on a square, which are important cases, especially if we consider that it is one of the novelties of our approach that we do not demand quasi-uniformity of the triangulation.
\begin{lemma}\label{lem:nonObtuseAcuteSequence}
	Let $\mathcal{T}_h$ be an $A$-nonobtuse, shape-regular triangulation of a polyhedral domain $\Omega\subset \setR^n$. Denote the nodes of a simplex $T\in\mathcal{T}_h$ by $x_0,\dots,x_n$. Define $\skp{f}{g}_{T,A}:=\int_T A\nabla f\cdot\nabla g\dx$. Then, there is a constant $\tau \in (0,1)$ that only depends on the shape-regularity parameter $\Gamma$ of $\mathcal{T}_h$, such that any pair of
nodes $x_j$ and $x_k$ of $T$ can be connected by a sequence of nodes $\{x_j=y_0,y_1,\dots,y_N=x_k\}$ of $T$,  with $N \leq n$, belonging to the same simplex $T$ and such that
	\begin{equation}\label{eq:AcuteSequence}
		-\int_\Omega A\nabla \psi_i\cdot \nabla \psi_{i+1}\dx\geq \tau \int_\Omega A\nabla\psi_i\cdot\nabla\psi_i\dx,\qquad i=0,\dots,N-1,
	\end{equation}
	where $\psi_i$ is the basis function associated with the node $y_i$ in the sense of Definition \ref{def:LagrangeBasis}
%
\end{lemma}

\begin{proof} The proof proceeds in two steps. For two $n$-component vector functions $\zeta$ and $\xi$ defined on $\Omega$ we shall write
	\begin{equation*}
\skp{\zeta}{\xi}_{T,A}:=\int_T A \zeta \cdot \xi \dx\quad \mbox{and}\quad \skp{\zeta}{\xi}_A:=\int_\Omega A \zeta \cdot \xi \dx.
	\end{equation*}

 STEP 1. Let $L \subset \set{0, \dots, n}$ and
  $M:= \set{0,\dots,n} \setminus L$ with $L,M \neq \emptyset$. We will show that
  \begin{align}\label{eq:GraphConnected}
    \max_{l\in L,\,m \in M} (-\skp{\nabla \psi_l}{\nabla \psi_m}_{T,A}) \geq\tau
    \max_{r\in\{0,\dots,n\}}
    \skp{\nabla \psi_r}{\nabla \psi_r}_A
  \end{align}
	for some $\tau>0$ that depends on the shape-regularity constant $\Gamma$.
  Since the $\skp{\nabla \psi_l}{\nabla \psi_m}_{T,A}\leq 0$ for $l \neq m$,
  it suffices to prove that
  \begin{align*}
    (I) &:= \sum_{l\in L, \, m \in M} (-\skp{\nabla \psi_l}{\nabla \psi_m}_{T,A} )\gtrsim
          \max_{r\in\{0,\dots,n\}}
          \skp{\nabla \psi_r}{\nabla \psi_r}_A.
  \end{align*}
  As $\sum_{r=0}^n \psi_r=1$ on $T$ and $L \cup M = \set{0,\dots,n}$, we obtain
  \begin{align*}
    (I) &= \sum_{l\in L} \bigg( -\skp{\nabla \psi_l}{\nabla 1}_{T,A} + \sum_{m \in L}  \skp{\nabla \psi_l}{\nabla  \psi_m}_{T,A} \bigg) = \sum_{l,m \in L} \skp{\nabla \psi_l}{\nabla \psi_m}_{T,A} =: (II).
  \end{align*}
  Now define~$\eta := \sum_{l \in L} \psi_l$.
  Since~$\eta(x_m)=0$ for some~$m \in M$ and $A$ is uniformly elliptic, we obtain by \Poincare{}'s
  inequality, Lemma \ref{lem:prod-rearrange} and inverse estimates that
  \begin{align*}
    (II) = \skp{\nabla\eta}{\nabla\eta}_{T,A} &\geq c \norm{\nabla \eta}_{L^2(T)}^2  \sim h_T^{-2} \norm{\eta}_{L^2(T)}^2
           \sim h_T^{-2} \abs{T}\,\norm{\eta}_{L^\infty(T)}^2 \\
         &\sim h_T^{-2} \abs{T} \sim  \max_{r \in \{0,\dots,n\}} \skp{\nabla \psi_r}{\nabla \psi_r}_{T,A}\sim \max_{r \in \{0,\dots,n\}} \skp{\nabla \psi_r}{\nabla \psi_r}_A,
  \end{align*}
  where we have used the shape-regularity of the triangulation. That completes Step 1.

STEP 2.
Let us now assume that we were not able to find a sequence that connects $x_i$ and $x_j$ such that inequality \eqref{eq:AcuteSequence} holds. This would imply that we can find disjoint sets $K\subset\{1,\dots,n\}$ and $J\subset\{1,\dots,n\}$ with $i\in K$ and $j\in J$ such that inequality \eqref{eq:GraphConnected} does not hold, but this is obviously a contradiction.

\end{proof}

\section{A discrete Caccioppoli-type inequality}
In this section, we prove a discrete version of inequality \eqref{eq:ContCac}. First, we will define what we mean by an approximate solution to an elliptic partial differential equation.
\begin{definition}\label{def:DiscreteSolution}
	Let $A\in L^{\infty}(\Omega;\setR^{n\times n})$ be a uniformly elliptic matrix-valued  function.	
	Furthermore, let $F\in L^p(\Omega;\setR^n)$ and $f\in L^q(\Omega)$ be given functions with $p>n$ and $q>n/2$. Furthermore, let $\mathcal{T}_h$ be a triangulation of the polyhedral domain $\Omega$ and $V_h$ the corresponding finite element space. We call $u_h\in V_{h,0}$ an approximate solution to
	\begin{equation*}
		-\divergence(A\nabla u) = f - \divergence  F
	\end{equation*}
	with zero Dirichlet boundary condition provided that we have
	\begin{equation}\label{eq:DiscreteInhom}
		\int_\Omega A\nabla u_h\cdot \nabla \phi_h\dx=\int_\Omega f\phi_h\dx+\int_\Omega F\cdot \nabla \phi_h\dx
	\end{equation}
	for every $\phi_h\in V_{h,0}$.
\end{definition}

To prove the Caccioppoli inequality \eqref{eq:ContCac} we had to test equation \eqref{eq:contWeakSol} against $(u-c)_+\abs{\eta}^2$. In the discrete case however, $(u_h-c)_+\abs{\eta}^2$
is not an admissible test function in equation \eqref{eq:DiscreteInhom} because it is not in $V_{h,0}$. In particular, this leads to the conclusion that we will have to use a nodal version of $(u_h-c)_+$. This unfortunately means that $\nabla(u_h-c)_+$ and $\nabla u_h$ will not coincide on $\supp (u_h-c)_+$ anymore.
To overcome this problem, we will introduce the notion of \textit{discrete subsolution}, prove a Caccioppoli-type inequality for discrete subsolutions, and prove that the nodal version of $(u_h-c)_+$ is indeed a discrete subsolution under an additional assumption on the right-hand side, which we shall state.
\begin{definition}\label{def:Subsolution}
		Let $A\in L^{\infty}(\Omega;\setR^{n\times n})$ be a uniformly elliptic matrix-valued  function and let
		$F\in L^p(\Omega;\setR^n)$ and $f\in L^q(\Omega)$ be given functions with with $p>n$ and $q>n/2$. Furthermore, let $\mathcal{T}_h$ be a triangulation of the polyhedral domain $\Omega$ and $V_h$ the corresponding finite element space. We call $u_h\in V_{h,0}$ a discrete subsolution to
	\begin{equation*}
		-\divergence(A\nabla u) = f - \divergence  F
	\end{equation*}
	provided that we have
	\begin{equation}\label{eq:DiscreteInhomSubsol}
		\int_\Omega A\nabla u_h\cdot \nabla \phi_h\dx\leq \int_\Omega f\phi_h\dx+\int_\Omega F\cdot \nabla \phi_h\dx
	\end{equation}
	for every $\phi_h\in V_{h,0}$ with $\phi_h\geq 0$ on $\overline\Omega$. Of course, every solution in the sense of equation \eqref{eq:DiscreteInhom} is automatically a  subsolution in the sense of inequality \eqref{eq:DiscreteInhomSubsol}.
\end{definition}

Henceforth we shall concentrate on the case of $n \geq 3$ and recall that $2^\star:=2n/(n-2)$. The statements and proofs of the results below are easily adjusted in the case of $n=2$. Note that a similar inequality for discrete solutions to equations with vanishing right-hand side has been shown as an intermediate result in \cite{NS74} on quasi-uniform meshes and in \cite{DGS11} for highly graded meshes using superapproximation techniques.

\begin{theorem}[Caccioppoli inequality for discrete subsolutions]\label{thm:CaccioppoliInhom}
	Let $A\in L^{\infty}(\Omega;\setR^{n\times n})$ be a uniformly elliptic matrix-valued function and let
	$F\in L^p(\Omega;\setR^n)$ and $f\in L^q(\Omega)$ be given functions with $\frac1 p = \frac1 n - \frac\delta n$ and $\frac1 q = \frac2 n - \frac\delta n$, where $\delta>0$. Furthermore, let $\mathcal{T}_h$ be a shape-regular triangulation of the polyhedral domain $\Omega$ and $V_h$ the corresponding finite element space.
	Let $u_h\in V_h$ be a discrete subsolution to $-\divergence(A\nabla u)=f-\divergence  F$ in the sense of inequality \eqref{eq:DiscreteInhomSubsol}, let $\eta\in C_0^{\infty}(\Omega)$ be nonnegative, and define $\eta_h=\Pi_h\eta$. Then, we have
	\begin{equation}\label{eq:DiscreteCaccioppoliInhom}
		\begin{aligned}
		&\int_\Omega \abs{\nabla u_h}^2\abs{\eta_h}^2\dx\lesssim \int_\Omega \abs{u_h}^2\abs{\nabla \eta_h}^2\dx\\
		&	\quad+\left(\norm{F}_{L^p(\Omega)}^2+\norm{f}_{L^q(\Omega)}^2\right)\left(\norm{ \eta_h}_{L^{2^\star}(\supp u_h)}^2+\norm{\nabla \eta_h}_{L^2(\supp u_h)}^2\right)\abs{\supp\eta_h\cap\support\, u_h}^{\frac{2\delta}{n}}.
	\end{aligned}
	\end{equation}
\end{theorem}

\begin{proof}
	We write $\rho_h=\Pih\left(\eta_h^2\right)$ and test equation \eqref{eq:DiscreteInhomSubsol} against $\phi_h=\Pih\left(\rho_h u_h\right)$ to get
	\begin{equation}\label{eq:SubsolTested1}
		L:=\int_\Omega A\nabla u_h\cdot \nabla \Pih\left(\rho_h u_h\right)\dx\leq \int_\Omega f\Pih\left(\rho_h u_h\right)\dx+\int_\Omega F\cdot \nabla \Pih\left(\rho_h u_h\right)\dx.
	\end{equation}
	First we will consider the left-hand side of equation \eqref{eq:SubsolTested1}:
	\begin{equation}\label{eq:SubsolLHS}
		\begin{aligned}
			L&=\int_\Omega A\nabla u_h\cdot \nabla \Pih\left(\rho_h u_h\right)\dx\\
			&=\int_\Omega A\nabla u_h\cdot \nabla\left(\rho_h u_h\right)\dx-\int_\Omega A\nabla u_h\cdot \nabla \left(\rho_h u_h-\Pih\left(\rho_h u_h\right)\right)\dx\\
			&=:L_I-L_{II}.
		\end{aligned}
	\end{equation}
	Next we decompose $L_I$ in equation \eqref{eq:SubsolLHS} once more to get
	\begin{equation}\label{eq:SubsolLHS1}
		L_I=\int_\Omega \rho_h A\nabla u_h\cdot \nabla u_h \dx + \int_\Omega (A \nabla u_h \cdot\nabla \rho_h)\, u_h \dx =: L_{I_1} + L_{I_2}.
	\end{equation}
	For $L_{I_1}$ we use that $A$ is uniformly elliptic (cf. inequality \eqref{eq:uniformlyElliptic}) to get
	\begin{equation}\label{eq:SubsolLHS2}
		\int_\Omega \rho_h A\nabla u_h\cdot \nabla u_h\dx\gtrsim \int_\Omega \rho_h \abs{\nabla u_h}^2\dx=:\tilde{L}_{I_1}.
	\end{equation}	
	For $R_{I_2}$ in equation \eqref{eq:SubsolLHS} we use that $A\in L^{\infty}(\Omega;\setR^{n\times n})$ to get
	\begin{equation}\label{eq:SubsolLHS3}
		\begin{aligned}
			\bigabs{L_{I_2}}&\leq\norm{A}_{L^\infty(\Omega)}\int_\Omega \abs{\nabla u_h} \abs{\nabla \rho_h} \abs{u_h}\dx\\
			&\lesssim \sum_{T\in\mathcal{T}_h} \abs{T} \max_T \abs{\nabla u_h} \max_T \abs{\nabla\rho_h} \max_T \abs{u_h} =: R_{I}.
		\end{aligned}
	\end{equation}
	We will now estimate the term $L_{II}$ in inequality \eqref{eq:SubsolLHS}.
	By Lemma~\ref{lem:prod-commutator} (cf. inequality \eqref{eq:prod-commutator2}), $A\in L^\infty(\Omega;\setR^{n\times n})$ and Remark \ref{rem:bound1} we get
	\begin{equation}\label{eq:SubsolLHS4}	
		\begin{aligned}
			\bigabs{L_{II}}&\leq \sum_{T\in\mathcal{T}_h} |T|\max_T  \abs{\nabla u_h} \,\max_T\bigabs{\nabla (\rho_h u_h-\Pih (\rho_h u_h))} \\
			&\lesssim \sum_{T\in\mathcal{T}_h} \abs{T} \max_T \abs{\nabla u_h} \max_T \abs{\nabla\rho_h} \max_T \abs{u_h}\,\dx = R_{I}.
		\end{aligned}
	\end{equation}
  Substituting inequalities \eqref{eq:SubsolLHS2} and \eqref{eq:SubsolLHS3} into equation \eqref{eq:SubsolLHS1}, and inserting equation \eqref{eq:SubsolLHS1} and inequality \eqref{eq:SubsolLHS3} into \eqref{eq:SubsolLHS}, and then equation \eqref{eq:SubsolLHS} back into \eqref{eq:SubsolTested1} yields
	\begin{equation}\label{eq:SubsolTested2-1}
		\begin{aligned}
			\tilde{L}_{I_1}&= \int_\Omega \rho_h \abs{\nabla u_h}^2\dx\lesssim \sum_{T\in\mathcal{T}_h} \abs{T} \max_T \abs{\nabla u_h} \max_T \abs{\nabla\rho_h} \max_T \abs{u_h}\\
			&\quad +\int_\Omega f\Pih\left(\rho_h u_h\right)\dx+\int_\Omega F\cdot \nabla \Pih\left(\rho_h u_h\right)\dx=:R_{I}+R_{II}+R_{III}.
		\end{aligned}
	\end{equation}
	Next, note that for any $T\in\mathcal{T}_h$ we have
  \begin{equation}\label{eq:NablaRhoH}
    \max_T \abs{\nabla \rho_h} = \max_T \bigabs{\nabla \Pih (\eta_h^2)} \lesssim \max_T \abs{\nabla (\eta_h^2)} \lesssim \max_T \abs{\eta_h} \max_T \abs{\nabla \eta_h}.
  \end{equation}
	This allows us to estimate $R_{I}$ from inequality \eqref{eq:SubsolTested2-1}:
  \begin{equation*}
    R_{I}  \lesssim \sum_{T\in\mathcal{T}_h} \abs{T} \max_T \abs{\nabla u_h} \max_T \abs{\eta_h} \max_T \abs{\nabla \eta_h} \max_T \abs{u_h}.
  \end{equation*}
	Now we use Lemma \ref{lem:prod-rearrange} to get
	\begin{equation*}
		R_{I} \lesssim  \sum_T \abs{T} \dashint_T \abs{\nabla u_h} \abs{\eta_h}\, \abs{\nabla \eta_h}\abs{u_h}\,\dx.
	\end{equation*}
By applying Young's inequality this yields
	\begin{equation*}
		R_{I}\leq \sum_{T\in\mathcal{T}_h} \abs{T} \bigg( \epsilon \dashint_T \abs{\nabla u_h}^2\abs{\eta_h}^2\,\dx + C_\epsilon \dashint_T \abs{\nabla \eta_h}^2\abs{u_h}^2\,\dx \bigg)
	\end{equation*}
	for any $\epsilon>0$. Finally, we note that we know from Lemma \ref{lem:Pi-etah2} that $\eta_h^2\leq\rho_h$ and get
	\begin{equation}\label{eq:SubsolRHS1}
		R_{I}\leq  \epsilon \tilde{L}_{I_1} + C_\epsilon \int_\Omega \abs{\nabla \eta_h}^2 \abs{u_h}^2\,\dx
	\end{equation}
	after performing the sums. We can now insert inequality \eqref{eq:SubsolRHS1} into \eqref{eq:SubsolTested2-1} and absorb $\epsilon \tilde{L}_{I_1}$ into the left-hand side to get
	\begin{equation}\label{eq:SubsolTested2}
		\begin{aligned}
			\tilde{L}_{I_1}&= \int_\Omega \rho_h \abs{\nabla u_h}^2\dx\lesssim \int_\Omega \abs{\nabla \eta_h}^2 \abs{u_h}^2\,\dx\\
			&\quad +\int_\Omega f\Pih\left(\rho_h u_h\right)\dx+\int_\Omega F\cdot \nabla \Pih\left(\rho_h u_h\right)\dx=:\tilde{R}_{I}+R_{II}+R_{III}.
		\end{aligned}
	\end{equation}
	
	We will now consider $R_{II}$ in inequality \eqref{eq:SubsolTested2}. Let $\mathcal{S}:=\{T\in\mathcal{T}_h:T\subset \support\, \eta_h\cap\supp u_h\}$; we then have that
	\begin{equation*}
		\bigabs{R_{II}}=\left|\int_\Omega f\Pih\left(\rho_h u_h\right)\dx\right|\leq \sum_{T\in\mathcal{S}}\abs{T}\left(\dashint_T \abs{f}\dx\right)\max_{ T}\left|u_h\rho_h\right|.
	\end{equation*}
	We can now use inequality \eqref{eq:PiL-stable_1} to deduce that
	\begin{equation*}
		\bigabs{R_{II}}\leq \sum_{T\in\mathcal{S}}\abs{T}\left(\dashint_T \abs{f}\dx\right)\max_{ T}\abs{u_h}\left(\max_{ T}\abs{\eta_h}\right)^2.
	\end{equation*}
	Using Lemma \ref{lem:prod-rearrange}, we get
	\begin{equation*}
		\begin{aligned}
			\bigabs{R_{II}}&\lesssim \sum_{T\in\mathcal{S}}\abs{T}\left(\dashint_T \abs{f}\dx\right)\left(\dashint_T \abs{u_h\eta_h}\dx\right)\left(\dashint_T \abs{\eta_h}\dx\right)\\
			&=\sum_{T\in\mathcal{S}}\abs{T}^{-2}\norm{f}_{L^1(T)}\norm{u_h\eta_h}_{L^1(T)}\norm{\eta_h}_{L^1(T)}.
		\end{aligned}
	\end{equation*}
	Recall that $f\in L^q(\Omega)$ with $\frac 1q= \frac 2n - \frac \delta n$ and $\delta >0$. H\"older's inequality yields
	\begin{align*}
		\bigabs{R_{II}}&\lesssim\sum_{T\in\mathcal{S}}\abs{T}^{-2}\norm{f}_{L^q(T)}\abs{T}^{1-\frac 2n + \frac \delta n}\norm{u_h\eta_h}_{L^{2^\star}(T)}\abs{T}^{\frac 12 + \frac 1n} \norm{\eta_h}_{L^{2^\star}(T)}\abs{T}^{\frac 12 + \frac 1n}\\
		&=\sum_{T\in\mathcal{S}}\abs{T}^{\frac\delta n}\norm{f}_{L^q(T)}\norm{u_h\eta_h}_{L^{2^\star}(T)}\norm{\eta_h}_{L^{2^\star}(T)}.
	\end{align*}
	Now note that $\frac 1q+\frac{1}{2^{\star}}+\frac{1}{2^{\star}}+\frac{\delta}n=1$. We can use H\"older's inequality for sums to get
	\begin{equation}\label{eq:SubsolRHS2}
		\begin{aligned}
			\bigabs{R_{II}}&\lesssim \left(\sum_{T\in\mathcal{S}} \abs{T}\right)^{\frac \delta n}\,\left(\sum_{T\in\mathcal{S}} \norm{f}_{L^q(T)}^q\right)^{\frac 1q}
			\,\left(\sum_{T\in\mathcal{S}} \norm{u_h\eta_h}_{L^{2^\star}(T)}^{2^{\star}}\right)^{\frac{1}{2^{\star}}}\,\left(\sum_{T\in\mathcal{S}} \norm{\eta_h}_{L^{2^{\star}}(T)}^{2^{\star}}\right)^{\frac{1}{2^{\star}}}.
		\end{aligned}
	\end{equation}
	As the interiors of the elements $T$ involved in the sums above are disjoint, we know that
	\begin{equation}\label{eq:NormSum}
		\sum_{T\in\mathcal{S}} \norm{g}_{L^r(T)}^r=\norm{g}_{L^r\left(\bigcup_{T\in\mathcal{S}} T\right)}^r
		\end{equation}
		for any $1\leq r <\infty$ and any $g\in L^r\left(\bigcup_{T\in\mathcal{S}} T\right)$. Using this in inequality \eqref{eq:SubsolRHS2} and noting that $\mathcal{S} \subset \supp u_h$ gives
	\begin{equation*}
		\bigabs{R_{II}}\lesssim \bigabs{\supp \eta_h\cap \supp u_h}^{\frac \delta n} \norm{f}_{L^q(\Omega)} \norm{u_h\eta_h}_{L^{2^{\star}}(\Omega)}\norm{\eta_h}_{L^{2^{\star}}(\supp u_h)}.
	\end{equation*}
	Young's inequality and the Sobolev embedding theorem then yield	
	\begin{equation}\label{eq:SubsolRHS3}
		\begin{aligned}
			\bigabs{R_{II}}\lesssim\epsilon\norm{\nabla\left(u_h\eta_h\right)}_{L^{2}(\Omega)}^2+ C_{\epsilon} \bigabs{\supp \eta_h\cap \supp u_h}^{\frac {2\delta} {n}} \norm{f}_{L^q(\Omega)}^2 \norm{\eta_h}_{L^{2^\star}(\supp u_h)}^2.
		\end{aligned}
	\end{equation}
	Note that
	\begin{equation}\label{eq:NormNablaUHEtaH}
		\begin{aligned}
			\norm{\nabla\left(u_h\eta_h\right)}_{L^{2}(\Omega)}^2&\lesssim \norm{u_h\nabla\eta_h}_{L^{2}(\Omega)}^2+\norm{\eta_h\nabla u_h}_{L^{2}(\Omega)}^2\\
			&\leq \int_\Omega \rho_h \abs{\nabla u_h}^2\dx+ \int_\Omega \abs{u_h}^2\abs{\nabla \eta_h}^2\dx,
		\end{aligned}
	\end{equation}
	where we used Lemma \ref{lem:Pi-etah2} in the second step.
	Finally, substituting this into inequality \eqref{eq:SubsolRHS3} gives
	\begin{equation}\label{eq:SubsolRHS4}
	\begin{aligned}
		\hspace{-2.5mm}\bigabs{R_{II}}&\lesssim  \epsilon \int_\Omega \rho_h \abs{\nabla u_h}^2\dx + \epsilon \int_\Omega \abs{u_h}^2\abs{\nabla \eta_h}^2\dx \\
		&\quad + C_{\epsilon} \bigabs{\supp \eta_h\cap \supp u_h}^{\frac {2\delta} {n}} \norm{f}_{L^q(\Omega)}^2 \norm{\eta_h}_{L^{2^\star}(\supp u_h)}^2.
	\end{aligned}
	\end{equation}
	
	We will now focus on $R_{III}$ in inequality \eqref{eq:SubsolTested2}.
	Splitting it up into a sum over simplices and using inequality \eqref{eq:PiL-stable_1} yields
	\begin{equation}\label{eq:SubsolRHS5}
		\begin{aligned}
			\bigabs{R_{III}}&=\left|\int_\Omega F\cdot\nabla\big(\Pih\left(u_h\rho_h\right)\big)\dx\right|\\
			&\leq  \sum_{T\in\mathcal{S}} \abs{T} \dashint_T \abs{F}\dx\, \max_{ T}\left|\nabla \Pih \left(u_h \rho_h\right)\right|\\
			&\lesssim \sum_{T\in\mathcal{S}} \abs{T} \dashint_T \abs{F}\dx\,  \max_{ T}\left|\nabla \left(u_h \rho_h\right)\right|\\
			&\leq  \sum_{T\in\mathcal{S}} \abs{T} \dashint_T \abs{F}\dx \, \max_{ T}\left|\left(\nabla u_h\right) \rho_h\right|+ \sum_{T\in\mathcal{S}} \abs{T} \dashint_T \abs{F}\dx \, \max_{ T}\left|\left(\nabla \rho_h\right) u_h\right|\\
			&=: R_{III_1}+R_{III_2}.
		\end{aligned}
	\end{equation}
	We rearrange $R_{III_1}$ from inequality \eqref{eq:SubsolRHS5} and use Lemma \ref{lem:prod-rearrange}:
	\begin{align*}
		R_{III_1}&\lesssim\sum_{T\in\mathcal{S}} \abs{T} \dashint_T \abs{F}\dx\, \dashint_T \abs{(\nabla u_h)\eta_h}\dx \, \dashint_T \abs{\eta_h}\dx\\
		&=\sum_{T\in\mathcal{S}} \abs{T}^{-2} \int_T \abs{F}\dx \int_T \abs{(\nabla u_h)\eta_h}\dx \int_T \abs{\eta_h}\dx.
	\end{align*}
	Recall that $F\in L^p(\Omega)$ with $\frac 1p = \frac 1n - \frac \delta n$. H\"older's inequality gives
	\begin{align*}
		R_{III_1}&\lesssim \sum_{T\in\mathcal{S}} \abs{T}^{-2} \norm{F}_{L^p(T)}\abs{T}^{1-\frac 1 n +\frac \delta n}\norm{(\nabla u_h)\eta_h}_{L^2(T)}\abs{T}^{\frac 12} \norm{\eta_h}_{L^{2^{\star}}(T)}\abs{T}^{\frac 12 + 	\frac 1n}\\
		&=\sum_{T\in\mathcal{S}} \abs{T}^{\frac\delta n} \norm{F}_{L^p(T)}\norm{(\nabla u_h)\eta_h}_{L^2(T)} \norm{\eta_h}_{L^{2^{\star}}(T)}.
	\end{align*}
	Now H\"older's inequality for sums and the identity \eqref{eq:NormSum} give
	\begin{equation*}
		R_{III_1}\lesssim \abs{\supp \eta_h \cap \supp u_h}^{\frac{\delta}{n}}\norm{F}_{L^p(\Omega)}\norm{(\nabla u_h)\eta_h}_{L^2(\Omega)}\norm{\eta_h}_{L^{2^{\star}}(\supp u_h)}.
	\end{equation*}
	Finally, we can use $\eta_h^2\leq\rho_h$ and Young's inequality to get
	\begin{equation}\label{eq:SubsolRHS6}
		\begin{aligned}
			R_{III_1}\lesssim \epsilon \int_\Omega \abs{\nabla u_h}^2 \rho_h\dx + C_{\epsilon}\norm{F}_{L^p(\Omega)}^2\abs{\supp\eta_h\cap \supp u_h}^{\frac {2\delta} n} \norm{\eta_h}_{L^{2^\star}(\supp u_h)}^2.
		\end{aligned}
	\end{equation}
	
	Next we consider $R_{III_2}$ from inequality \eqref{eq:SubsolRHS5}. Recall $\max_{ T} \abs{\nabla \rho_h}\lesssim \max_{ T}\abs{\eta_h}\max_{ T}\abs{\nabla\eta_h}$ as in inequality \eqref{eq:NablaRhoH}. We get
	\begin{equation*}
		R_{III_2}\lesssim\sum_{T\in\mathcal{S}} \abs{T}\dashint_T\abs{F}\dx \, \max_{ T}\abs{u_h} \max_{ T}\abs{\eta_h}\max_{ T}\abs{\nabla\eta_h}.
	\end{equation*}
	With Lemma \ref{lem:prod-rearrange}, this becomes
	\begin{equation*}
		R_{III_2}\lesssim \sum_{T\in\mathcal{S}} \abs{T}^{-2} \int_T\abs{F}\dx \int_T\abs{u_h\eta_h}\dx \int_T\abs{\nabla \eta_h}\dx.
	\end{equation*}
	Completely analogously to $R_{III_1}$, we get
	\begin{equation*}
		R_{III_2}\leq \epsilon \norm{\nabla(u_h\eta_h)}_{L^2(\Omega)}^2+C_\epsilon\abs{\supp\eta_h\cap \supp u_h}^{\frac{2\delta}{n}}\norm{F}_{L^p(\Omega)}^2\norm{\nabla\eta_h}_{L^2(\supp u_h)}^2.
	\end{equation*}
	We use inequality \eqref{eq:NormNablaUHEtaH} to deduce that
	\begin{equation}\label{eq:SubsolRHS7}
		\begin{aligned}
		R_{III_2}&\lesssim \epsilon \int_\Omega \abs{\nabla u_h}^2 \rho_h\dx+\epsilon \int_\Omega \abs{u_h}^2\abs{\nabla \eta_h}^2\dx\\
		&\quad+C_\epsilon \abs{\supp\eta_h\cap \supp u_h}^{\frac{2\delta}{n}}\norm{F}_{L^p(\Omega)}^2\norm{\nabla\eta_h}_{L^2(\supp u_h)}^2.
		\end{aligned}
	\end{equation}
	
	Inserting inequalities \eqref{eq:SubsolRHS6} and \eqref{eq:SubsolRHS7} into inequality \eqref{eq:SubsolRHS5} and using Poincar\'e's inequality on the last factor in the second summand in \eqref{eq:SubsolRHS6} gives, after renaming constants,
	\begin{equation}\label{eq:SubsolRHS8}
		\begin{aligned}
		\bigabs{R_{III}}&\lesssim \epsilon \int_\Omega \abs{\nabla u_h}^2\rho_h\dx+\epsilon \int_\Omega \abs{u_h}^2\abs{\nabla\eta_h}^2\dx\\
		&\quad+C_\epsilon \abs{\supp\eta_h\cap \supp u_h}^{\frac{2\delta}{n}}\norm{F}_{L^p(\Omega)}^2\norm{\nabla\eta_h}_{L^2(\supp u_h)}^2.
		\end{aligned}
	\end{equation}
	We can finally substitute inequalities \eqref{eq:SubsolRHS4} and \eqref{eq:SubsolRHS8} into inequality \eqref{eq:SubsolTested2} and absorb the $$\epsilon\int_\Omega\abs{\nabla u_h}^2\rho_h\dx$$ term into the left-hand side. This yields
	\begin{equation}
		\begin{aligned}
		&\int_\Omega \rho_h \abs{\nabla u_h}^2\dx\lesssim \int_\Omega \abs{\nabla \eta_h}^2 \abs{u_h}^2\,\dx\\
		&\quad+\abs{\supp\eta_h\cap \supp u_h}^{\frac{2\delta}{n}}\left(\norm{\eta_h}_{L^{2^\star}(\supp u_h)}^2+\norm{\nabla\eta_h}_{L^2(\supp u_h)}^2\right)\left(\norm{F}_{L^p(\Omega)}^2+\norm{f}_{L^q(\Omega)}^2\right).
			\end{aligned}
		\end{equation}
		Using $\eta_h^2\leq\rho_h$ from Lemma \ref{lem:Pi-etah2} on the left-hand side proves inequality \eqref{eq:DiscreteInhomSubsol} and the theorem.
\end{proof}

\smallskip

This discrete Caccioppoli inequality has the following direct corollary.
\begin{corollary}\label{cor:CaccioppoliInhomCor}
	Let $A\in L^{\infty}(\Omega;\setR^{n\times n})$ be a uniformly elliptic matrix-valued function and let $F\in L^p(\Omega;\setR^n)$ and $f\in L^q(\Omega)$
	be given functions with $\frac 1p = \frac 1 n-\frac {\delta}{n}$ and $\frac 1q=\frac 2n -\frac{\delta}{n}$, where $\delta>0$. Furthermore, let $\mathcal{T}_h$ be a shape-regular triangulation of the polyhedral domain $\Omega$ and $V_h$ the corresponding finite element space.
	Let $u_h\in V_h$ be a discrete subsolution to $-\divergence(A\nabla u)=f-\divergence F$ in the sense of inequality \eqref{eq:DiscreteInhomSubsol}, let $\eta\in C_0^{\infty}(\Omega)$ be nonnegative and define $\eta_h:=\Pi_h\eta$. Then, we have
	\begin{equation}\label{eq:DiscreteCaccioppoliInhomCor}
	\begin{aligned}
		&\int_\Omega \abs{\nabla\left(u_h \eta_h\right)}^2\dx\lesssim \int_\Omega \abs{u_h}^2\abs{\nabla \eta_h}^2\dx\\
		&\quad+\left(\norm{F}_{L^p}^2+\norm{f}_{L^q}^2\right)\left(\norm{\nabla \eta_h}_{L^2(\support\, u_h)}^2+\norm{\eta_h}^2_{L^{2^\star}(\support\, u_h)}\right)\abs{\supp\eta_h\cap \supp u_h}^{\frac{2\delta}{n}}.
		\end{aligned}
	\end{equation}
\end{corollary}
\begin{proof}
	We use that
	\begin{equation*}
		\int_\Omega \abs{\nabla\left(u_h \eta_h\right)}^2\dx \lesssim \int_\Omega \abs{\nabla u_h}^2\abs{\eta_h}^2\dx+\int_\Omega \abs{u_h}^2\abs{\nabla \eta_h}^2\dx
	\end{equation*}
	and apply Theorem \ref{thm:CaccioppoliInhom}.
\end{proof}

We will now introduce the \textit{nodal maximum} of two functions in $V_h$. This will also lead to a suitable notion of the \textit{positive part} of a continuous piecewise affine function.
\begin{definition}\label{def:NodalMax}
	Let $\mathcal{T}_h$ be a triangulation of the polyhedral domain $\Omega$ and $V_h$ the corresponding finite element space of continuous piecewise affine functions. We define the nodal maximum of two functions $u_h\in V_h$ and $v_h\in V_h$ as
	\begin{equation*}
		u_h\curlyvee v_h:=\sum_i \left(u_h(x_i)\vee v_h(x_i)\right)\psi_i.
	\end{equation*}
	Furthermore, we define the nodal positive part of a function $u_h\in V_h$ as
	\begin{equation*}
		(u_h)_+:=u_h\curlyvee 0=\sum_i \left(u_h(x_i)\right)_+\psi_i.
	\end{equation*}
\end{definition}
We also need the following technical assumption of $F$, which will be referred to as assumption ($\star$).

\begin{definition}\label{def:AssumptionStar}
	Let $F\in L^p(\Omega;\setR^n)$. We will say that $F$ satisfies assumption ($\star$) if there exists a $G\in L^p(\Omega;\setR^n)$ such that
	\begin{equation}\label{eq:AssumptionStar}
		\int_\Omega G\cdot \nabla \phi\dx\geq\Bigabs{\int_\Omega F\cdot \nabla \phi\dx}\geq 0\\
	\end{equation}
	for any nonnegative $\phi\in C^\infty_0(\Omega)$. By density, this inequality extends to all nonnegative $\phi\in W^{1,p'}_0(\Omega)$.
\end{definition}

\begin{remark}
	The condition \eqref{eq:AssumptionStar} is restrictive, but still admits several nontrivial cases. First and foremost, it includes every $F\in L^p(\Omega;\setR^n)$ where $\divergence F$ has a fixed sign as a distribution, with $G:=-(\mbox{sgn } \divergence F)\, F$. It also includes cases such as the following, where $\divergence F$ changes sign in the sense of distributions:
suppose that $n=2$, $\Omega=(-2,2)^2$ and, for $x = (x_1,x_2)^{\rm T} \in \Omega$, let
	\begin{equation*}
		F(x):= \left\{
		\begin{array}{rl}
			-e_1&\text{ for } \abs{x_1}\geq 1,\\
			~e_1&\text{ for } \abs{x_1}< 1,
		\end{array}
               \right.
	\end{equation*}
	where $e_1=(1,0)^{\rm T}$. Obviously $F\in L^\infty(\Omega;\setR^2)$ and $\divergence F=2\delta_{-1}\otimes \chi_{(-2,2)}-2\delta_{1}\otimes \chi_{(-2,2)}$, where $\delta_t$ denotes the Dirac-distribution concentrated on the point $t\in\setR$. Consider the function $G$ defined by
	\begin{equation*}
		G(x)= \left\{
		\begin{array}{rl}
		3e_1&\text{ for } x_1\leq -1,\\
			e_1&\text{ for } -1< x_1<1,\\
			-e_1&\text{ for } x_1\geq 1,
		\end{array}
              \right.	
        \end{equation*}
	which gives $\divergence G=-2\delta_{-1}\otimes \chi_{(-2,2)}-2\delta_{1}\otimes \chi_{(-2,2)}$, whereby $\divergence (G \pm F) = -4\delta_{\pm 1}\otimes \chi_{(-2,2)}\leq 0$ in the sense of distributions, and therefore assumption ($\star$) is satisfied in this case. In general, ($\star$) is fulfilled if there is an $L^p$-function $G$ with $-\divergence G \geq \abs{\divergence F}$ in the sense of distributions.
\end{remark}

This leads to the following theorem, which connects the notions of \textit{positive part} and \textit{subsolution}.
\begin{theorem}\label{thm:SubsolNodalMax}
	Let $\mathcal{T}_h$ be an $A$-nonobtuse triangulation of the polyhedral domain $\Omega$. Let $u_h$ be a discrete subsolution to
	$-\divergence(A\nabla u_1) = f_1-\divergence  F_1$ and suppose that $v_h$ is a discrete subsolution to
	$-\divergence(A\nabla u_2) = f_2 - \divergence  F_2$ for $L^p$-functions $F_1$ and $F_2$ and $L^q$-functions $f_1$ and $f_2$. Suppose further
	that $F_1$ satisfies assumption ($\star$) from Definition \ref{def:AssumptionStar} with dominating function $G_1$ and $F_2$ satisfies
	assumption ($\star$) with dominating function $G_2$. Then, $u_h \curlyvee v_h$ is a discrete subsolution to
	$$-\divergence(a (u_1\curlyvee u_2)) = f_1\vee f_2-\divergence  (G_1+G_2).$$
\end{theorem}

\begin{proof}
	As a first step, we will show that
	\begin{equation}\label{eq:NodalMaxBasis}
		\int_\Omega A\nabla (u_h \curlyvee v_h)\cdot\nabla\psi_j\dx\leq \int_\Omega (f_1 \vee f_2)\psi_j\dx+\left(\int_\Omega F_1\cdot \nabla\psi_j\dx\right)\vee \left(\int_\Omega F_2\cdot \nabla \psi_j\dx\right)
	\end{equation}
	for all $j$.
	So we fix an arbitrary $j$ and assume w.l.o.g. that $u_h(x_j)\geq v_h(x_j)$ and therefore $u_h(x_j)\vee v_h(x_j)=u_h(x_j)$. As $\mathcal{T}_h$ is $A$-nonobtuse, we have $\int_\Omega A\nabla\psi_i\cdot\nabla\psi_j\dx\leq 0$ for all $i\neq j$ from equation \eqref{eq:aNonObtuse} and find
	\begin{equation}\label{eq:NodalMax1}
		\begin{aligned}
			&\int_\Omega A \nabla u_h\cdot\nabla\psi_j\dx = \sum_i u_h(x_i) \int_\Omega A\nabla\psi_i\cdot\nabla\psi_j\dx\\
			&= u_h(x_j) \int_\Omega A\nabla\psi_j\cdot\nabla\psi_j\dx+\sum_{i\neq j} u_h(x_i)\int_\Omega A\nabla\psi_i\cdot\nabla\psi_j\dx\\
			&\geq (u_h(x_j)\vee v_h(x_j)) \int_\Omega A\nabla\psi_j\cdot\nabla\psi_j\dx+\sum_{i\neq j} (u_h(x_i)\vee v_h(x_i))\int_\Omega A\nabla\psi_i\cdot\nabla\psi_j\dx\\
			&=\int_\Omega A\nabla (u_h\curlyvee v_h)\cdot\nabla\psi_j\dx.
		\end{aligned}
	\end{equation}
	On the other hand, we have
	\begin{equation}\label{eq:NodalMax2}
		\int_\Omega (f_1 \vee f_2)\psi_j\dx\geq \int_\Omega f_1\psi_j\dx
	\end{equation}
	and
	\begin{equation}\label{eq:NodalMax3}
		\left(\int_\Omega F_1\cdot \nabla\psi_j\dx\right)\vee \left(\int_\Omega F_2\cdot \nabla \psi_j\dx\right)\geq \int_\Omega F_1\cdot \nabla\psi_j\dx.
	\end{equation}
	Combining inequalities \eqref{eq:NodalMax1}, \eqref{eq:NodalMax2} and \eqref{eq:NodalMax3} with
	\begin{equation*}
		\int_\Omega A\nabla u_h  \cdot \nabla\psi_j \dx\leq\int_\Omega F_1\cdot \nabla \psi_j+ f_1\psi_j\dx,
	\end{equation*}
	which follows from the fact that $u_h$ is a discrete subsolution to $-\divergence(A\nabla u_1)=f_1-\divergence  F_1$, yields inequality \eqref{eq:NodalMaxBasis}.
	We will now use the fact that $F_1$ and $F_2$ satisfy assumption ($\star$) and therefore inequality \eqref{eq:AssumptionStar} with dominating functions $G_1$ and $G_2$ respectively. We thus have that
	\begin{equation}\label{eq:NodalMax4}
		\begin{aligned}
			\left(\int_\Omega F_1\cdot \nabla\psi_j\dx\right)\vee \left(\int_\Omega F_2\cdot \nabla \psi_j\dx\right)&\leq \left(\int_\Omega G_1\cdot \nabla\psi_j\dx\right)\vee \left(\int_\Omega G_2\cdot \nabla \psi_j\dx\right)\\
			&\leq\int_\Omega \left(G_1+ G_2\right)\cdot \nabla \psi_j\dx.
			\end{aligned}
		\end{equation}
		Together with inequality \eqref{eq:NodalMaxBasis} this gives
		\begin{equation}\label{eq:NodalMaxBasis2}
			\int_\Omega A\nabla (u_h \curlyvee v_h)\cdot\nabla\psi_j\dx\leq \int_\Omega (f_1 \vee f_2)\psi_j\dx+\int_\Omega \left(G_1+ G_2\right)\cdot \nabla \psi_j\dx.
		\end{equation}
		
		In general, for a nonnegative $\phi_h\in  V_{h,0}$ we write $\phi_h=\sum_j \phi_h(x_j)\psi_j$ and use the fact that both sides of inequality \eqref{eq:NodalMaxBasis2} are linear in $\psi_j$:
		\begin{align*}
			\int_\Omega A\nabla (u_h \curlyvee v_h)\cdot\nabla\phi_h\dx &=\sum_j\phi_h(x_j)\int_\Omega A\nabla (u_h \curlyvee v_h)\cdot\nabla\psi_j\dx\\
			&\leq \sum_j \phi_h(x_j)\int_\Omega (f_1 \vee f_2)\psi_j\dx+\int_\Omega \left(G_1+ G_2\right)\cdot \nabla \psi_j\dx\\
			&=\int_\Omega (f_1 \vee f_2)\phi_h\dx+\int_\Omega \left(G_1+ G_2\right)\cdot \nabla \phi_h\dx.
		\end{align*}
		This implies that $u_h\vee v_h$ is a discrete subsolution to $-\divergence  (A\nabla (u_1\curlyvee u_2))=f_1\vee f_2 -\divergence(G_1+G_2)$ and proves the theorem.
\end{proof}

By setting $f_1 = f$, $F_1 = F$, $f_2=0$ and $F_2=0$ we deduce the following immediate corollary.
\begin{corollary}\label{cor:PositivePartSubsol}
    Let $\mathcal{T}_h$ be an $A$-nonobtuse triangulation of the polyhedral domain $\Omega$. Let $u_h$ be a discrete subsolution
    to $-\divergence(A \nabla u) = f-\divergence  F$ for $F\in L^p(\Omega; \setR^n)$ and $f\in L^q(\Omega)$, with $p$ and $q$ as previously. Furthermore, suppose that $F$ satisfies assumption
    ($\star$) from Definition \ref{def:AssumptionStar} with dominating function $G$ and let $c\in\setR$ be a constant. Then $(u_h-c)_+$ is a discrete subsolution to $-\divergence  (A\nabla w)= \abs{f}-\divergence  G$.
\end{corollary}
\begin{proof}
	We have
	\begin{align*}
		\int_\Omega A\nabla (u_h-c)_+  \cdot \nabla\phi_h \dx&\leq\int_\Omega F\cdot \nabla \phi_h+ f\phi_h\dx,\\
		\int_\Omega A\nabla (-u_h-c)_+  \cdot \nabla\phi_h \dx&\leq -\int_\Omega F\cdot \nabla \phi_h -f\phi_h\dx,
	\end{align*}
	for any nonnegative $\phi_h\in V_{h,0}$. Recall that we defined the nodal positive part as $(u_h-c)_+=(u_h-c)\vee 0$
	in Definition \ref{def:NodalMax}. We have $\pm f \vee 0 \leq \abs{f}$ and $\pm \int_\Omega F\cdot\nabla \phi\dx\leq \int_\Omega G\cdot\nabla \phi\dx$. Because $v_h\equiv 0$ solves $-\divergence  (A\nabla v_h)=0$, we can apply Theorem \ref{thm:SubsolNodalMax} with $f_1=f$, $F_1=F$, $f_2\equiv 0$ and $F_2\equiv 0$ to deduce the assertion of the corollary.
\end{proof}

Finally, Corollary \ref{cor:PositivePartSubsol} yields the desired discrete Caccioppoli-type inequalities for the truncated function $(u_h-c)_+$ with $c \in \setR$, which we now state and prove.

\begin{corollary}\label{cor:SubsolNodalMax2}
    Let $\mathcal{T}_h$ be an $A$-nonobtuse triangulation of the polyhedral domain $\Omega$. Let $u_h$ be a discrete subsolution to
    $-\divergence(A \nabla u) = f-\divergence  F$ for $F\in L^p(\Omega;\setR^n)$ and $f\in L^q(\Omega)$, with $p$ and $q$ as previously. Furthermore, suppose that $F$ satisfies assumption ($\star$) from Definition \ref{def:AssumptionStar} with dominating function $G$. Let $\eta_h$ be as in Theorem \ref{thm:CaccioppoliInhom}. For any $c\in\setR$ we then have the following bounds:
    \begin{equation}\label{eq:CaccioppoliTruncated1}
        \begin{aligned}
            &\int_\Omega \abs{\nabla (u_h-c)_+\eta_h}^2\dx\lesssim \int_\Omega (u_h-c)_+^2\abs{\nabla\eta_h}^2\dx+\left(\norm{G}_{L^p(\Omega)}^2+\norm{f}_{L^q(\Omega)}^2\right)\\
						&\qquad\cdot\left(\norm{\nabla \eta_h}_{L^2(\support\, (u_h-c)_+)}+\norm{ \eta_h}_{L^{2^\star}(\support\, (u_h-c)_+)}^2\right)\abs{\supp\eta_h\cap\support\, (u_h-c)_+}^{\frac{2\delta}{n}},
        \end{aligned}
    \end{equation}
    and
    \begin{equation}\label{eq:CaccioppoliTruncated2}
        \begin{aligned}
            &\int_\Omega \abs{\nabla \left((u_h-c)_+\eta_h\right)}^2\dx\lesssim \int_\Omega (u_h-c)_+^2\abs{\eta_h}^2\dx+\left(\norm{G}_{L^p(\Omega)}^2+\norm{f}_{L^q(\Omega)}^2\right)\\
						&\qquad\cdot \left(\norm{\nabla \eta_h}_{L^2(\support\, (u_h-c)_+)}+\norm{ \eta_h}_{L^{2^\star}(\support\, (u_h-c)_+)}^2\right)\abs{\supp\eta_h\cap\support\, (u_h-c)_+}^{\frac{2\delta}{n}}.
        \end{aligned}
    \end{equation}
\end{corollary}

\begin{proof}
	Corollary \ref{cor:SubsolNodalMax2} guarantees that $(u_h-c)_+$ is a discrete subsolution to $-\divergence  (A\nabla w)=\abs{f}-\divergence  G$. Furthermore $(u_h-c)_+$ is obviously nonnegative. Thus we can apply Theorem \ref{thm:CaccioppoliInhom} and Corollary \ref{cor:CaccioppoliInhomCor} to deduce the stated claims.
\end{proof}

\section{Interior $C^{\alpha}$-estimates for approximate solutions}
We will now prove the desired uniform a priori $C^\alpha$-bound for sequences of continuous piecewise affine finite element approximations in the interior of $\Omega$. To this end, we will first prove an $L^\infty$-bound, and will then deduce the discrete $C^\alpha$-bound. We emphasize that we have not, so far, assumed any kind of (quasi-)uniformity of the triangulation, nor shall we do so hereafter. Our results therefore apply on graded and adaptively refined triangulations. The main result of this subsection is encapsulated in the following theorem.

\begin{theorem}\label{thm:DiscreteHoelderInteriour}
 Let $\frac 1 p=\frac 1 n- \frac \delta n$ and $\frac 1q=\frac 2n-\frac \delta n$ for some $\delta>0$, let $f\in L^{q}(\Omega)$, suppose that the function $F\in L^p(\Omega;\setR^n)$ satisfies assumption ($\star$) from Definition \ref{def:AssumptionStar} with dominating function $G\in L^p(\Omega;\setR^n)$ and let $A\in L^\infty(\Omega;\setR^{n\times n})$ be a uniformly elliptic matrix-valued function. Furthermore, let $u_h$ be the finite element approximation to the solution $u$ of $-\divergence(A\nabla u)= f-\divergence  F$ in the sense of Definition \ref{def:DiscreteSolution} on a shape-regular, $A$-nonobtuse triangulation $\mathcal{T}_h$ of the polyhedral domain $\Omega$ with respective finite element space $V_h$, i.e.,
 \begin{equation}
  -\int_\Omega A\nabla u_h\cdot \nabla \phi_h\dx=\int_\Omega f\phi_h\dx+\int_\Omega F\cdot \nabla \phi_h\dx
 \end{equation}
 for any function $\phi_h\in V_{h,0}$. Furthermore, let $\kappa$ and $Q$ be the constants from equation \eqref{eq:regularBall} and inclusion \eqref{eq:OmegaPrime2}. Assume that $x_0\in T$ for some $T\in\mathcal{T}_h$ and $B(x_0,4\kappa^{-1}QR')\subset\Omega$ for some $R'\geq h_T$.
 Then, there is an $\alpha>0$, such that, for every ball $B(x_0,R)\subset\Omega$,
 \begin{equation}\label{eq:oscRAlpha}
  \osc_{B(x_0,R)}u_h\lesssim R^\alpha.
 \end{equation}

\end{theorem}
The hypotheses of Theorem  \ref{thm:DiscreteHoelderInteriour}  will be assumed to hold throughout this section.
We will need a few additional technical lemmas concerning the projections of the functions $\phi_k$, which we have defined in equation \eqref{eq:DefPhikCont}. Before stating these, we introduce the following definition.

\begin{definition}\label{def:B_kDiscrete}
	For a constant $\lambda_\infty>0$ that is to be chosen later, a radius $R>0$, $x_0\in\Omega$, such that $B(x_0,2R)\subset\Omega$, and $k=0,1,2...$, let
	\begin{align*}
		\lambda_k:=\left(1-2^{-k}\right)\lambda_\infty\qquad \mbox{and}\qquad
		B_k:=B\left(x_0,\left(1+2^{-k}\right)R\right),
	\end{align*}
	and consider a sequence of nonnegative $C^\infty$-functions $\tilde{\eta}_k$, such that
	\begin{align*}
		\chi_{B_{k+1}}\leq\tilde{\eta}_k\leq\chi_{B_{k}}\qquad \mbox{and}\qquad
		|\nabla\tilde{\eta}_k|\lesssim 2^k R^{-1}.
	\end{align*}
	We then define, for any triangulation $\mathcal{T}_h$,
	\begin{equation*}
		\eta_k:=\Pih\tilde{\eta}_k.
	\end{equation*}
	\end{definition}

\begin{lemma}\label{lem:eta_k}
  Let $B_k$ and $\eta_k$ be defined as in Definition \ref{def:B_kDiscrete}. Then, we have
  \begin{enumerate}
  \item $\max_\Omega |\nabla\eta_k|\lesssim 2^k R^{-1}$; \label{itm:eta_k_a}
  \item If $\max_T|\eta_{k+1}|>0$ for some $T\in\mathcal{T}_h$, then we
    have $\max_T|\eta_k|=1$; \label{itm:eta_k_b}
  \item $\max_T|\nabla\eta_{k+1}|\lesssim
    2^{k+1}R^{-1}\max_T|\eta_k|$; \label{itm:eta_k_c}
  \item If $a$ is a polynomial, we have
    $\dashint_T |a|^2|\nabla\eta_{k+1}|^2\dx\lesssim R^{-2}
    2^{2(k+1)}\dashint_T|a|^2 \eta_k^2$. \label{itm:eta_k_d}
  \end{enumerate}
\end{lemma}
\begin{proof}
  Assertion \ref{itm:eta_k_a} is clear since
  \begin{align*}
    \max_\Omega|\nabla\eta_k|=\max_\Omega|\nabla\left(\Pih
    \tilde{\eta}_k\right)|\leq\max_\Omega|\nabla\tilde{\eta}_k|\lesssim 2^k R^{-1}
  \end{align*}
	by equation \eqref{eq:PiL-stable_2}.
For assertion \ref{itm:eta_k_b}, we note that if $\max_T|\eta_{k+1}|>0$, there is at least one node $x_0$ with $\eta_{k+1}(x_0)>0$ and therefore $x_0\in B_k$ and $\eta_k(x_0)=\tilde{\eta_k}(x_0)=1$.
	Assertion \ref{itm:eta_k_c} is a direct consequence of assertions \ref{itm:eta_k_a} and \ref{itm:eta_k_b}.
	The inequality stated in assertion \ref{itm:eta_k_d} follows by using
  Lemma~\ref{lem:prod-rearrange}:
  \begin{align*}
    \dashint_T |a|^2|\nabla\eta_{k+1}|^2\dx
    &\sim\max_T |a|^2\max_T|\nabla\eta_{k+1}|^2
    \\
    &\lesssim  2^{2(k+1)} R^{-2}\max_T |a|^2 \max_T|\eta_{k}|^2
    \sim
      2^{2(k+1)}R^{-2}\dashint_T|a|^2
      \eta_k^2\dx.
  \end{align*}
That completes the proof of the lemma.
\end{proof}

We will also need a discrete counterpart of the weak-type estimate \eqref{eq:U_kCont3}, which we shall now state.
\begin{lemma}
  \label{lem:weaktype}
  Consider $\eta_k$ as in Definition \ref{def:B_kDiscrete} and let  $A_k:=\{\eta_k^2|(v_h-\lambda_k-c_0)_+|^2>0\}$ for some $c_0\in\setR$ and $v_h\in V_h$. Then, we have
  \begin{equation}\label{eq:weaktypeDiscrete}
   |A_{k+1}|\leq\frac{2^{2k}}{\lambda_\infty^2}\int_\Omega |\eta_k (u_h-\lambda_k-c_0)_+|^2\dx.
  \end{equation}
\end{lemma}
\begin{proof}
  First observe that $A_{k+1}$ is a union of finitely many $T_i\in\mathcal{T}_h$. On every such $T_i$, we have $\max_{T_i}\eta_k=1$ by Lemma \ref{lem:eta_k}, assertion \ref{itm:eta_k_b}, as $\max_{T_i}\eta_{k+1}>0$. For at least one node $x_0$ of $T_i$ we have $v_h(x_0)>\lambda_{k+1}+c_0$, and therefore
	\begin{equation*}
		(v_h(x_0)-\lambda_k-c_0)_+=v_h(x_0)-\lambda_k-c_0>\lambda_{k+1}- \lambda_k=2^{-k}\lambda_\infty.
	\end{equation*}
  Using this and Lemma \ref{lem:prod-rearrange} we get
  \begin{align*}
    |T_i|
    &\leq\frac{2^{2k}}{\lambda_\infty^2}|T_i|
      \max_{T_i}(v_h-\lambda_k-c_0)_+^2\max_{T_i}\eta_k^2
    \\
    &\sim\frac{2^{2k}}{\lambda_\infty^2}
      |T_i|\dashint_{T_i}(v_h-\lambda_k-c_0)_+^2\eta_k^2\dx
 =\frac{2^{2k}}{\lambda_\infty^2}\int_{T_i}(v_h-\lambda_k-c_0)_+^2\eta_k^2\dx.
  \end{align*}
  As the $T_i$ are disjoint and $\text{supp}\, \eta_k\subset A_{k+1}$,
  summing over all $i$ yields the assertion of the lemma.
\end{proof}

This result allows us to prove a discrete version of the $L^\infty$-estimate from Theorem \ref{thm:ContBounded}. We shall again confine ourselves to the case when $n \geq 3$ and write $2^\star:=2n/(n-2)$; the proofs below are easily adjusted when $n=2$.
\begin{theorem}\label{thm:DiscreteBounded}
 Assume that all the assumptions of Theorem \ref{thm:DiscreteHoelderInteriour} are satisfied. Furthermore, let $x_0\in  T$ for some $T\in\mathcal{T}_h$ be a point such that $B(x_0,2R)\subset \Omega$ for some $R\geq h_T$. Then, we have
 \begin{equation}\label{eq:DiscreteBounded}
    \max_{\Omega'(B(x_0,R))}(u_h-c)_+^2\lesssim \dashint_{\Omega(B(x_0,2R))}(u_h-c)_+^2\dx+\left(\norm{G}_{L^p}^2+\norm{f}_{L^q}^2\right)R^{2\delta}
 \end{equation}
 for all $c\in\setR$, with $\Omega'(B(x_0,R))$ as defined in \eqref{eq:OmegaPrime}.
\end{theorem}

\begin{proof}	
	For $B=B(x_0,R)$ we use the notation $2B:=B(x_0,2R)$. Using the notions from Definition \ref{def:B_kDiscrete}, we define the sequence
	\begin{equation*}
		a_k:=\dashint_{\Omega(2B)} (u_h-\lambda_k-c)_+^2\eta_{k}^2\dx\qquad \mbox{for $k=0,1,2,\dots$}.
	\end{equation*}
	We write $A_k:=\{(u_h-\lambda_k-c)_+^2\eta_k^2>0\}$ as in Lemma \ref{lem:weaktype} and use H\"older's inequality to deduce that
	\begin{equation}\label{eq:Bounded1}
		\begin{aligned}
			a_{k+1}&=\dashint_{\Omega(2B)}(u_h-\lambda_{k+1}-c)_+^2\eta_{k+1}^2\dx\\
			&\lesssim \left(\dashint_{\Omega(2B)}\bigabs{(u_h-\lambda_{k+1}-c)_+\eta_{k+1}}^{2^\star}\dx\right)^{\frac{2}{2^\star}}\left(\frac{|A_{k+1}|}{\abs{\Omega(2B)}}\right)^{\frac 2 n}.
		\end{aligned}
	\end{equation}
	We have assumed that $R\geq h_T$. Recall that there is a $Q>2$ such that $\Omega(2B)\subset B(x_0,QR)$ by inequality \eqref{eq:regularBall}. This means that we have $\abs{\Omega(2B)}\sim R^n$. Using this and the Sobolev embedding theorem in inequality \eqref{eq:Bounded1} gives
	\begin{equation}\label{eq:Bounded2}
		a_{k+1}\lesssim R^2\dashint_{\Omega(2B)}\left|\nabla ((u_h-\lambda_{k+1}-c)_+\eta_{k+1})\right|^2\dx\; \left(\frac{|A_{k+1}|}{R^n}\right)^{\frac 2 n}.
	\end{equation}
	We will consider the factors separately. For the first one we can apply the discrete Caccioppoli inequality \eqref{eq:CaccioppoliTruncated2} to get
	\begin{align*}
		&\dashint_{\Omega(2B)}\left|\nabla ((u_h-\lambda_{k+1}-c)_+\eta_{k+1})\right|^2\dx\\
		&\quad \lesssim \dashint_{\Omega(2B)} (u_h-\lambda_{k+1}-c)_+^2|\nabla \eta_{k+1}|^2\dx\\
		 &\qquad+\frac{1}{\abs{\Omega(2B)}}\left(\|G\|_{L^p(\Omega)}^2+\|f\|^2_{L^q(\Omega)}\right)\abs{\support\,\eta_{k+1}\cap\support\, (u_h-\lambda_{k+1}-c)_+}^{\frac{2\delta}{n}}\\
		&\qquad \quad \cdot\left(\norm{\eta_{k+1}}_{L^{2^\star}(\support\, (u-\lambda_{k+1}-c)_+)}^2+\norm{\nabla\eta_{k+1}}_{L^{2}(\support\, (u-\lambda_{k+1}-c)_+)}^2\right).
	\end{align*}
	Now, by applying Lemma \ref{lem:eta_k}, assertion \ref{itm:eta_k_d}, to the first summand, the bound $\max_\Omega |\nabla\eta_k|\leq 2^k R^{-1}$ from Lemma \ref{lem:eta_k}, assertion  \ref{itm:eta_k_a}, together with $\eta_k\leq 1$ and $\abs{A_k}\leq\abs{\supp\eta_{k+1}}\lesssim R^n$ to the second summand, we deduce that
	\begin{align*}
		&\dashint_{\Omega(2B)}\left|\nabla ((u_h-\lambda_{k+1}-c)_+\eta_{k+1})\right|^2\dx\\
		&\quad\lesssim \frac{2^{2k}}{R^2}\dashint_{\Omega(2B)}|(u_h-c-\lambda_k)_+|^2|\eta_k|^2\dx\\
		 &\qquad +R^{2\delta-2}\left(\|G\|_{L^p(\Omega)}^2+\|f\|_{L^q(\Omega)}^2\right)\left(\left(\frac{\abs{A_{k+1}}}{R^n}\right)^{1-\frac 2n+\frac{2\delta}{n}}+2^{2k}\frac{|A_{k+1}|}{R^n}\right).
	\end{align*}
	With the weak-type estimate \eqref{eq:weaktypeDiscrete} and because $\lambda_{k+1}>\lambda_k$, this becomes
	\begin{equation}\label{eq:Bounded3}
		\begin{aligned}
			&\dashint_{\Omega(2B)}\left|\nabla ((u_h-\lambda_{k+1}-c)_+\eta_{k+1})\right|^2\dx\\
			&\lesssim \frac{2^{2k}}{R^2}\dashint_{\Omega(2B)}|(u_h-c-\lambda_k)_+|^2|\eta_k|^2\dx
			+R^{2\delta-2}\left(\|F\|_{L^p(\Omega)}^2+\|f\|_{L^q(\Omega)}^2\right)\\
			&\quad \cdot\frac{2^{2k}}{\lambda_\infty^2}\left(\dashint_{\Omega(2B)} |\eta_k (u_h-\lambda_k-c)_+|^2\dx+\left(\dashint_{\Omega(2B)} |\eta_k (u_h-\lambda_k-c)_+|^2\dx\right)^{1-\frac 2n+\frac{2\delta}{n}}\left(\lambda_\infty^2\right)^{\frac{2}{n}-\frac{2\delta}{n}}\right)\\
			&= \frac{2^{4k}}{R^2}  \left(a_k+\left(a_k+a_k^{1-\frac{2}{n}+\frac{2\delta}{n}}\left(\lambda_\infty^2\right)^{\frac{2}{n}-\frac{2\delta}{n}}\right)\frac{R^{2\delta}\left(\|G\|_{L^p(\Omega)}^2+\|f\|_{L^q(\Omega)}^2\right)}{\lambda_\infty^2}\right).
		\end{aligned}
	\end{equation}
	For the second factor of inequality \eqref{eq:Bounded2} we use the weak-type estimate \eqref{eq:weaktypeDiscrete} to get
	\begin{equation}\label{eq:Bounded4}
		\begin{aligned}
		\left(\frac{|A_{k+1}|}{R^n}\right)^{\frac 2 n}&\leq\left(\frac{2^{2k}}{\lambda_\infty^2}\dashint_\Omega |\eta_k (u_h-\lambda_k-c)_+|^2\dx\right)^{\frac 2n}\\
		&=2^{\frac {4k}{n}}\left(\frac{a_k}{\lambda_\infty^2}\right)^{\frac2n}.
		\end{aligned}
	\end{equation}
	Substituting inequalities \eqref{eq:Bounded3} and \eqref{eq:Bounded4} into inequality \eqref{eq:Bounded2} yields
	\begin{equation*}
		a_{k+1}\lesssim  2^{4\left(1+\frac 1n\right)k}a_k\left(\left(\frac{a_k}{\lambda_\infty^2}\right)^{\frac 2n}+\left(\left(\frac{a_k}{\lambda_\infty^2}\right)^{\frac 2n}+\left(\frac{a_k}{\lambda_\infty^2}\right)^{\frac {2\delta}{n}}\right)\frac{R^{2\delta}\left(\|G\|_{L^p(\Omega)}^2+\|f\|_{L^q(\Omega)}^2\right)}{\lambda_\infty^2}\right).
	\end{equation*}
	If we now choose $\lambda_\infty^2\sim \max\left\{R^{2\delta}\left(\|F\|_{L^p(\Omega)}^2+\|f\|_{L^q(\Omega)}^2\right),\dashint_{\Omega(2B)}(u_h-c)_+^2\right\}\dx$, we obtain
	\begin{equation*}
		a_{k+1}\lesssim 2^{4k(1+\frac 1n)}a_k\left(\left(\frac{a_k}{\lambda_\infty^2}\right)^{\frac 2n}+\left(\frac{a_k}{\lambda_\infty^2}\right)^{\frac {2\delta}{n}}\right).
	\end{equation*}
	With the help of Corollary \ref{cor} we get $a_k\rightarrow 0$, as $k\rightarrow \infty$ after choosing $\gamma=\lambda_\infty^2$, $b=2^{4(1+\frac{1}{n})}$ and $\alpha=\frac{2\delta}{n}$, if $$\lambda_\infty^2\sim a_0\leq \dashint_{\Omega(2B)}(u_h-c)_+^2\dx.$$
Because $\lim_{k \rightarrow \infty} a_k =0$, passing to the limit $k \rightarrow \infty$ in the definition of $a_k$ we deduce that $(u_h-c)_+^2\leq\lambda_\infty^2$ on $\Omega'\left(B(x_0,R)\right)\subset \bigcap_k \supp\eta_k$.
By recalling our choice of $\lambda_\infty^2$ we thus have on $\Omega'\left(B(x_0,R)\right)\subset \bigcap_k \supp\eta_k$ the following bound:
	\begin{align*}
		 (u_h-c)_+^2&\leq\lambda_\infty^2\lesssim\max\left\{\left(\|G\|_{L^p(\Omega)}^2+\|f\|_{L^q(\Omega)}^2\right)R^{2\delta},\dashint_{\Omega(2B)}(u_h-c)_+^2\dx\right\}\\
		&\sim \left(\|G\|_{L^p(\Omega)}^2+\|f\|_{L^q(\Omega)}^2\right)R^{2\delta}+\dashint_{\Omega(2B)}(u_h-c)_+^2\dx,
	\end{align*}
	which proves the stated claim.
\end{proof}

Next we will prove an estimate showing that if a value of $u_h$ is small at a node, then it cannot be too large at neighbouring nodes. By enabling us to control not only the maximum of a function on a simplex, but also its minimum, this will help us to recover a property analogous to the one in inequality \eqref{eq:A_kCont} that will be a discrete version of Lemma \ref{lem:ContLemma}.

\begin{lemma}\label{lem:DiscreteSmallBallSmallFunction}
    Under the assumptions of Theorem \ref{thm:DiscreteHoelderInteriour}, let $v_h$ be a discrete subsolution to $-\divergence(A\nabla u)=f-\divergence  F$
    with $0\leq v_h\leq 1$. Then, there exist constants $\tau \in (0,1)$, $N>0$ and $C>0$ such that if $x_i$ and $x_j$ are nodes of the same simplex $T\in\mathcal{T}_h$
    with $\partial T\cap\partial\Omega=\emptyset$, then we have
    \begin{equation}
        v_h(x_j)\leq 1-\tau^N+\tau^N v_h(x_i)+C\left(\norm{F}_{L^p(\Omega)}+\norm{f}_{L^q(\Omega)}\right)h_T^{\delta}.
    \end{equation}
\end{lemma}

\begin{proof}
	We follow the ideas from \cite{AguCaf86}, Lemma 1.7. By Lemma \ref{lem:nonObtuseAcuteSequence}, we can connect the nodes $x_i$ and $x_j$ by a sequence of nodes $\{x_i=y_0,y_1,\dots,y_N=x_j\}$ belonging to the same simplex $T$ with
	\begin{equation}\label{eq:AcuteSequence3}
		-\int_\Omega A\nabla \psi_i\cdot \nabla \psi_{i+1}\dx\geq \tau \int_\Omega A\nabla\psi_i\cdot\nabla\psi_i\dx,\qquad i=0,\dots,N-1,
	\end{equation}
	where we denote by $\psi_i$ the basis function associated with the node $y_i$ in the sense of Definition \ref{def:LagrangeBasis} and $\tau \in (0,1)$ only depends on the shape-regularity parameter $\Gamma$ of the triangulation. Also note that all of the $y_i$ are nodes of the same simplex $T$, which means that $N\leq n$. We denote the remaining nodes of $\mathcal{T}_h$ by $y_i$ with $i\geq N+1$. We shall prove the claim by induction over the sequence of nodes connecting $x_i$ and $x_j$.
The base step (corresponding to $N=0$ when $x_j=x_i$) is clear. Suppose therefore that $N \geq 1$.

	Let $v_h$ be a discrete subsolution to $-\divergence(A\nabla u)=f-\divergence  F$, such
	that $0 \leq v_h \leq 1$. Fix an integer $k\in\{0,\dots,N-1\}$, test equation \eqref{eq:DiscreteInhomSubsol} against $\phi_h=\psi_{k+1}$ and write $v_h=\sum_l v_h(x_l)\psi_l$ to get
	
		\begin{equation*}
		\sum_l v_h(y_l) \langle \nabla \psi_l,\nabla\psi_{k+1}\rangle_A=\int_\Omega A \nabla v_h\cdot\nabla\psi_{k+1}\dx\leq\int_\Omega F\cdot\nabla\psi_{k+1}+f\psi_{k+1}\dx.
	\end{equation*}
	This leads to
	\begin{equation}\label{eq:Induction1}
		\begin{aligned}
			v_h(y_{k+1})&\leq-\sum_{l\neq k+1} \frac{\langle\nabla\psi_l,\nabla\psi_{k+1}\rangle_A}{\langle \nabla\psi_{k+1},\nabla\psi_{k+1}\rangle_A}v_h(y_l)\\
			&\quad +\int_\Omega\frac{ F\cdot\nabla\psi_{k+1}}{\langle \nabla\psi_{k+1},\nabla\psi_{k+1}\rangle_A}\dx+\int\frac{f\psi_{k+1}}{\langle\nabla\psi_{k+1},\nabla\psi_{k+1}\rangle_A} \dx\\
		&=:I+II+III.
		\end{aligned}
	\end{equation}
	First, we consider term $I$. Since $0\leq v_h\leq 1$, and because $\mathcal{T}_h$ is $A$-nonobtuse in the sense of inequality \eqref{eq:aNonObtuse}, we find that
	\begin{align*}
		I&=-\sum_{l\neq k+1}\frac{\langle\nabla\psi_l,\nabla\psi_{k+1}\rangle_A}{\langle\nabla\psi_{k+1},\nabla\psi_{k+1}\rangle_A}v_h(y_l) \\
		&=-\sum_{l\neq k,k+1}\frac{\langle\nabla\psi_l,\nabla\psi_{k+1}\rangle_A}{\langle\nabla\psi_{k+1},\nabla\psi_{k+1}\rangle_A}v_h(y_l)-\frac{\langle\nabla\psi_k,\nabla\psi_{k+1}\rangle_A}{\langle\nabla\psi_{k+1},\nabla\psi_{k+1}\rangle_A}v_h(y_k)\\
		&\leq \sum_{l\neq k+1}-\frac{\langle\nabla\psi_l,\nabla\psi_{k+1}\rangle_A}{\langle\nabla\psi_{k+1},\nabla\psi_{k+1}\rangle_A}+\frac{\langle\nabla\psi_k,\nabla\psi_{k+1}\rangle_A}{\langle\nabla\psi_{k+1},\nabla\psi_{k+1}\rangle_A}-\frac{\langle\nabla\psi_k,\nabla\psi_{k+1}\rangle_A}{\langle\nabla\psi_{k+1},\nabla\psi_{k+1}\rangle_A}v_h(y_k).
	\end{align*}
Since $0=\nabla 1 = \sum_l \nabla \psi_l$ and therefore $-\sum_{l\neq k+1}\frac{\skp{\nabla\psi_l}{\nabla \psi_{k+1}}_A}{\skp{\nabla\psi_{k+1}}{\nabla\psi_{k+1}}_A}=\frac{\skp{\nabla\psi_{k+1}}{\nabla \psi_{k+1}}_A}{\skp{\nabla\psi_{k+1}}{\nabla\psi_{k+1}}_A}=1$, we get	
	\begin{align*}
		 I&\leq1+\frac{\langle\nabla\psi_k,\nabla\psi_{k+1}\rangle_A}{\langle\nabla\psi_{k+1},\nabla\psi_{k+1}\rangle_A}-\frac{\langle\nabla\psi_k,\nabla\psi_{k+1}\rangle_A}{\langle\nabla\psi_{k+1},\nabla\psi_{k+1}\rangle_A}v_h(y_k)\\
		 &=1+\frac{\langle\nabla\psi_k,\nabla\psi_{k+1}\rangle_A}{\langle\nabla\psi_{k+1},\nabla\psi_{k+1}\rangle_A}(1-v_h(y_k)).
	\end{align*}
	Because the sequence $y_i$ satisfies inequality \eqref{eq:AcuteSequence3},  this becomes
	\begin{equation}\label{eq:Induction2}
		I \leq 1-\tau+\tau v_h(y_k).
	\end{equation}
	
	Next we consider term $II$ from inequality \eqref{eq:Induction1}. We know that $\supp \psi_{k+1}= \overline{\Omega(y_{k+1})}$ and $0\leq\psi_l\leq 1$. Hence, $\norm{\psi_{k+1}}_{L^r(\Omega)}\sim h_T^{\frac{n}{r}}$ as $y_{k+1}\in T$. By equation \eqref{eq:h_Tdef} we find that $\norm{\nabla\psi_{k+1}}_{L^r(\Omega)}\sim h_T^{\frac nr -1}$. By the uniform ellipticity of $A$ we have $\skp{\nabla\psi_{k+1}}{\nabla\psi_{k+1}}_A\gtrsim \norm{\nabla\psi_{k+1}}_{L^2(\Omega)}^2$. This gives
	\begin{equation}\label{eq:Induction3}
		II\lesssim \frac{\|F\|_{L^p(\Omega)}\|\nabla\psi_{k+1}\|_{L^{\left(1-\frac 1n +\frac{\delta}{n}\right)^{-1}}(\Omega)}}{\|\nabla\psi_{k+1}\|_{L^2(\Omega)}^2}\sim\|F\|_{L^p(\Omega)}\frac{h_T^{n-2+\delta}}{h_T^{n-2}}\sim\|F\|_{L^p(\Omega)}h_T^\delta.
	\end{equation}
	
	Analogously, we find for term $III$, using H\"older's inequality, that
	\begin{equation}\label{eq:Induction4}
		III\lesssim \frac{\|f\|_{L^q(\Omega)}\|\psi_{k+1}\|_{L^{\left(1-\frac2n+\frac{\delta}{n}\right)^{-1}}(\Omega)}}{\|\nabla\psi_{k+1}\|^2_{L^2(\Omega)}}\sim\|f\|_{L^q(\Omega)}\frac{h_T^{n-2+\delta}}{h_T^{n-2}}\sim\|f\|_{L^q(\Omega)}h_T^\delta.
	\end{equation}
	Inserting the inequalities \eqref{eq:Induction2}, \eqref{eq:Induction3} and \eqref{eq:Induction4} into inequality \eqref{eq:Induction1} gives
	\begin{equation}\label{eq:Induction5}
		v_h(y_{k+1})\leq 1-\tau + \tau v_h(y_k)+C\left(\norm{F}_{L^p(\Omega)}+\norm{f}_{L^q(\Omega)}\right)h_T^\delta.
	\end{equation}
	
	Now assume that
	\begin{equation*}
		v_h\left(y_k\right)\leq 1 -\tau^k+\tau^k v_h\left(x_i\right)+C\left(\norm{F}_{L^p(\Omega)}+\norm{f}_{L^q(\Omega)}\right) h_T^\delta \sum_{m=0}^{k-1} \tau^m.
	\end{equation*}
	We recall that $y_0:=x_i$ and calculate
	\begin{align*}
		v_h\left(y_{k+1}\right)&\leq 1-\tau+\tau v_h\left(y_k\right)+C\left(\norm{F}_{L^p(\Omega)}+\norm{f}_{L^q(\Omega)}\right)h_T^\delta\\
		&\leq 1 - \tau +\tau \left(1 - \tau^k + \tau^k v_h(y_0)+C\left(\norm{F}_{L^p(\Omega)}+\norm{f}_{L^q(\Omega)}\right)h_T^\delta\sum_{m=0}^{k-1} \tau^m\right)\\
		&\qquad + C\left(\|f\|_{L^q(\Omega)}+\|F\|_{L^p(\Omega)}\right) h_T^\delta\\
		&=1 - \tau^{k+1} + \tau^{k+1}v_h(y_0) +C\left(\norm{F}_{L^p(\Omega)}+\norm{f}_{L^q(\Omega)}\right)h_T^\delta\sum_{m=1}^{k} \tau^m\\
		&\qquad +C\left(\norm{F}_{L^p(\Omega)}+\norm{f}_{L^q(\Omega)}\right)h^\delta\\
		&=1 - \tau^{k+1} + \tau^{k+1}v_h(y_0) +C\left(\norm{F}_{L^p(\Omega)}+\norm{f}_{L^q(\Omega)}\right)h_T^\delta\sum_{m=0}^{k} \tau^m.
	\end{align*}	
	By induction, this proves by setting $k=N-1$ and recalling again that $y_0:=x_i$ and $y_N=x_j$  the bound
	\begin{equation*}
		v_h\left(x_j\right)\leq 1 -\tau^N+\tau^N v_h\left(x_i\right)+C\left(\norm{F}_{L^p(\Omega)}+\norm{f}_{L^q(\Omega)}\right) h_T^\delta \sum_{m=0}^{N-1} \tau^m.
	\end{equation*}
	Because $\sum_{m=0}^{N-1}\tau^m$ is bounded and depends only on $N\leq n$ and $\tau$, which depends only on the shape-regularity parameter, we deduce the assertion of the lemma.
\end{proof}

This leads to the following lemma, which is a discrete version of Lemma \ref{lem:ContLemma}. It is a generalisation of ideas from \cite{AguCaf86}, where the discussion was restricted to Laplace's equation.
\begin{lemma}\label{lem:SmallSup}
    Under the assumptions of Theorem \ref{thm:DiscreteHoelderInteriour}, let $B(x_0,R)$ be a ball with $R\geq h_T$ and $B(x_0,QR)\subset\Omega$ where $Q$ is the constant from inequality \eqref{eq:regularBall} and $x_0\in  T$ for some $T\in\mathcal{T}_h$. Furthermore, let $v_h$ be a discrete subsolution to $-\divergence(A\nabla u)=f-\divergence  F$ with $0\leq v_h\leq 1$ on $B(x_0,QR)$. Assume further, that
    \begin{equation}\label{eq:PoincareAssumption2}
        \Bigabs{\bigcup_{i\,:\,v_h(x_i)=0}P_i\cap \Omega(B(x_0,2R))}\geq \beta \abs{\Omega(B(x_0,2R))}
    \end{equation}
    for some $\beta>0$. Then, there exist constants $R_0$, $\theta \in (0,1)$ and $C>0$, such that
    \begin{equation}
        \sup_{\Omega'(B(x_0,R))} v_h\leq 1-\theta + C\left(\norm{G}_{L^p(\Omega)}+\norm{f}_{L^q(\Omega)}\right)R^\delta
    \end{equation}
    for any $R\in\left(h_T,R_0\right)$.
\end{lemma}
\begin{proof}
	We recall our notational convention $2B:=B(x_0,2R)$ and define, for any $R\geq h_T$,
	\begin{align*}
		\mu:=1-\frac{\tau^N}{4},\quad\;
		\lambda_k:=1-(1-\mu)^k,\quad\;
		v_k:=\frac{1}{1-\lambda_k}(v_h-\lambda_k)_+,\quad\;
		A_k:=\Omega(2B)\cap\{{\color{blue}v_k}>0\}.
	\end{align*}
	Again, the  operation $(...)_+$ is meant in a nodal sense and $N$ and $\tau$ are the constants from Lemma \ref{lem:DiscreteSmallBallSmallFunction}.
	
	As we have $v_h\leq 1$, we have $0\leq v_k\leq 1$ as well. Note that $v_k$ is a subsolution to $-\divergence(A\nabla v)=\frac{f}{(1-\mu)^k}-\frac{\divergence  G}{(1-\mu)^k}$ because of Corollary \ref{cor:PositivePartSubsol}, so the local $L^\infty$-estimates from Theorem \ref{thm:DiscreteBounded} hold. This means that we get
	\begin{equation}\label{eq:Linfty}\begin{aligned}
		\sup_{\Omega'(B)}{v_k}&\lesssim \bigg(\dashint_{\Omega(2B)}v_k^2\dx\bigg)^{\frac 12}+C\left(\frac{\|G\|_{L^p}+\|f\|_{L^q}}{(1-\mu)^k}\right) R^\delta\\
		&\lesssim \left(\frac{|A_k|}{R^n}\right)^{\frac 12}+C\left(\frac{\|G\|_{L^p}+\|f\|_{L^q}}{(1-\mu)^k}\right) R^\delta.
	\end{aligned}\end{equation}
	If we can find a $\bar{k}$ such that $|A_{\bar{k}}|\leq \beta R^n+C_\beta R^{n+2\delta}$ for a small $\beta$, we will have $v_{\bar{k}}\leq \frac 12 + \tilde{C}R^\delta$. Assume that $v_{k+1}(x_i)>0$ for some $x_i\in \Omega(2B)$.
Because the sequence $(\lambda_k)_{k \geq 0}$ is strictly monotonically increasing, if $v_{k+1}(x_i)>0$ for some $x_i\in \Omega(2B)$, as has been assumed, then
$v_h(x_i)- \lambda_k > v_h(x_i) - \lambda_{k+1}>0$, whereby $(v_h(x_i)- \lambda_k)_+ = v_h(x_i)- \lambda_k$. Hence, we have
\begin{align*}
	 0<v_{k+1}(x_i)&=\frac{v_h(x_i)-\lambda_{k+1}}{1-\lambda_{k+1}}=\frac{1-\lambda_{k}}{1-\lambda_{k+1}}\frac{v_h(x_i)-\lambda_k}{1-\lambda_k}+\frac{\lambda_k-\lambda_{k+1}}{1-\lambda_{k+1}}\\
	&=\frac{1-\lambda_{k}}{1-\lambda_{k+1}} v_k(x_i)+\frac{\lambda_k-\lambda_{k+1}}{1-\lambda_{k+1}}\\
	&=\frac{v_k(x_i)}{1-\mu}+\frac{(1-\mu)^{k+1}-(1-\mu)^{k}}{(1-\mu)^{k+1}}\\
	&=\frac{v_k(x_i)}{1-\mu}+\frac{1}{1-\mu}(1-\mu-1)\\
	&=\frac{v_k(x_i)-\mu}{1-\mu},
\end{align*}
and therefore
\begin{equation}\label{eq:nuLeqVk}
	\mu<v_k(x_i).
\end{equation}

We also define a sequence $R_k$, such that
	\begin{equation}\label{eq:R_k}
		\frac{C\left(\|G\|_{L^p}+\|f\|_{L^q}\right)}{(1-\mu)^k}\left( 2 Q R_k \right)^\delta=\frac{\tau^N}{4}
	\end{equation}
	and assume that $R_k\geq h_T$.
	By inequality \eqref{eq:radiiShapeRegular} and inclusion \eqref{eq:regularBall}, we know that
	\begin{equation}\label{eq:R_k2}
		2 QR\geq\text{diam}\left(\Omega\left(B(x_0,2R)\right)\right)\geq h_S
	\end{equation}	
	for any $S\in \mathcal{T}_h$ with $S\subset\Omega(B(x_0,2R))$, as long as $R\geq h_T$. By combining equality \eqref{eq:R_k} and inequality \eqref{eq:R_k2} we have that
	\begin{equation}\label{eq:R_k3}
		\frac{C\left(\|G\|_{L^p}+\|f\|_{L^q}\right)}{(1-\mu)^k}h_S^\delta\leq \frac{\tau^N}{4}
	\end{equation}
	for any $R\in[h_T, R_k]$ and any $S\subset\Omega(2B)$.

Assume that $R\leq R_k$ for some $k$. Again, note that $v_k$ is a subsolution to $-\divergence(A\nabla v)=\frac{f}{(1-\mu)^k}-\frac{\divergence  F}{(1-\mu)^k}$. This means, by Lemma \ref{lem:DiscreteSmallBallSmallFunction}, that for any node $x_j$ of the same simplex $S\subset\Omega(B(x_0,2 R))$ as $x_i$, we get
\begin{align*}
	1-\frac{\tau^N}{4}&=\mu<v_k(x_i)\leq 1-\tau^N+\tau^N v_k(x_j)+C\left(\|\tilde{F}\|_p+\|\tilde{f}\|_q\right)h_S^\delta\\
	&\leq 1-\tau^N+\tau^N v_k(x_j)+\frac{C\left(\|F\|_p+\|f\|_q\right)}{(1-\mu)^k}h_S^\delta.
\end{align*}
Now recall that $R\leq R_k$ and note the inequality \eqref{eq:R_k3} to get
\begin{equation}\label{eq:u_kgeq12}
	1-\frac{\tau^N}{4}{\color{blue}<} 1-\tau^N+\tau^N v_k(x_j)+\frac{\tau^N}{4},
\end{equation}
which implies that $v_k(x_j)>\frac 12$. Recall from inequality \eqref{eq:nuLeqVk} that we have $\frac{1}{2}\leq\mu<v_k(x_i)$. Thus, we get 
\begin{equation*}
	A_{k+1}\subset \left\{v_k>\frac 12\right\}\cap \Omega(B(2R)),
\end{equation*}
as $v_k(x_k)>\frac 12$ at all nodes of $A_{k+1}$, because if an $S\in\mathcal{T}_h$ satisfies $S\subset A_k$, there is a node $x_0$ of $S$ such that $v_k(x_0)>\lambda_{k+1}$, and this means that we have shown that $v_k(x_j)\geq\frac 12$ for any other node of that simplex $S$ by inequality \eqref{eq:u_kgeq12}. Furthermore, as we are using piecewise affine basis functions, we know that $\max_S v_h$ and $\min_S v_h$ are attained at a node on any simplex $S\in\mathcal{T}_h$. Therefore, we get

\begin{equation*}
	\left\{0<v_k<\frac 12\right\}\cap A_{k+1}=\emptyset.
\end{equation*}
However, as we have $\left\{0<v_k<\frac 12\right\}\cap \Omega(2B)\subset A_k$ and $A_{k+1}\subset A_k$, it follows that
\begin{equation}\label{eq:AkMeasure}
	|A_k|\geq\left|\left(\left\{0<v_k<\frac 12\right\}\cap \Omega(2B)\right)\cup A_{k+1}\right|=\left|\left\{0<v_k<\frac 12\right\}\cap\Omega(B(2R))\right|+|A_{k+1}|.
\end{equation}
Now note that if $v_h(x_i)=0$, we also have $v_k(x_i)=0$. This means that inequality \eqref{eq:PoincareAssumption} holds for all $v_k$ for the same $\gamma=\beta$ by assumption \eqref{eq:PoincareAssumption2}. This allows us to use the Poincar\'e-type inequality from Theorem \ref{thm:Poincare}. Note that $\Omega(2B)$ is an $\alpha$-John domain, and the angle $\alpha$ only depends on the shape-regularity constant $\Gamma$. Thus, the Poincar\'e constant of $\Omega(2B)$ only depends on $\Gamma$. For additional details, we refer to \cite{DRS10}. Hence,
\begin{align*}
	|A_{k+1}|&\leq 2\int_{\Omega(2B)}\min\left\{v_k,\frac 12\right\}\dx\\
	&\lesssim R \int_{\Omega(2B)}\left|\nabla\left(\min\left\{v_k,\frac 12\right\}\right)\right|\dx\\
	&\leq R\int_{\Omega(2B)\cap\{0<v_k<\frac 12\}} |\nabla v_k|\dx\\
	&\leq R\left(\int_{\Omega(2B)}|\nabla v_k|^2\dx\right)^{\frac{1}{2}}\left|\left\{0<v_k<\frac 12\right\}\cap \Omega(B(2R))\right|^{\frac 12}.
\end{align*}
Here, the $\min$ is meant in the pointwise rather than in the nodal sense. 
As we have $0\leq v_k\leq 1$, we can apply the discrete Caccioppoli-type inequality \eqref{eq:CaccioppoliTruncated2} to deduce that
\begin{align*}
	R^2\int_{\Omega(2B)} |\nabla v_k|^2\dx&\lesssim\int_{B(2Q^2R)}\dx |v_k|^2+\frac{C\left(\|G\|^2_{L^p}+\|f\|_{L^q}^2\right)}{(1-\mu)^k}R^{n+2\delta}\\
	&\lesssim R^n+C_kR^{n+2\delta}.
\end{align*}
This gives
\begin{equation*}
	|A_{k+1}|\lesssim \left(R^n+\frac{C\left(\|G\|_{L^p}^2+\|f\|^2_{L^q}\right)}{(1-\mu)^k} R^{n+2\delta}\right)^{\frac 1 2}\left|\left\{0<v_k<\frac 12\right\}\cap \Omega(2B)\right|^{\frac 12}.
\end{equation*}
Combining this with \eqref{eq:AkMeasure} yields
\begin{equation*}
	|A_{k+1}|\lesssim \left(R^n+\frac{C\left(\|G\|^2_{L^p}+\|f\|^2_{L^q}\right)}{(1-\mu)^k} R^{n+2\delta}\right)^{\frac 12}(|A_k|-|A_{k+1}|)^{\frac 12}.
\end{equation*}
Note that $|A_0|\lesssim R^n$. With the help of Lemma \ref{lem:iteration2} we deduce that, for any $R\leq R_k$,
\begin{align*}
	|A_k|&\leq\frac{\sqrt{\max_{i\leq k}\left(CR^n+\frac{C\left(\|G\|^2_{L^p}+\|f\|^2_{L^q}\right)}{(1-\mu)^i} R^{n+2\delta}\right)}\sqrt{|A_0|}}{\sqrt{k}}\\
	&\lesssim\frac{\sqrt{C}}{\sqrt{k}}R^n+\sqrt{\frac{C\left(\|G\|^2_{L^p}+\|f\|^2_{L^q}\right)}{k(1-\mu)^k}}R^{n+\delta},
\end{align*}
where we have used that that all norms on $\setR^2$ are equivalent, so in particular $\sqrt{x^2+y^2}\sim|x|+|y|$.
This means that for any $\beta>0$, there exists a $\bar{k}$ such that, for any $R\leq R_{\bar{k}}$,
\begin{equation*}
	|A_{\bar{k}}|\leq \beta R^n+C_\beta \left(\|G\|_{L^p}+\|f\|_{L^q}\right) R^{n+\delta}.
\end{equation*}
Thanks to inequality \eqref{eq:Linfty} this gives, for some $C>0$,
$$\sup_{\Omega'(B(R))}v_{\bar{k}}\leq \frac 12+C\left(\|G\|_{L^p}+\|f\|_{L^q}\right)R^{\delta}.$$
 Recall that $v_k=\frac{(v_h-\lambda_k)_+}{1-\lambda_k}$. Then, we get on $\Omega'(B(R))$ that
\begin{align*}
	v_h(x)&\leq(1-\lambda_{\bar{k}})v_{\bar{k}}(x)+\lambda_{\bar{k}}\\
	&\leq(1-\lambda_{\bar{k}})\left(\frac 12+C\left(\|G\|_{L^p}+\|f\|_{L^q}\right)R^{\delta}\right) +\lambda_{\bar{k}}\\
	&=\frac 12 (\lambda_{\bar{k}}+1)+(1-\lambda_{\bar{k}})C\left(\|G\|_{L^p}+\|f\|_{L^q}\right)R^{\delta}\\
	&=\frac 12\left(2-(1-\mu)^{\bar{k}}\right)+ (1-\lambda_{\bar{k}})C\left(\|G\|_{L^p}+\|f\|_{L^q}\right)R^{\delta}\\
	&\leq 1-\frac 12 (1-\mu)^{\bar{k}}+ C\left(\|G\|_{L^p}+\|f\|_{L^q}\right)R^{\delta}\\
	&=1-\frac 12 \left(\frac{\tau^N}{4}\right)^{\bar{k}}+ C\left(\|G\|_{L^p}+\|f\|_{L^q}\right)R^\delta.
\end{align*}
This proves the lemma with $\theta=\frac 12 \left(\frac{\tau}{4}\right)^{\bar{k}}$ and $R_0=R_{\bar{k}}$.
\end{proof}

Lemma \ref{lem:SmallSup} allows us to prove a bound on the oscillation of $u_h$, stated in the following lemma.
\begin{lemma}\label{lem:oscDecay}
    Under the assumptions of Theorem \ref{thm:DiscreteHoelderInteriour}, let $R_0$ be the constant from Lemma \ref{lem:DiscreteSmallBallSmallFunction}, let $x_0\in  T$ for some $T\in\mathcal{T}_h$ and $R\in(h_T,R_0)$, and suppose that $B(x_0,QR)\subset\Omega$. Then, there exist constants $\theta \in (0,1)$ and $C>0$ such that
    \begin{equation}\label{eq:DiscreteOscillations}
        \osc_{\Omega'(B(x_0,R))}u_h\leq (1-\theta)\osc_{B(x_0,2QR)}u_h+C\left(\|G\|_{L^p}+\|f\|_{L^q}\right)R^\delta.
    \end{equation}
\end{lemma}

\begin{proof}
	We begin by noting that $u_h$ is by construction continuous on $\overline\Omega$. As $\osc (cu_h+d)=|c|\osc u_h$, we first set $\tilde{u}_h=u_h-\frac 12\left(\max_{\Omega(2B)}u_h+\min_{\Omega(2B)}u_h\right)$. We then rescale to $\tilde{\tilde{u}}_h={\tilde{u}_h}/{\|\tilde{u}_h\|_{L^\infty(\Omega(2B))}}$ to get $-1\leq\tilde{\tilde{u}}_h\leq 1$ on $\Omega(2B)$. Then, $\tilde{\tilde{u}}_h$ is still an approximate solution to $-\divergence  (A\nabla u)=\tilde{f}-\divergence\tilde{F}$  with a rescaled right-hand side, i.e., for $\tilde{f}={f}/{\|\tilde{u}_h\|_{L^\infty(\Omega(2B))}}$ and $\tilde{F}={F}/{\|\tilde{u}_h\|_{L^\infty(\Omega(2B))}}$  and $$2=\osc_{\Omega(2B)}\tilde{\tilde{u}}_h\leq\osc_{B(x_0,2QR)}\tilde{\tilde{u}}_h.$$
	Note that
	\begin{equation*}
		\bigcup_{i\,:\,u_+(x_i)=0}\Omega(x_i)=\bigcup_{i\,:\,u_h(x_i)\leq 0}\Omega(x_i)
	\end{equation*}
	and
	\begin{equation*}
		\left(\bigcup_{i\,:\,u_h(x_i)\leq 0}\Omega(x_i)\right)\cup\left(\bigcup_{i\,:\,-u_h(x_i)\leq 0}\Omega(x_i)\right)=\Omega.
	\end{equation*}
	This means that the inequality \eqref{eq:PoincareAssumption2} is satisfied for at least one of the functions $\tilde{\tilde{u}}_+$ and $(-\tilde{\tilde{u}}_h)_+$. As $u_h$ is a finite element solution, both $\tilde{\tilde{u}}_+$ and $(-\tilde{\tilde{u}}_h)_+$ are subsolutions, with right-hand sides $\tilde{f}$, $\tilde{F}$ and $-\tilde{f}$, $-\tilde{F}$, respectively. Furthermore, $\tilde{F}$ and $-\tilde{F}$ both satisfy assumption $(\star)$ with the dominating function $\tilde{G}=G/\norm{\tilde{u}_h}_{L^\infty(\Omega(2B))}$. As $\osc \tilde{\tilde{u}}_h=\osc(-\tilde{\tilde{u}}_h)$, we are free to choose the one for which inequality \eqref{eq:PoincareAssumption2} holds.
	
	We can therefore apply Lemma \ref{lem:SmallSup} to either $\tilde{\tilde{u}}_h$ or $-\tilde{\tilde{u}}_h$ and get
	\begin{align*}
		\osc_{\Omega'(B(R))}\tilde{\tilde{u}}_h &\leq \sup_{B(R)}\tilde{\tilde{u}}_++1 \leq 2-\theta+D\left(\|\tilde{G}\|_{L^p}+\|\tilde{f}\|_{L^q}\right)R^\delta\\
		 &=\left(1-\frac{\theta}{2}\right)\osc_{B(x_0,2QR)}\tilde{\tilde{u}}_h+\frac{D\left(\|G\|_{L^p}+\|f\|_{L^q}\right)}{\|\tilde{u}_h\|_{L^\infty(\Omega(2B))}}R^\delta.
	\end{align*}
		Hence, after multiplying through by $\|\tilde{u}_h\|_{L^\infty(\Omega(2B))}$, using that $\osc u_h=\|\tilde{u}_h\|_{L^\infty(\Omega(2B))} \osc \tilde{\tilde{u}}_h$,
		and redefining $\theta\leadsto\frac 12\theta$ we have that
		
	\begin{equation*}
		\osc_{\Omega'(B(x_0,R))}u_h\leq (1-\theta)\osc_{B(x_0,2QR)} u_h +D\left(\|G\|_{L^p}+\|f\|_{L^q}\right)R^\delta,
	\end{equation*}
	which completes the proof.
	\end{proof}

We are now ready to prove the interior H\"older-regularity result from Theorem \ref{thm:DiscreteHoelderInteriour}.

\begin{proof}[Proof of Theorem \ref{thm:DiscreteHoelderInteriour}]
	First we will need to show that $u_h$ is uniformly bounded on balls $B(x_0,2QR)$ if $B(x_0,4\kappa^{-1}QR)\subset\Omega$. Inequality \eqref{eq:DiscreteBounded} from Theorem \ref{thm:DiscreteBounded} yields
	\begin{equation}\label{eq:uniformBound}
		 \|u_h\|^2_{L^\infty(2QB)}\leq\norm{u_h}_{L^\infty(\Omega'(2Q\kappa^{-1}B))}^2\lesssim\|u_h\|^2_{L^2(\Omega(4\kappa^{-1}QB)}+C\left(\|f\|^2_{L^p(\Omega)}+\|G\|_{L^q(\Omega)}^2\right).\\
	\end{equation}
	
Thanks to the homogeneous Dirichlet boundary condition satisfied by $u_h$, we can use Poincar\'e's inequality to deduce that
	\begin{equation*}
		\|u_h\|_{L^2(\Omega(4QB))}^2\leq\|u_h\|_{L^2(\Omega)}^2\leq\|\nabla u_h\|_{L^2(\Omega)}^2.
	\end{equation*}
Because of the uniform ellipticity of $A$ and testing equation \eqref{eq:DiscreteInhom} against $\phi_h=u_h$ we get
	\begin{equation*}
		\|\nabla u_h\|_{L^2(\Omega)}^2\lesssim\int_\Omega A\nabla u_h\cdot \nabla u_h\dx=\int_\Omega F\cdot \nabla u_h +f u_h\dx.
	\end{equation*}
	We can now use H\"older's inequality and Sobolev embedding. Also note that $p>n\geq 2$  and  $q > \frac n 2 \geq \frac{2n}{n+2}$; hence,
	\begin{align*}
		\norm{\nabla u_h}_{L^2(\Omega)}^2&\leq\|\nabla u_h\|_{L^2(\Omega)}\|F\|_{L^2(\Omega)}+\|f\|_{L^{\frac{2n}{n+2}}(\Omega)}\|u_h\|_{L^{2^\star}(\Omega)}\\
		&\lesssim \|\nabla u_h\|_{L^2(\Omega)}\|F\|_{L^p(\Omega)}+\|f\|_{L^q(\Omega)}\|\nabla u_h\|_{L^2(\Omega)}.
	\end{align*}	Together with inequality \eqref{eq:uniformBound} this shows that for any $B$ with $4QB\subset\Omega$, $\|u_h\|^2_{L^\infty(2QB)}$ is uniformly bounded.

	Next we will prove inequality \eqref{eq:oscRAlpha} for different values of $R$. If
$$R\in\left[h_T,\left(\frac 1{2CQ}\left(\frac{\tau^N}{4}\right)^{\bar{k}+1}\right)^{\frac{1}{\delta}}\right]=:[R_0,R_1]\quad \mbox{ (see the conditions of the Lemma \ref{lem:oscDecay} and eqn. \eqref{eq:radiiShapeRegular})},$$
we can use Lemma \ref{lem:oscDecay} to get
	\begin{equation}\label{eq:oscBall}
		\osc_{\Omega'(B(x_0,R))}u_h\leq (1-\theta)\osc_{B(x_0,2QR)} u_h +D\left(\|G\|_{L^p(\Omega)}+\|f\|_{L^q(\Omega)}\right)R^\delta.
	\end{equation}
	Now, the inclusion \eqref{eq:OmegaPrime2} guarantees that there is a $\kappa>0$, depending on shape-regularity, such that $B(\kappa R)\subset\Omega'(B(R))$. This gives, together with inequality \eqref{eq:oscBall}, that
	\begin{equation}
		\osc_{B(x_0,\kappa R)}u_h\leq(1-\theta)\osc_{B(x_0,2QR)} u_h +D\left(\|G\|_{L^p(\Omega)}+\|f\|_{L^q(\Omega)}\right)R^\delta,
	\end{equation}	
	which allows us to apply Lemma \ref{lem:C-alpha-iteration} with $\phi(r)=\osc_{B(2Qr)}u_h$, $\tau=\frac {\kappa} {2Q}$, $1-\theta=(2Q\kappa^{-1})^{-\alpha_1}$ and $\alpha_2=\delta$. If it turns out that $\delta\geq\alpha_1$, we can use a weaker norm for $f$ and $F$ in our Caccioppoli estimates that result in a smaller $\delta$. This leaves us with
	\begin{equation}\label{eq:HoelderInteriour1}
		\osc_{B(x_0,R)}u_h\leq \left(\frac{C}{R_1^{\delta}}+C\right)R^{\delta}.
	\end{equation}
	If $R>\left(\frac 1{2CQ}\left(\frac{\tau^N}{4}\right)^{\bar{k}+1}\right)^{\frac{1}{\delta}}$, we have $2CQ\left(\frac{4}{\tau^N}\right)^{\bar{k}+1} R^\delta\geq 1$ and get		
\begin{equation}\label{eq:HoelderInteriour2}
		\osc_{B(x_0,R)}u_h\leq 2 \norm{u_h}_{L^{\infty}(B(R))}\leq 4CQ \norm{u_h}_{L^\infty(B(x_0,R))}\left(\frac{4}{\tau^N}\right)^{\bar{k}+1}R^\delta.
	\end{equation}
		
	For $R\leq  h_T$, assume first that $B(x_0,R)\subset \Omega( T)$. We note that $u_h$ is piecewise affine on every $S\in\mathcal{T}_h$.
	Denote by $L_{u_h,S}$ the Lipschitz constant of $u_h$ on a simplex $S\in\mathcal{T}_h$. For a ball $B(x_0,R)$ we have that
	\begin{equation}\label{eq:HoelderInteriour3}
		\osc_{B(x_0,R)}u_h\leq\left(\max_{\substack{S\in\mathcal{T}_h\\S\cap B(x_0,R)\neq\emptyset}} L_{u_h,S}\right)R.
	\end{equation}
	On the other hand, the $L_{u_h,S}$ are given by
	\begin{equation}\label{eq:HoelderInteriour4}
		L_{u_h,S}=\frac{\osc_{B_{i,S}}u_h}{2R_{i,S}},
	\end{equation}
	where $B_{i,S}$ is the inscribed ball of $S$ and $R_{i,S}$ its radius. Denote by $y_0$ the centre of $B_{i,S}$. We then find using inequality \eqref{eq:HoelderInteriour1} that
	\begin{equation}\label{eq:HoelderInteriour5}
		\osc_{B_{i,S}}u_h\leq \osc_{B(y_0, h_S)}u_h\lesssim h_S^\delta.
	\end{equation}
	By shape-regularity, we have $R_{i,S}\gtrsim h_S$. Substituting this and inequality \eqref{eq:HoelderInteriour5} into \eqref{eq:HoelderInteriour4} yields
	\begin{equation}\label{eq:HoelderInteriour6}
		L_{u_h,S}\sim h_S^{\delta-1}.
	\end{equation}
	Because $B(x_0,R)\subset \Omega( T)$, we get $h_S\sim h_T$ for all $S\in\mathcal{T}_h$ with $S\cap B(x_0,R)\neq\emptyset$ from equation \eqref{eq:Intersection}. Inserting this, together with the relation \eqref{eq:HoelderInteriour6}, into inequality \eqref{eq:HoelderInteriour3} yields
	\begin{equation}\label{eq:HoelderInteriour7}
		\osc_{B(x_0,R)} u_h \lesssim h_T^{\delta-1}R\leq R^\delta.
	\end{equation}
	By equation \eqref{eq:regularPatch} there is a $\kappa>0$,
        depending only on the shape-regularity constant, such that if $B(x_0,R)\cap\left(\Omega\setminus \Omega( T)\right) \neq \emptyset$, we necessarily have $R\geq \kappa h_T$. However, we can then use inequality \eqref{eq:HoelderInteriour1} to deduce that
	\begin{equation}\label{eq:HoelderInteriour8}
		\osc_{B(x_0,R)}u_h\leq \osc_{B(x_0, h_T)}u_h\lesssim  h_T^\delta\leq \frac{1}{\kappa^\delta} R^\delta.
	\end{equation}
	Taking $C$ as the maximum constant from inequalities \eqref{eq:HoelderInteriour1}, \eqref{eq:HoelderInteriour2}, \eqref{eq:HoelderInteriour7} and \eqref{eq:HoelderInteriour8} finally gives
	\begin{equation*}
		\osc_{B(x_0,R)}u_h\leq c R^{\delta}
	\end{equation*}
	and proves the theorem.
\end{proof}

\section{$C^\alpha$-Regularity at the boundary}
So far we have always assumed that we are far away from the boundary $\partial\Omega$ of $\Omega$.
In order to prove uniform H\"older regularity up to the boundary we will use suitable truncations of $u_h$ to be able to use $u_h\eta_h^2$ as a test function in equation \eqref{eq:DiscreteInhomSubsol} even if $\eta_h$ is not necessarily compactly supported in $\Omega$. 
For convenience we will write $\Omega^c:=\setR^n\setminus\Omega$.
\begin{definition}\label{def:outerCone}
  A domain $\Omega\subset\setR^n$ satisfies a uniform outer cone
  condition if, for any $x\in\partial\Omega$, there is a cone $C$ with
  $C\cap\Omega=\emptyset$, with its tip at $x$ and angle greater than
  some $\alpha_0>0$. In particular, this means that the complement~$\Omega^c$ is
  fat in the sense that there is a $\mu_0>0$ such that, for any
  $x\in\partial\Omega$ and $R>0$, we have
  \begin{equation}\label{eq:OutsidOmegaCone}
    \abs{B(x,R)\cap\Omega^c}\geq \mu_0 \abs{B(x,R)}.
  \end{equation}
\end{definition}

We note that a Lipschitz polyhedral domain automatically satisfies a uniform cone condition.
This allows us to state the main theorem of the section.
\begin{theorem}\label{thm:HoelderBoundary}
	Let $\Omega\subset\setR^n$ be a  polyhedral domain (which thereby satisfies a uniform outer cone condition in the sense of Definition \ref{def:outerCone}). Furthermore, let $p,q$ be defined via $\frac 1p=\frac 1n-\frac{\delta}n$ and $\frac 1q=\frac 2n-\frac{\delta}n$, let $f\in L^q(\Omega)$ and let $F\in L^p(\Omega;\setR^n)$ satisfy assumption ($\star$) with dominating function $G\in L^p(\Omega;\setR^n)$, and let $A\in L^\infty (\Omega;\setR^{n\times n})$ be a uniformly elliptic matrix-valued function. Let $\mathcal{T}_h$ be an $A$-nonobtuse, shape-regular triangulation of the polyhedral domain $\Omega$ with respective finite element spaces $V_h$ and ${\color{blue}V_{h,0}}$. Let $u_h\in  V_{h,0}$ be a continuous piecewise affine finite element approximation to the solution of the equation $-\divergence(A\nabla u)=f-\divergence  F$. Furthermore, assume that $u_h\vert_{\partial\Omega}\in C^\beta(\partial\Omega)$, uniformly in $h$.
 Then, there is an $\alpha \in (0,1)$ such that
	\begin{equation*}
		u_h\in C^{\alpha}(\overline{\Omega})
	\end{equation*}
	and
	\begin{equation*}
		\abs{u_h}_{C^\alpha(\overline\Omega)}\lesssim \norm{G}_{L^p(\Omega)}+\norm{f}_{L^q(\Omega)}+D,
	\end{equation*}
	uniformly in $h$ where $D$ depends on $\abs{u_h\vert_{\partial\Omega}}_{C^\beta}$, $A$, $\delta$, $\alpha$ and the shape-regularity parameter $\Gamma$ of $\mathcal{T}_h$.
\end{theorem}

Again, we will break down the proof into several lemmas. Recall that we write $ \mathcal{B} (x_0,R)$ for the connected component of $B(x_0,R)\cap\Omega$ that contains $x_0$.

Henceforth we shall assume that $n \geq 3$ and write $2^\star:= 2n/(n-2)$. The proofs are easily adjusted in the case of $n=2$.

\begin{theorem}\label{Thm:extendedCac}
	   Let $\mathcal{T}_h$ be an $A$-nonobtuse, shape-regular triangulation of the polyhedral domain $\Omega$. Let $u_h$ be a nonnegative subsolution to $-\divergence(A \nabla u_h) = f-\divergence  F$ for $F\in L^p$ and $f\in L^q$. Furthermore, suppose that $F$ satisfies assumption ($\star$) from Definition \ref{def:AssumptionStar} with dominating function $G$. For any $\eta\in C^\infty(\setR^n)$ define $\eta_h:=\Pi_h\eta$. Suppose furthermore that $u_h=0$ on $\partial\Omega\cap\supp\eta_h$. Then, we have
    \begin{equation}\label{eq:CaccioppoliTruncated1Boundary}
        \begin{aligned}
             &\int_\Omega \abs{\nabla u_h}^2\abs{\eta_h}^2\dx \lesssim \int_\Omega u_h^2\abs{\nabla\eta_h}^2\dx\\
 						&\quad +\left(\norm{{G}}_{L^p(\Omega)}^2+\norm{f}_{L^q(\Omega)}^2\right)\left(\norm{\nabla \eta_h}_{L^2(\support\, u_h)}+\norm{\eta_h}_{L^{2^\star}(\support\, u_h)}^2\right)\abs{\supp\eta_h}^{\frac{2\delta}{n}}.
        \end{aligned}
    \end{equation}
	\end{theorem}
	\begin{proof}
		Note that $\Pih(u_h\Pih(\eta_h^2))\in V_{h,0}$ because we assumed that $u_h=0$ on $\partial\Omega\cap\supp\eta_h$. This means that we can test inequality \eqref{eq:DiscreteInhomSubsol} against $\phi_h=\Pih(u_h\Pih(\eta_h^2))$ and follow the steps of the proof of Theorem \ref{thm:CaccioppoliInhom}.
	\end{proof}
	
	Thus we arrive at the following $L^{\infty}$-norm bound that is valid near the boundary.
	\begin{theorem}\label{thm:boundaryBound}
		Under the assumptions of Theorem \ref{thm:HoelderBoundary}, suppose that $v_h$ is a discrete subsolution to $-\divergence  (A\nabla u)=f-\divergence  F$. Suppose furthermore that $v_h=0$ on $\partial\Omega\cap \Omega(B(x_0,2R))$ for some ball $B(x_0,R)$. Then, we have
		\begin{equation}\label{eq:LInftyBoundary}
			\sup_{\Omega'( \mathcal{B} (x_0,R))}(v_h-c)_+^2\lesssim R^{-n}\int_{\Omega( \mathcal{B} (x_0,2R))}(v_h-c)_+^2\dx+\left(\norm{f}_{L^q(\Omega)}^2+\norm{G}_{L^p(\Omega)}^2\right)R^{2\delta}.
		\end{equation}
	\end{theorem}
	\begin{proof}
		By Theorem \ref{thm:SubsolNodalMax}, $(v_h-c-\lambda_k)_+$ (in the notation of the proof of Theorem \ref{thm:DiscreteBounded}) is a non-negative subsolution. We also have $(v_h-c-\lambda_k)_+=0$ on $\partial\Omega\cap \Omega(B(x_0,2R))$ because $c\geq 0$, $\lambda_k>0$ and $v_h=0$ on $\partial\Omega\cap \Omega(B(x_0,2R))$. 
		Therefore, the result follows using inequality \eqref{eq:CaccioppoliTruncated1Boundary} and proceeding as in the proof of Theorem \ref{thm:DiscreteBounded}.
	\end{proof}

\begin{lemma}
	    Under the assumptions of Theorem \ref{thm:HoelderBoundary}, let $v_h\in V_h$ be a discrete subsolution to the equation $-\divergence(A\nabla u)=f-\divergence  F$ with $0\leq v_h\leq 1$.
	    Then, there exist constants $\tau \in (0,1)$, $N>0$ and $C>0$ such that if $x_i$ and $x_j$ are nodes of the same simplex $T\in\mathcal{T}_h$ with $v_h=0$ on $\overline{\Omega(T)}\cap\partial\Omega$ then we have
    \begin{equation}\label{eq:5.45}
        v_h(x_i)\leq 1-\tau^N+\tau^N v_h(x_j)+C\left(\norm{G}_{L^p(\Omega)}+\norm{f}_{L^q(\Omega)}\right)h_T^{\delta}.
    \end{equation}
\end{lemma}
\begin{proof}
	We would like to follow the steps in the proof of Lemma \ref{lem:DiscreteSmallBallSmallFunction}, which requires being
able to test equation \eqref{eq:DiscreteInhomSubsol} against Lagrange basis functions $\psi_i$. If $x_i\in \Omega$, we have $\psi_i\in V_{h,0}$ and inequality \eqref{eq:Induction1} holds. Let us therefore assume that $x_i\in\partial\Omega$. We then have that $v_h(x_i)=0$ because of the assumption $v_h=0$ on $\overline{\Omega(T)}\cap\partial\Omega$. This yields the following inequality:
	
	\begin{equation}\label{eq:TestAgainstBoundaryBasis}
		\begin{aligned}
			\int_\Omega A\nabla v_h\cdot \nabla \psi_i\dx&=\sum_j v_h(x_j)\int_\Omega A\nabla\psi_j\cdot\nabla\psi_i\dx\\
			&=\sum_{j\neq i} v_h(x_j)\int_\Omega A\nabla\psi_j \cdot \nabla\psi_i\dx \leq 0,
		\end{aligned}
	\end{equation}
	where we have used in the last step that the mesh is $A$-nonobtuse. Inequality \eqref{eq:Induction1} thereby simplifies to \begin{equation}\label{eq:Induction1Boundary}
		\begin{aligned}
			v_h(y_{k+1})&\leq-\sum_{l\neq k+1} \frac{\langle\nabla\psi_l,\nabla\psi_{k+1}\rangle_A}{\langle \nabla\psi_{k+1},\nabla\psi_{k+1}\rangle_A}v_h(y_l).
		\end{aligned}
	\end{equation}
	This means that we can proceed as in the proof of Lemma \ref{lem:DiscreteSmallBallSmallFunction} from here on.
\end{proof}

This leads to the following variant of Lemma \ref{lem:SmallSup}.
\begin{lemma}\label{lem:smallSupBoundary}
    Under the assumptions of Theorem \ref{thm:HoelderBoundary}, let $B(x_0,R)$ be a ball with $R\geq h_T$ where $x_0\in  T$ for some $T\in\mathcal{T}_h$. Furthermore, let $v_h$ be a discrete subsolution to $-\divergence(A\nabla u)=f-\divergence  F$ with $0\leq v_h\leq 1$ on $ \mathcal{B} (x_0,QR)$. Assume further that
    \begin{equation}\label{eq:PoincareAssumptionBoundary}
        \left|B(x_0,QR)\cap\Omega^c\right|\geq \beta \abs{\Omega( \mathcal{B} (x_0,QR))}
    \end{equation}
    for some $\beta>0$, where $Q$ is the constant from inclusion \eqref{eq:regularBall} and $v_h=0$ on $\partial\Omega\cap \Omega\left(\overline{\Omega(B(2R))}\right)$.

    Then, there exist constants $R_0$, $\theta \in (0,1)$ and $C>0$ such that
    \begin{equation}\label{eq:DiscreteSmallBallSmallFunctionBoundary}
        \sup_{\Omega'( \mathcal{B} (x_0,R))} v_h\leq 1-\theta + C\left(\norm{F}_p+\norm{f}_q\right)R^\delta
    \end{equation}
    for any $R\in\left(h_T,R_0\right)$.
\end{lemma}
\begin{proof}
	We can follow the proof of Lemma \ref{lem:SmallSup} step by step, because the condition $v_h=0$ on $\partial\Omega\cap \Omega\left(\overline{\Omega(B(2R))}\right)$ ensures that inequalities \eqref{eq:CaccioppoliTruncated1Boundary}, \eqref{eq:LInftyBoundary} and \eqref{eq:5.45} hold. Furthermore, assumption \eqref{eq:PoincareAssumptionBoundary}
	guarantees, together with $v_h=0$ on $\partial\Omega\cap \Omega\left(\overline{\Omega(B(2R))}\right)$, that Poincar\'e's inequality can be applied to $v_h$ on $\Omega( \mathcal{B} (x_0,QR)$.
\end{proof}

This now allows us to prove Theorem \ref{thm:HoelderBoundary}.

\begin{proof}[Proof of Theorem \ref{thm:HoelderBoundary}]
	If we have $B(x_0,4\kappa^{-1}QR')\subset \Omega$ for an $x_0\in T$ for some $T\in\mathcal{T}_h$ and $R'\geq  h_T$, Theorem \ref{thm:DiscreteHoelderInteriour} yields
	\begin{equation*}
		\osc_{B(x_0,R)}u_h\lesssim R^\alpha.
	\end{equation*}
	 On the other hand, if $B(x_0,4\kappa^{-1}QR')\cap\partial\Omega\neq\emptyset$ for every $R'\geq h_T$, we write $u_h=(u_h)_+-(-u_h)_+$
	 and consider $(u_h)_+$ and $(-u_h)_+$ separately; both are nonnegative subsolutions in $V_{h}$. We write $\xi=\min\{\beta,\delta\}$ (See assumptions of Theorem \ref{thm:HoelderBoundary}). Let us consider $(u_h)_+$. We can assume that $u_h(x_0)=0$ for some point $x_0$ of $\partial\Omega$ because we are only interested in oscillations and could consider $u_h-u_h(x_0)$. Recall that $u_h\vert_{\partial\Omega}\in C^\beta(\partial\Omega)\subset C^\xi(\partial\Omega)$ and write $\abs{u_h\vert_{\partial\Omega}}_{C^\xi}=D$. Then, we have  $((u_h)_+-D\,\mathrm{diam}(\Omega)^\xi)_+\in V_{h,0}$. By Theorem \ref{thm:boundaryBound}, this gives that $((u_h)_+-D\,\mathrm{diam}(\Omega)^\xi)_+$ is uniformly bounded, which means that $(u_h)_+$ is uniformly bounded by $\sup_{\Omega}\big(((u_h)_+-D\,\mathrm{diam}(\Omega)^\xi)_+\big)+\,D\mathrm{diam}(\Omega)^\xi$. 
	
	 If $R\geq  h_T$ and $B(x_0,4\kappa^{-1}QR)\cap\Omega^c\neq\emptyset$, we know that $B(y,\kappa^{-1}QR)\subset B(x_0,5\kappa^{-1}QR)$ for some $y\in\partial\Omega$ and together with inequality \eqref{eq:OutsidOmegaCone} this gives assumption \eqref{eq:PoincareAssumptionBoundary}. We can also assume that $u_h(y)=0$. We write
	 $$v_h=\frac{((u_h)_+-D R^{\xi})_+}{\norm{((u_h)_+-D R^{\xi})_+}_{L^\infty(\mathcal{B}(x_0,5\kappa^{-1}QR))}}$$
	 with $D=c 5 \kappa^{-1}QR$.
	 This means that
	$$\osc_{\mathcal{B}(x_0,5\kappa^{-1}QR)}v_h=\norm{v_h}_{L^\infty(\mathcal{B}(x_0,5\kappa^{-1}QR))}=1$$
	and
	$$v_h\vert_{\partial\Omega\cap B(x_0,5\kappa^{-1}QR)}=0.$$
	Note that $v_h$ is a discrete subsolution to
\[-\divergence(A\nabla v)=\frac{1}{\norm{((u_h)_+-D R^{\xi})_+}_{L^\infty(\mathcal{B}(x_0,5\kappa^{-1}QR))}}(-\divergence  G+f).\]
	Hence we  can apply Lemma \ref{lem:smallSupBoundary} for $R\in [h_T,\frac 15 \kappa R_0]$ to get
	$$\sup_{\mathcal{B}(x_0,\kappa R)}v_h\leq 1-\theta+ C\frac{\norm{G}_{L^p(\Omega)}+\norm{f}_{L^q(\Omega)}}{\norm{((u)_+-DR^{\delta})_+}_{L^\infty(\mathcal{B}(x_0,5\kappa^{-1}QR))}}R^\xi.$$
	This leads to	
	\begin{equation}\label{eq:oscBoundary}
		\begin{aligned}
			 \osc_{ \mathcal{B} (x_0,\kappa R)}(u_h)_+&=\sup_{\mathcal{B}(x_0,\kappa R)}(u_h)_+\leq\sup_{ \mathcal{B} (x_0,\kappa R)}((u_h)_+-DR^{\xi})_+ +DR^{\xi}\\
			&=\norm{((u_h)_+-DR^{\xi})_+}_{L^\infty( \mathcal{B} (x_0,5\kappa^{-1}QR))}\sup_{ \mathcal{B} (x_0,\kappa R)}v_h+DR^{\xi}\\
			&\leq (1-\theta)\norm{((u_h)_+-DR^{\xi})_+}_{L^\infty( \mathcal{B} (x_0,5\kappa^{-1}QR))} \\
			&\qquad+\left(C\|G\|_{L^p(\Omega)}+C\|f\|_{L^q(\Omega)}+D\right)R^{\xi}\\
			&\leq (1-\theta)\norm{(u_h)_+}_{L^\infty({\mathcal{B}(x_0,5\kappa ^{-1}QR)})}\\
			&\qquad+\left(C\|G\|_{L^p(\Omega)}+C\|f\|_{L^q(\Omega)}+D\right)R^{\xi}\\
			&\leq (1-\theta)\osc_{\mathcal{B}(x_0,5\kappa^{-1}QR)}(u_h)_+ +\left(C\|G\|_{L^p(\Omega)}+C\|f\|_{L^q(\Omega)}+D\right)R^{\xi}.
		\end{aligned}
	\end{equation}
	This means that we can apply Lemma \ref{lem:C-alpha-iteration} to deduce that
	\begin{equation*}
		\osc_{\mathcal{B}(x_0,R)}(u_h)_+\lesssim \left(\|G\|_{L^p(\Omega)}+\|f\|_{L^q(\Omega)}+D\right) R^\xi.
	\end{equation*}
	The proofs for $R\leq h_T$ and $R\geq \frac 15 \kappa R_0$ are completely analogous to the respective parts in the proof of Theorem \ref{thm:DiscreteHoelderInteriour}. We then repeat this argument for $(-u_h)_+$. By combining the resulting bound
on $\osc_{\mathcal{B}(x_0,R)}(-u_h)_+$ with the above bound on $\osc_{\mathcal{B}(x_0,R)}(u_h)_+$ we deduce that
\begin{equation}\label{eq:calBOsc}
	\osc_{\mathcal{B}(x_0,R)}(u_h)\leq C \left(\|G\|_{L^p(\Omega)}+\|f\|_{L^q(\Omega)}+D\right) R^\xi.
\end{equation}

To finish the proof, we finally have to look at cases where $B(x_0,R)$ has multiple connected components. 
On a single connected component we have 
$$\abs{u_h(x)-u_h(y)}\leq C \left(\|G\|_{L^p(\Omega)}+\|f\|_{L^q(\Omega)}+D\right) \abs{x-y}^\xi$$ 
by inequality \eqref{eq:calBOsc}. Now, let $x$ and $y$ be two points in $\Omega(B(x_0,R))$ that are in different connected components. We find points $a,b\in\partial\Omega$ with $\abs{a-b}\leq\abs{x-y}$, such that $\abs{x-a}\leq\abs{x-y}$ and $\abs{y-b}\leq\abs{x-y}$ and $a$ is in the same connected component of $\Omega(B(x_0,R))$ as $x$ and $b$ is in the same connected component as $y$ (for example by just connecting $x$ and $y$ by a line and taking $a$ and $b$ as the intersections between that line and $\partial\Omega$). Recall that $\left| u_h\vert_{\partial\Omega}\right|_{C^\xi(\partial\Omega)}=D$ and therefore, we have $\abs{u_h(a)-u_h(b)}\leq D\abs{a-b}^\xi$.
Therefore, we can write
\begin{equation}\label{eq:connectHoelder}
	\begin{aligned}
		\abs{u_h(x)-u_h(y)}&\leq\abs{u_h(x)-u_h(a)}+\abs{u_h(a)-u_h(b)}+\abs{u_h(b)-u_h(y)}\\
		&\leq C \left(\|G\|_{L^p(\Omega)}+\|f\|_{L^q(\Omega)}+D\right)\big(\abs{x-a}^\xi+\abs{a-b}^\xi+\abs{b-y}^\xi\big).
	\end{aligned}
\end{equation}
Note that $0<\xi<1$. From Jensen's inequality we deduce that $(a_1^\xi+a_2^\xi+a_3^\xi)\leq 3\left(\frac{1}{3}(a+b+c)\right)^\xi$. This and $\abs{x-a}+\abs{a-b}+\abs{b-y}\leq 3\abs{x-y}$ yield together with inequality \eqref{eq:connectHoelder} that
\begin{align*}
	\abs{u_h(x)-u_h(y)}&\leq C \left(\|G\|_{L^p(\Omega)}+\|f\|_{L^q(\Omega)}+D\right)3\big(\frac{1}{3}(\abs{x-a}+\abs{a-b}+\abs{b-y})\big)^\xi\\
	&\leq  C \left(\|G\|_{L^p(\Omega)}+\|f\|_{L^q(\Omega)}+D\right)3 \abs{x-y}^\xi.
\end{align*}
Finally, this proves that
\begin{equation*}
		\abs{u_h}_{C^\alpha(\overline{\Omega})}\lesssim \|G\|_{L^p(\Omega)}+\|f\|_{L^q(\Omega)}+D,
	\end{equation*}
	which completes the proof.
\end{proof}

We conclude with the application of the discrete De Giorgi theory to a nonlinear elliptic problem.

\begin{theorem}[Uniformly elliptic nonlinear equations]\label{thm:discreteNonLinear}
	Let $\mathcal{T}_h$ be a shape-regular, nonobtuse triangulation of the polyhedral domain $\Omega$ with associated finite element space $V_h$. Furthermore, let $F\in L^p(\Omega;\setR^n)$ satisfy assumption ($\star$) from Definition \ref{def:AssumptionStar} and let $f\in L^q(\Omega)$ with $q$ and $p$ defined as in Theorem \ref{thm:HoelderBoundary}. Let $a:\Omega\times\setR\times\setR^n\rightarrow\setR$ satisfy $0<c\leq a(\cdot,\cdot,\cdot)\leq C<\infty$. Furthermore, let $u_h\in V_{h,0}$ be a discrete solution to $-\divergence  (a(x,u_h,\nabla u_h)\nabla u_h)=f-\divergence  F$, i.e., a function $u_h\in V_{h,0}$, that satisfies
	\begin{equation}\label{eq:discreteSolutionNonlinearElliptic}
		\int_\Omega a(x,u_h,\nabla u_h) \nabla u_h\cdot \nabla \phi_h\dx=\int_\Omega f\phi_h\dx+\int_\Omega F\cdot\nabla \phi_h\dx
	\end{equation}
	for all $\phi_h \in V_{h,0}$. Then, there exist $\alpha \in (0,1)$ and $D>0$ depending only on $c$, $C$, $f$, $F$, $\Omega$ and the shape-regularity parameter $\Gamma$ of $\mathcal{T}_h$, such that
	\begin{equation*}
		\abs{u_h}_{C^\alpha(\overline{\Omega})} \leq D.
	\end{equation*}
\end{theorem}

\begin{proof}
	First note that any nonobtuse triangulation is $A$-nonobtuse for a scalar $A$. The same steps as in the proof to Theorem \ref{thm:SubsolNodalMax} (with $f_2=0$ and $F_2=0)$ show that
	\begin{equation}\label{eq:SubsolNonlinear}
		\int_\Omega a(x,u_h,\nabla u_h)\nabla (u_h-c)_+\cdot\nabla \phi_h\dx\leq \int_\Omega f\phi_h\dx+\int_\Omega G\cdot\nabla \phi_h\dx
	\end{equation}
	for any $\phi_h\in V_{h,0}$.
	Because $a(x,u_h,\nabla u_h) \mathbbm{1}_{n \times n}$ is a uniformly elliptic matrix for any $x\in\Omega$, $u_h(x)\in\setR$ and $\nabla u_h(x)\in\setR^n$, we can follow the proofs of Theorem \ref{thm:HoelderBoundary} and \ref{thm:DiscreteHoelderInteriour}.
\end{proof}

\section{Conclusions}
\label{sec5}

We have shown that under the condition of shape-regularity and $A$-nonobtuseness, continuous piecewise affine finite element approximations to the elliptic partial differential equation $-\divergence  (A\nabla u)=f-\divergence F$, with $\Omega$ a bounded open polyhedral domain in $\setR^n$, $A\in L^\infty(\Omega;\setR^{n\times n})$ a uniformly elliptic matrix-valued function,  $f\in L^{q}(\Omega)$, $F\in L^p(\Omega;\setR^n)$, with $\frac 1 p=\frac 1 n- \frac \delta n$ and $\frac 1q=\frac 2n-\frac \delta n$ for some $\delta>0$,
are uniformly H\"older-continuous up to the boundary $\partial \Omega$, of the domain $\Omega$ if $F$ satisfies condition ($\star$) from Definition \ref{def:AssumptionStar}. In our forthcoming papers, following on from Theorem \ref{thm:discreteNonLinear} above, we will derive uniform $L^\infty(\Omega)$ norm bounds on sequences of finite element approximations to $p$- and $\phi$-Laplacian systems where $A$ is given by $\abs{\nabla u}^{p-2}$ or $\phi'(\abs{\nabla u})\abs{\nabla u}^{-1}$, respectively, and therefore fails to be uniformly elliptic or bounded.

\section*{Funding}

This work was supported by the UK Engineering and Physical Sciences
Research Council under the grant EP/L015811/1. The research of
L.~Diening was supported by the DFG through the grant CRC 1283.
T.~Scharle was supported by the Clarendon Fund.

\bibliographystyle{amsalpha}
\bibliography{IMANUM-refs}



\end{document}